\documentclass[reqno,tbtags,12pt]{amsart}

\usepackage{amsmath,amssymb,amsthm,amsfonts,amscd,array,graphicx,color}
\usepackage[in]{fullpage}
\usepackage[OT2,T1]{fontenc}
\DeclareSymbolFont{cyrletters}{OT2}{wncyr}{m}{n}
\DeclareMathSymbol{\Lob}{\mathalpha}{cyrletters}{76}

\numberwithin{equation}{section}

\newcommand{\R}{\mathbb{R}}

\newcommand{\CH}{\mathcal{H}}

\newcommand{\CC}{\mathcal{C}}

\newcommand{\E}{\mathbb{E}}
\newcommand{\C}{\mathbb{C}}
\newcommand{\Z}{\mathbb{Z}}

\newcommand{\X}{\mathbb{X}}
\newcommand{\V}{\mathbb{V}}

\newcommand{\FX}{\mathfrak{X}}
\newcommand{\tFX}{\widetilde{\mathfrak{X}}}
\newcommand{\hFX}{\widehat{\mathfrak{X}}}
\newcommand{\hpi}{\hat{\pi}}

\newcommand{\CI}{\mathcal{I}}
\newcommand{\CF}{\mathcal{F}}
\newcommand{\hCF}{\widehat{\mathcal{F}}}

\newcommand{\CO}{\mathcal{O}}
\newcommand{\bS}{\mathbb{S}}

\newcommand{\A}{\mathcal{A}}

\newcommand{\CV}{\mathcal{V}}
\newcommand{\tCV}{\widetilde{\mathcal{V}}}
\newcommand{\tU}{\widetilde{U}}
\newcommand{\tV}{\widetilde{V}}
\newcommand{\tW}{\widetilde{W}}
\newcommand{\tR}{\widetilde{R}}
\newcommand{\tD}{\widetilde{D}}
\newcommand{\talpha}{\widetilde{\alpha}}
\newcommand{\hp}{\hat{p}}

\newcommand{\CG}{\mathcal{G}}

\newcommand{\uC}{\underline{C}}

\newcommand{\I}{\mathcal{I}}

\newcommand{\bx}{\mathbf{x}}
\newcommand{\bv}{\mathbf{v}}
\newcommand{\by}{\mathbf{y}}

\newcommand{\bo}{\mathbf{o}}
\newcommand{\bm}{\mathbf{m}}
\newcommand{\bn}{\mathbf{n}}

\newcommand{\bell}{\boldsymbol{\ell}}

\newcommand{\tz}{\tilde{z}}
\newcommand{\hz}{\hat{z}}
\newcommand{\tC}{\widetilde{C}}
\newcommand{\hC}{\widehat{C}}
\newcommand{\pr}{\mathop{\mathrm{pr}}\nolimits}
\newcommand{\tpr}{\mathop{\widetilde{\mathrm{pr}}}\nolimits}
\newcommand{\hpr}{\mathop{\widehat{\mathrm{pr}}}\nolimits}

\newcommand{\arcosh}{\mathop{\mathrm{arcosh}}\nolimits}
\newcommand{\Arcosh}{\mathop{\mathrm{Arcosh}}\nolimits}

\newcommand{\Arcsin}{\mathop{\mathrm{Arcsin}}\nolimits}
\newcommand{\Arccos}{\mathop{\mathrm{Arccos}}\nolimits}
\newcommand{\Log}{\mathop{\mathrm{Log}}\nolimits}
\newcommand{\Hom}{\mathop{\mathrm{Hom}}\nolimits}
\renewcommand{\Im}{\mathop{\mathrm{Im}}\nolimits}
\renewcommand{\Re}{\mathop{\mathrm{Re}}\nolimits}

\newcommand{\lk}{\mathop{\mathrm{lk}}\nolimits}
\newcommand{\Var}{\mathop{\mathrm{Var}}\nolimits}

\newtheorem{theorem}{Theorem}[section]
\newtheorem{propos}[theorem]{Proposition}
\newtheorem{cor}[theorem]{Corollary}
\newtheorem{lem}[theorem]{Lemma}

\theoremstyle{definition}

\newtheorem{example}[theorem]{Example}
\newtheorem{defin}[theorem]{Definition}
\newtheorem{remark}[theorem]{Remark}

\author{Alexander A. Gaifullin}

\thanks{This work is supported by the Russian Science Foundation under grant 14-50-00005}

\title[The Bellows conjecture in Lobachevsky spaces]{The analytic continuation of volume and the Bellows conjecture in Lobachevsky spaces}

\date{}

\address{Steklov Mathematical Institute of Russian Academy of Sciences}

\email{agaif@mi.ras.ru}


\begin{document}

\begin{abstract}
A \textit{flexible polyhedron} in an $n$-dimensional space~$\X^n$ of constant curvature is a polyhedron with rigid $(n-1)$-dimensional faces and hinges at $(n-2)$-dimensional faces. The \textit{Bellows conjecture} claims that, for $n\ge 3$, the volume of any flexible polyhedron is constant during the flexion. The Bellows conjecture in Euclidean spaces~$\E^n$ was proved by Sabitov for $n=3$ (1996) and by the author for $n\ge 4$ (2012). Counterexamples to the Bellows conjecture in open hemispheres~$\bS^n_+$ were constructed by Alexandrov for $n=3$ (1997) and by the author for $n\ge 4$ (2015). In this paper we prove the Bellows conjecture for bounded flexible polyhedra in odd-dimensional Lobachevsky spaces. The proof is based on the study of the analytic continuation of the volume of a simplex in the Lobachevsky space considered as a function of the hyperbolic cosines of its edge lengths.
\end{abstract}

\maketitle

\section{Introduction}

Let $\X^n$ be one of the three $n$-dimensional spaces of constant curvature, namely, the Euclidean space~$\E^n$, or the Lobachevsky space~$\Lambda^n$, or the sphere~$\bS^n$. We shall always normalize the metrics on the sphere~$\bS^n$ and on the Lobachevsky space~$\Lambda^n$ so that their (sectional) curvatures are equal to~$1$ and~$-1$, respectively. We consider an oriented connected closed  $(n-1)$-dimensional polyhedral surface~$P$ in~$\X^n$ with rigid $(n-1)$-faces and with hinges at $(n-2)$-faces. We consider continuous deformations~$P_t$ of~$P$ such that all~$P_t$ have the same combinatorial type, and every $(n-1)$-face of~$P_t$ remains congruent to itself during the deformation, while the dihedral angles at $(n-2)$-faces of~$P_t$ are allowed to vary continuously. Such deformations~$P_t$ are called \textit{flexions\/} of~$P$. A flexion~$P_t$ is called \textit{non-trivial} if it is not induced by an ambient rotation of~$\X^n$. A polyhedral surface~$P$ is called a \textit{flexible polyhedron\/} if it admits a non-trivial flexion. Notice that the surface~$P$ is not required to be embedded, though embedded flexible polyhedra are  of  a special interest. The two-dimensional case is trivial, since all generic polygons in~$\E^2$, $\Lambda^2$, and~$\bS^2$ with at least four sides are flexible, and all triangles are rigid. So further we assume that $n\ge 3$. 

The study of flexible polyhedra started with Bricard's classification of flexible octahedra in~$\E^3$, see~\cite{Bri97}. According to this classification, there are three families of flexible octahedra in~$\E^3$, and none of them contains an embedded octahedron. The first example of an embedded flexible polyhedron in~$\E^3$ was constructed by Connelly~\cite{Con77}, see also~\cite{Kui78}. 
The simplest of the presently known  embedded flexible polyhedra was constructed by Steffen in 1978. It has $9$ vertices. This polyhedron and its unfolding are shown in Fig.~\ref{fig_steffen}. A more detailed description of it can be found in~\cite{Con78a}. Kuiper~\cite[Sect.~2.7]{Kui78} noticed that the analogues of Bricard's flexible octahedra and of Connelly's flexible polyhedron exist both in~$\Lambda^3$ and in~$\bS^3$, see also~\cite{Sta06}. For a long time the problem on existence of flexible polyhedra in dimensions~$4$ and higher remained open. Examples of flexible self-intersecting polyhedra in dimension~$4$ were constructed by Walz (unpublished) and Stachel~\cite{Sta00}. In dimensions $5$ and higher, the first examples of flexible polyhedra were constructed by the author~\cite{Gai13} in all the three spaces~$\E^n$, $\Lambda^n$, and~$\bS^n$. In addition, in~\cite{Gai15} the author showed that in spheres~$\bS^n$ of all dimensions there exist embedded flexible polyhedra with the combinatorial type of the cross-polytope that are contained in the open hemispheres~$\bS^n_+$. (Recall that the cross-polytope is the  regular polytope dual to the cube.) However, the problem on the existence of embedded flexible polyhedra in Euclidean and Lobachevsky spaces of dimensions~$4$ and higher is still open.

\begin{figure}
\begin{center}
\includegraphics[scale=0.25]{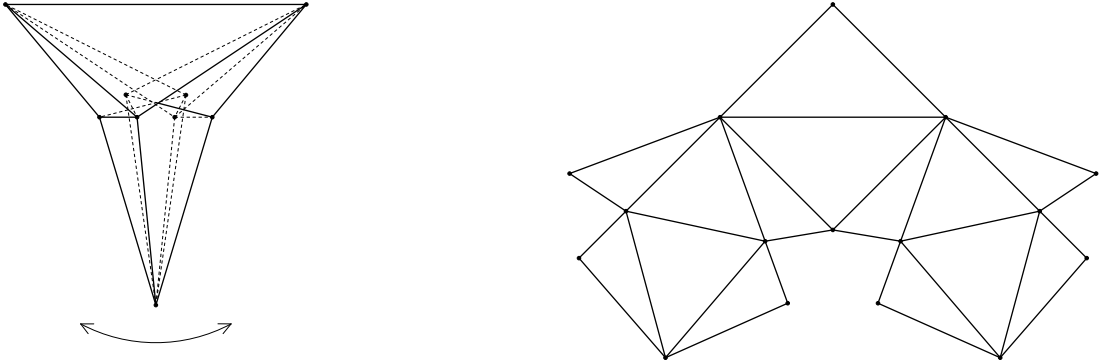}
\end{center}
\caption{Steffen's flexible polyhedron and its unfolding}\label{fig_steffen}
\end{figure}

In 1978 Connelly~\cite{Con78} conjectured that the volume of any flexible polyhedron in~$\E^3$ is constant during the flexion. This assertion is now called the \textit{Bellows conjecture\/}. Notice that this conjecture makes sense for all polyhedra without the requirement of embeddedness, since for self-intersecting polyhedra there is a natural concept of a \textit{generalized oriented volume\/}, see Section~\ref{section_def_res} for details. The proof of the Bellows conjecture by Sabitov~\cite{Sab96}--\cite{Sab98b} is one of the most amazing results in the theory of flexible polyhedra. Another proof was obtained in~\cite{CSW97}. Notice that flexible polygons change their areas during the flexion, so no analogue of the Bellows conjecture holds in dimension~$2$.

It is natural to consider the analogues of the Bellows conjecture for flexible polyhedra in all spaces of constant curvature of dimensions~$3$ and higher. If $\X^n=\E^n$ or $\Lambda^n$, then the \textit{Bellows conjecture in~$\X^n$} says that the volume of any flexible polyhedron in~$\X^n$ is constant during the flexion. The author~\cite{Gai11},~\cite{Gai12} proved the Bellows conjecture in all Euclidean spaces~$\E^n$, $n\ge 4$. 

The aim of the present paper is to prove that the Bellows conjecture holds for  all bounded flexible polyhedra in odd-dimensional Lobachevsky spaces~$\Lambda^{n}$, $n\ge 3$. 

\begin{theorem}\label{theorem_main}
The generalized oriented volume of any bounded flexible polyhedron in~$\Lambda^n,$ where $n$ is odd and $n\ge 3,$ is constant during the flexion.
\end{theorem}

To make this assertion rigorous, we need to give rigorous definitions of a flexible polyhedron and of a generalized oriented volume. This will be done in Section~\ref{section_def_res}. Theorem~\ref{theorem_main2} in that section is a more precise formulation of Theorem~\ref{theorem_main}.

In Lobachevsky spaces, alongside with bounded flexible polyhedra, one can consider flexible polyhedra of finite volume with some vertices on the absolute. It remains unknown whether the Bellows conjecture is true for such polyhedra. Also, it still remains unknown whether the Bellows conjecture is true in even-dimensional Lobachevsky spaces.

In the spherical case there exist the following trivial examples of flexible polyhedra with non-constant volumes. Consider any flexible polygon with non-constant area in the equatorial sphere~$\bS^2\subset\bS^3$, and take the suspension over it with the vertices at the poles of~$\bS^3$. It is easy to see that the volume of the obtained flexible polyhedron is non-constant. Iterating this construction, we can obtain examples of flexible polyhedra with non-constant volumes in spheres of all dimensions.  Such trivial examples become possible because they are too large and contain pairs of antipodal points of the sphere. Therefore, the Bellows conjecture for~$\bS^n$, $n\ge 3$, was usually formulated as follows: \textit{The volume of any flexible polyhedron contained in the open hemisphere $\bS^n_+\subset\bS^n$ is constant during the flexion.} However, this assertion is also false. The first example of a flexible polyhedron in~$\bS^3_+$ with non-constant volume was constructed by Alexandrov~\cite{Ale97}. Recently,  the author~\cite{Gai15} has constructed embedded flexible cross-polytopes with non-constant volumes in $\bS^n_+$ for all $n\ge 3$. Thus the Bellows conjecture in spheres is false. In~\cite{Gai15} the author has suggested the \textit{Modified bellows conjecture,} which claims that, for any flexible polyhedron~$P_t$ in~$\bS^n$, $n\ge 3$, one can replace some vertices of~$P_t$ with their antipodes so that the generalized oriented volume of the obtained flexible polyhedron will become constant during the flexion. Up to now, no counterexample to the Modified bellows conjecture is known.

A survey of some other results and problems on flexible polyhedra and their volumes can be found in~\cite{Sab11}.

The main ingredient of  our proof of Theorem~\ref{theorem_main} is  the study of the analytic continuation of the volume of a bounded simplex in~$\Lambda^n$ considered as a function of the hyperbolic cosines of its edge lengths. Much more studied is the problem of expressing the volumes of convex polytopes in~$\Lambda^n$ or~$\bS^n$ from their dihedral angles. The fundamental results in this area were obtained by Lobachevsky, Schl\"afli, Coxeter, and Milnor; a good survey can be found in~\cite[Ch. 7]{AVS88}. In particular, it is known that the functions expressing the volumes of simplices in~$\bS^n$ and~$\Lambda^n$ from their dihedral angles can be obtained from each other by analytic continuation, see~\cite{Cox35},~\cite{Aom77}. Besides, in the Lobachevsky case, the volume of a simplex can be expressed from its dihedral angles not only for bounded simplices but also for simplices with some vertices on the absolute. The functions expressing the volumes of simplices in~$\bS^n$ and~$\Lambda^n$ from their edge lengths have fewer good properties. Nevertheless, they are much more convenient in the study of flexible polyhedra, since  edge lengths are constant during flexions while  dihedral angles vary.

Let $\Delta^n\subset\Lambda^n$ be a bounded $n$-dimensional simplex with vertices $v_0,\ldots,v_n$, and let $\ell_{jk}$ be the length of an edge~$[v_{j}v_k]$. We put $c_{jk}=\cosh\ell_{jk}$, $c_{jj}=1$, and consider the matrix $C(\Delta^n)=(c_{jk})$ of size $(n+1)\times(n+1)$. If we identify the Lobachevsky space~$\Lambda^n$ with its standard vector model, namely, with the half of the hyperboloid $\langle\bx,\bx\rangle=1$, $x_0>0$ in the pseudo-Euclidean vector space~$\R^{1,n}$ with the inner product
$$
\langle\bx,\by\rangle=x_0y_0-x_1y_1-\cdots-x_ny_n,
$$
then the matrix $C(\Delta^n)$ will become the Gram matrix of the vertices of~$\Delta^n$.

We denote by~$\CG_{(n)}(\C)$ (respectively, by~$\CG_{(n)}(\R)$) the affine space of all complex (respectively, real) symmetric  matrices~$C$ of size $(n+1)\times(n+1)$ with units on the diagonal. Let $\CC_{\Lambda^n}\subset\CG_{(n)}(\R)$ be the subset consisting of all matrices~$C(\Delta^n)$ corresponding to non-degenerate bounded simplices $\Delta^n\subset\Lambda^n$. Obviously, $\CC_{\Lambda^n}$ is a domain in~$\CG_{(n)}(\R)$.

The simplex~$\Delta^n$ can be recovered from a matrix~$C\in\CC_{\Lambda^n}$ uniquely up to isometry, and the volume of~$\Delta^n$ is a real analytic function of~$C$, which will be denoted by~$V_{\Lambda^n}(C)$. Aomoto~\cite{Aom77} studied the analytic continuation of the function~$V_{\Lambda^n}(C)$ using Schl\"afli's formula for the differential of the volume, see Section~\ref{section_continue} for details. He proved that the analytic continuation of~$V_{\Lambda^n}(C)$ yields a multi-valued function on~$\CG_{(n)}(\C)$ that branches along~$X$ and has no other singularities, where $X$ is the hypersurface in~$\CG_{(n)}(\C)$ consisting of all matrices~$C$ such that at least one of the principal minors of~$C$ vanishes. 
In other words, the function $V_{\Lambda^n}(C)$ admits the analytic continuation along every path in $\CG_{(n)}(\C)\setminus X$. The multi-valued function on $\CG_{(n)}(\C)\setminus X$ obtained by the analytic continuation of~$V_{\Lambda^n}(C)$ will be denoted by~$\tV_{\Lambda^n}(C)$.  
 
Our main result concerning the analytic continuation of the volume of a simplex in an odd-dimensional Lobachevsky space is as follows. 

\begin{theorem}\label{theorem_key}
Suppose that $n$ is odd.   Let $\gamma\colon[0,1]\to \CG_{(n)}(\C)\setminus X$ be an arbitrary closed path such that\/ $\gamma(0)=\gamma(1)\in\CC_{\Lambda^n}$. Let $V_1(C)$ and $V_2(C)$ be two branches of the multi-valued function~$\tV_{\Lambda^n}(C)$ in a neighborhood of\/~$\gamma(0)$ such that\/ $V_2(C)$ is obtained from~$V_1(C)$ by the analytic continuation along~$\gamma$. Then 
\begin{equation*}
V_2(\gamma(0))-(-1)^{\lk(\gamma,\CH)}V_1(\gamma(0))\in i\R,
\end{equation*}
where $\CH\subset \CG_{(n)}(\C)$ is the hypersurface consisting of all degenerate matrices~$C,$ and\/ $\lk(\gamma,\CH)$ is the linking number of\/~$\gamma$ and~$\CH$. 
\end{theorem}

First, this result is the main technical step in our proof of Theorem~\ref{theorem_main}. Second, it seems to be interesting in itself.  There exists an analogue of this theorem for simplices in even-dimensional Lobachevsky spaces and in spheres. It is even easier to formulate, but, unfortunately, it does not yield any analogue of Theorem~\ref{theorem_main}. If $\X^n=\bS^n$, then we put $c_{jk}=\cos\ell_{jk}$ instead of $c_{jk}=\cosh\ell_{jk}$ so that $C(\Delta^n)$ is the Gram matrix of the vertices of~$\Delta^n$, where $\bS^n$ is realized as the unit sphere in the Euclidean vector space~$\R^{n+1}$. As in the Lobachevsky case, we denote by~$\CC_{\bS^n}$ the domain in~$\CG_{(n)}(\R)$ consisting of all matrices $C(\Delta^n)$ corresponding to non-degenerate simplices $\Delta^n\subset\bS^n$, and we denote by $\tV_{\bS^n}(C)$ the multi-valued analytic function obtained by the analytic continuation of the function~$V_{\bS^n}(C)$ computing the volume of a spherical simplex from the cosines of its edge lengths. Again, by a result of Aomoto~\cite{Aom77}, the ramification locus of~$\tV_{\bS^n}(C)$  is the same hypersurface~$X\subset \CG_{(n)}(\C)$. 

\begin{theorem}\label{theorem_key2}
Suppose that either\/ $\X^n=\Lambda^n$ and\/ $n$ is even or\/ $\X^n=\bS^n$ and\/ $n$ is arbitrary. Then all branches of the multi-valued function\/~$\tV_{\X^n}(C)$ are real in all points of\/~$\CC_{\X^n}$.
\end{theorem}

The main tool in our proofs of Theorems~\ref{theorem_key},~\ref{theorem_key2}, and~\ref{theorem_main} is Schl\"afli's formula for the differential of the volume of a convex polytope in~$\X^n$  that is deformed preserving its combinatorial type. Recall that this formula is as follows:
\begin{equation}\label{eq_Schlaefli}
KdV_n(P)=\frac{1}{n-1}\sum_{F\subset P,\,\dim F=n-2}V_{n-2}(F)\,d\alpha_F,
\end{equation}
where the sum is taken over all codimension~$2$ faces~$F$ of an $n$-dimensional polytope~$P$, $V_k$~denotes the $k$-dimensional volume, $\alpha_F$ is the dihedral angle of~$P$ at~$F$, and $K$ is the (sectional) curvature of~$\X^n$. This formula was proved by Schl\"afli~\cite{Sch58} for~$\bS^n$ and by Sforza~\cite{Sfo07} for~$\Lambda^n$, see also~\cite[Ch. 7, Sect.~2.2]{AVS88}. Actually, the convexity of~$P$ is unimportant, and formula~\eqref{eq_Schlaefli} holds true for all polyhedra, see Lemma~\ref{lem_Schlaefli}. Schl\"afli's formula is a very useful tool for studying the analytic properties of the volumes of simplices in~$\Lambda^n$ and~$\bS^n$. As it has already been mentioned, it was the main ingredient of Aomoto's construction of the analytic continuations of the functions~$V_{\X^n}(C)$. Let us also mention that Schl\"afli's formula was used by Rivin~\cite{Riv08} to give a simpler proof of  Milnor's conjecture on the continuity of the extension of the  function  computing the volume of a simplex from its dihedral angles, which was originally proved by Luo~\cite{Luo06}.

We shall use Schl\"afli's formula twice. First, we shall use it to prove Theorems~\ref{theorem_key} and~\ref{theorem_key2} by induction on~$n$. Second, we shall use it to deduce Theorem~\ref{theorem_main} from Theorem~\ref{theorem_key}. The scheme of our exposition will be roughly as follows. We shall show that the configuration space~$\Sigma$ of any flexible polyhedron~$P_t$ in~$\Lambda^{2m+1}$ is a connected component of a real affine variety, and  consider the analytic continuation of the generalized oriented volume of~$P_t$ to the complexification~$\Sigma_{\C}$ of this variety. On the one hand, Schl\"afli's formula for~$P_t$ will imply that the difference between any two branches of the obtained multi-valued analytic function~$\tCV$ is real. On the other hand, Theorem~\ref{theorem_key} will imply that this difference is purely imaginary on~$\Sigma$. Thus, we shall conclude that the function~$\tCV$ is single-valued. Finally, we shall show that the imaginary part of~$\tCV$ has logarithmic growth, and shall use Liouville's  theorem on entire functions to show that the function~$\tCV$  is constant.  

\begin{remark}\label{remark_even}
In the even-dimensional case both Schl\"afli's formula and Theorem~\ref{theorem_key2} will yield the same result, namely, they will yield that the difference between any two branches of the multi-valued function~$\tCV$ is real on~$\Sigma$. This is not sufficient to conclude that the function~$\tCV$ is single-valued. Hence we cannot apply Liouville's theorem on entire functions. Therefore our method for proving Theorem~\ref{theorem_main} does not work in even dimensions. Nevertheless, the question of whether the Bellows conjecture holds true in even-dimensional Lobachevsky spaces remains open, since no counterexample is known, too.
\end{remark}   

In the Euclidean space~$\E^n$ Schl\"afli's formula yields that the right-hand side of~\eqref{eq_Schlaefli} vanishes. Hence the differential of the volume does not enter this formula. This immediately implies that the \textit{total mean curvature}
\begin{equation}\label{eq_TMC}
\mathrm{TMC}(P)=\sum_{F\subset P,\,\dim F=n-2}V_{n-2}(F)(\pi-\alpha_F)
\end{equation}
 of a polyhedron~$P$ in~$\E^n$ is constant during flexions of~$P$, see~\cite{Ale85}, \cite{AlRi98}. Similarly,  for a flexible polyhedron~$P_t$ in~$\Lambda^n$, the linear combination $(n-1)V_n(P_t)-\mathrm{TMC}(P_t)$ is constant during the flexion. (See Remark~\ref{remark_TMC} for a rigorous definition of the total mean curvature of a non-embedded polyhedron.)

\begin{cor}\label{cor_TMC}
The total mean curvature of any bounded flexible polyhedron in~$\Lambda^n,$ where $n$ is odd and $n\ge 3,$ is constant during the flexion.
\end{cor}

In the three-dimensional case, there exist explicit formulae for the volumes of simplices in~$\bS^3$ and~$\Lambda^3$. The first formula of such kind was obtained by Sforza~\cite{Sfo07}. More convenient formulae in terms of Lobachevsky's function $\Lob(x)=-\int_0^x\log|2\sin t|\,dt$ or, equivalently, in terms of the dilogarithm $\mathrm{Li}_2(x)=-\int_0^x\log(1-t)\frac{dt}{t}$ were obtained in~\cite{ChKi99}, \cite{MuYa05}, \cite{MuUs05},~\cite{DeMe05}. It is possible that Theorems~\ref{theorem_key} and~\ref{theorem_key2} in the three-dimensional case can be deduced from these explicit formulae. However, this seems to be rather hard, since all these formulae are very cumbersome. In this paper, we do not use them.

This paper is organized as follows. Sections~\ref{section_fh}--\ref{section_proof_key} contain the proofs of Theorems~~\ref{theorem_key} and~\ref{theorem_key2}. The main technical lemma is Lemma~\ref{lem_tech}, which provides a sufficient condition that ensures that all branches of a multi-valued analytic function on a principal  Zariski open subset $\FX\subset\C^m$ are real on a connected component~$U$ of~$\FX\cap\R^m$. The key role in this  lemma is played by a special unipotent filtration on the ring of multi-valued analytic functions on a complex analytic manifold. The construction of this filtration is given in Section~\ref{section_fh}. Lemma~\ref{lem_tech} is formulated and proved in Section~\ref{section_tot_real}. In Section~\ref{section_continue} we  recall Aomoto's results on the analytic continuations of the functions~$V_{\X^n}(C)$.   In Sections~\ref{section_AB} and~\ref{section_proof_key} we apply Lemma~\ref{lem_tech} to prove  Theorems~~\ref{theorem_key} and~\ref{theorem_key2}. Section~\ref{section_def_res} contains all necessary definitions and notation concerning flexible polyhedra. Also in this section we give a more precise formulation of Theorem~\ref{theorem_main} (Theorem~\ref{theorem_main2}). In Sections~\ref{section_config}--\ref{section_proof_main_arbitrary}, we study the analytic continuation of the generalized oriented volume of a flexible polyhedron, and  prove Theorem~\ref{theorem_main2}. 

Throughout the paper, we shall deal with various classes of functions on complex analytic manifolds. Let us fix the terminology that we shall use in the sequel. We shall use the term `analytic function' for a \textit{multi-valued\/} analytic function, though in most cases we shall specify explicitly that the function under consideration is multi-valued. However, the term `holomorphic function' will always stand for a \textit{single-valued\/} analytic function, either in the point or on the whole manifold under consideration. The ring of holomorphic functions on a complex analytic manifold~$Y$ will be denoted by~$\CO(Y)$. Recall that a function that is holomorphic on the whole affine space~$\C^m$ is called an \textit{entire function.} Almost all complex analytic manifolds that are considered in this paper are smooth affine algebraic varieties over the field of complex numbers.  Following the standard terminology, the restriction of a polynomial in coordinates in~$\C^m$ to a (closed) affine submanifold $Y\subset \C^m$ will be called a \textit{regular function.} The ring of regular functions on~$Y$ will be denoted by~$\C[Y]$. Besides, by~$\C(Y)$ we shall denote the field of rational functions on~$Y$, i.\,e., the quotient field of the ring~$\C[Y]$.

The author is grateful to S.\,O.\,Gorchinsky, S.\,Yu.\,Nemirovsky, and I.\,Kh.\,Sabitov for useful discussions.

\section{Unipotent filtrations on spaces of analytic functions}\label{section_fh}

Let $\FX$ be a connected complex analytic manifold. Choose a base point $z^*\in\FX$. Let $\tFX$ be the universal covering of~$\FX$. It can be naturally identified with the set of homotopy classes of paths in $\FX$ starting at~$z^*$ with respect to the homotopy fixing the endpoints. For the base point in~$\tFX$ we take the point~$\tz^*$ corresponding to the constant path staying at~$z^*$. Let $p\colon\tFX\to \FX$ be the projection. The fundamental group 
$$
\pi=\pi_1(\FX,z^*)
$$ 
acts on~$\tFX$ from the left by deck transformations; we denote by~$T_{\gamma}$ the deck transformation corresponding to an element $\gamma\in\pi$.  For a path~$\alpha$ in~$\FX$ starting at~$z^*$, we denote by~$\tilde\alpha$ the lift of~$\alpha$ in~$\tFX$ starting at~$\tz^*$. Then the end of~$\tilde\alpha$  depends only on the homotopy class of~$\alpha$. If $\gamma$ is a loop, then the end of~$\tilde\gamma$ coincides with~$T_{\gamma}\tz^*$. With some abuse of notation, we denote the homotopy class of a loop by the same letter as the loop itself. Throughout this paper all paths and, in particular, all loops  are supposed to be piecewise smooth. For any path~$\alpha$, we denote by~$\alpha^{-1}$ the same path traversed in the opposite direction.

Let $\A^q(\tFX)$ be the space of holomorphic $q$-forms on~$\tFX$. In particular, $\A^0(\tFX)=\CO(\tFX)$ is the space of holomorphic functions on~$\tFX$. For $\gamma\in\pi$, the \textit{monodromy operator\/} $M_{\gamma}\colon\A^q(\tFX)\to\A^q(\tFX)$ is the pullback by~$T_{\gamma}$. The correspondence $\gamma\mapsto M_{\gamma}$ is an anti-homomorphism, i.\,e., $M_{\gamma_1\gamma_2}=M_{\gamma_2}M_{\gamma_1}$. The  \textit{variation operator\/}  $\Var_{\gamma}\colon\A^q(\tFX)\to\A^q(\tFX)$ is given by
\begin{equation*}
\Var_{\gamma}\theta=M_{\gamma}\theta-\theta.
\end{equation*}

Now, we define subspaces $\CF_k^q\subset\A^q(\tFX)$, $k\in\Z$, $k\ge -1$. First, we put $\CF_{-1}^q=\{0\}$. Second, we recursively take for $\CF_k^{q}$ the subspace of~$\A^q(\tFX)$ consisting of all~$\theta$ such that $\Var_{\gamma}\theta\in \CF_{k-1}^q$ for all $\gamma\in\pi$, $k=0,1,\ldots$. The filtration 
$$
\{0\}=\CF_{-1}^q\subset\CF_0^q\subset\CF_1^q\subset\CF_2^q\subset\cdots
$$
will be called the \textit{$\pi$-unipotent filtration.} Certainly, the union of all  $\CF_k^q$ generally does not coincide with~$\A^q(\tFX)$.

A $q$-form $\theta\in\A^q(\tFX)$ can be considered as a multi-valued analytic $q$-form on~$\FX$  with a chosen principal branch near~$z^*$. (The monodromy operators~$M_{\gamma}$ act by changing the principal branch without changing the multi-valued $q$-form itself.) In particular, the space $\A^q(\FX)$ of single-valued holomorphic $q$-forms on~$\FX$ can be naturally considered as a subspace of~$\A^q(\tFX)$. Namely, the embedding $\A^q(\FX)\hookrightarrow\A^q(\tFX)$ is given by the pullback~$p^*$ by the projection~$p$. It is easy to see that the subspace~$p^*\A^q(\FX)$ is exactly the subspace of all $\pi$-invariant holomorphic $q$-forms on~$\tFX$. Hence $p^*\A^q(\FX)=\CF^q_0$.

Though the above definitions have been given for an arbitrary~$q$, we shall mostly be interested in the cases $q=0$ and $q=1$, i.\,e., functions and $1$-forms.
Since the manifold~$\tFX$ is simply connected, any closed $1$-form $\theta\in\A^1(\tFX)$ is exact. 
Then the function 
\begin{equation*}
\I(\theta)(\tz)=\int_{\tz^*}^{\tz}\theta
\end{equation*}
is well defined and holomorphic on~$\tFX$.


\begin{lem}\label{lem_properties}
\textnormal{(1)} The spaces~$\CF_k^q$ are invariant under the monodromy operators~$M_{\gamma}$ and the variation operators~$\Var_{\gamma}$.

\textnormal{(2)} If\/ $\theta_1\in\CF_k^{q}$ and $\theta_2\in\CF_l^{r},$ then $\theta_1\wedge\theta_2\in\CF_{k+l}^{q+r}$. 

\textnormal{(3)} If\/ $\theta\in\CF_k^{q},$ then $d\theta\in\CF_{k}^{q+1}$.

\textnormal{(4)} If\/ $\theta\in\CF_k^{1}$ and $d\theta=0,$ then $\I(\theta)\in\CF_{k+1}^{0}$. 
\end{lem}

\begin{proof}
Assertion~(1)  follows immediately from the equality $$\Var_{\beta}M_{\gamma}=M_{\gamma}\Var_{\gamma\beta\gamma^{-1}}.$$

Let us prove assertion~(2) by induction on~$k+l$. If $k+l=-1$, then either $\theta_1=0$ or $\theta_2=0$, hence, the assertion is true. Assume that the assertion is proved for $k+l=m-1$, and prove it for $k+l=m$. For each $\gamma\in\pi$, we have
$$
\Var_{\gamma}(\theta_1\wedge\theta_2)=
(\Var_{\gamma}\theta_1)\wedge\theta_2+(M_{\gamma}\theta_1)\wedge(\Var_{\gamma}\theta_2).
$$
Since $\Var_{\gamma}\theta_1\in\CF_{k-1}^q$ and $\Var_{\gamma}\theta_2\in\CF_{l-1}^r$, the inductive assumption yields $\Var_{\gamma}(\theta_1\wedge\theta_2)\in\CF_{k+l-1}^{q+r}$. Hence $\theta_1\wedge\theta_2\in\CF^{q+r}_{k+l}$.

Assertion~(3) follows immediately by induction on~$k$. The basis of induction for $k=-1$ is obvious. The induction step is obtained from the  formula $\Var_{\gamma}d\theta=d\Var_{\gamma}\theta$. 

Now, let us prove assertion~(4) also by induction on~$k$. If $k=-1$, the assertion is trivial. Assume it for $k-1$ and prove it for~$k$, $k\ge 0$. Suppose that  $\theta\in\CF_k^{1}$ and $d\theta=0$. Then for each $\gamma\in\pi$, we have $\Var_{\gamma}\theta\in\CF_{k-1}^1$, $d\Var_{\gamma}\theta=0$,  and
$$
\Var_{\gamma}\I(\theta)=\I(\Var_{\gamma}\theta)+c
$$
for some constant $c\in\C$. By the inductive assumption, we have $\I(\Var_{\gamma}\theta)\in\CF_k^0$. Since $k\ge 0$, we have $c\in\CF_k^0$. Hence  $\Var_{\gamma}\I(\theta)\in\CF_k^0$. Therefore $\I(\theta)\in\CF_{k+1}^0$. 
\end{proof} 

For each connected covering~$\breve{\FX}$ of~$\FX$, the universal covering of~$\breve{\FX}$ coincides with~$\tFX$. Hence, replacing in the above construction the manifold~$\FX$ with the manifold~$\breve{\FX}$, we obtain a new filtration~$\breve{\CF}_k^q$ on~$\A^q(\tFX)$, which will be called the \textit{$\breve{\pi}$-unipotent filtration,} where $\breve{\pi}\subset\pi$ is the fundamental group of~$\breve{\FX}$. The following special case will be especially important for us. Let~$\hpi$ be the kernel of the natural epimorphism $\pi\to H_1(\FX;\Z_2)$, where $\Z_2=\Z/2\Z$ is the cyclic group of order~$2$, and let $\hFX$ be the covering of~$\FX$ corresponding to the subgroup~$\hpi$.  Then $\pi_1(\hFX,\hz^*)=\hpi$, where $\hz^*$ is the base point for~$\hFX$  chosen so that $\tz^*$ goes to~$\hz^*$, and then to~$z^*$  under the covering mappings $\tFX\to\hFX\to\FX$. We denote the projection $\tFX\to\hFX$ by $\hp$. Let 
$$
\{0\}=\hCF_{-1}^q\subset\hCF_0^q\subset\hCF_1^q\subset\hCF_2^q\subset\cdots
$$
be the $\hpi$-unipotent filtration on~$\A^q(\tFX)$. Then the space~$\hCF_k^q$ is the space consisting of all   $\theta\in\A^q(\tFX)$ such that $\Var_{\gamma}\theta\in\hCF_{k-1}^q$ for all $\gamma\in\hpi$.
Lemma~\ref{lem_properties} holds true for the filtration~$\hCF_k^q$. In particular,  assertion~(1) of this lemma says that the spaces~$\hCF_k^q$ are invariant under all monodromy operators~$M_{\gamma}$ such that $\gamma\in\hpi$. Since $\hpi$ is a normal subgroup of~$\pi$, the proof of  assertion~(1) of Lemma~\ref{lem_properties} in fact yields the following stronger assertion.

\begin{lem}
The spaces $\hCF_k^q$ are invariant under all monodromy operators $M_{\gamma}$, $\gamma\in\pi$.
\end{lem} 

\begin{example}
Put $\FX=\C\setminus\{0\}$, and consider the function $\Log z$ as an element of~$\CO(\tFX)$. The variation of this function along any loop is a constant of the form $2\pi i n$, $n\in\Z$. Hence $\Log z$ belongs to $\CF^0_1$ and, a fortiori, to~$\hCF^0_1$. A more interesting example is as follows. Put $\FX=\C\setminus\{-1,1\}$, and consider the function $\Arcsin z$ as an element of~$\CO(\tFX)$. The subgroup $\hpi\subset\pi$ has index~$4$ and consists the homotopy classes of all loops~$\gamma$ that have  even winding numbers around the points $1$ and $-1$. It is easy to see that the monodromy of~$\Arcsin z$ along any loop~$\gamma$ whose homotopy class belongs to~$\hpi$ yields a function of the form $\Arcsin z+2\pi n$, $n\in\Z$. Hence the function $\Arcsin z$ belongs to $\hCF^{0}_1$. However, it is easy to see that $\Arcsin z$  belongs to none of the spaces $\CF^0_k$.
\end{example}

Assume that the group  $H_1(\FX;\Z_2)$ is finite; let $2^s$ be the order of this group. We introduce the notation
$$
H=H_1(\FX;\Z_2),\qquad
H^*=H^1(\FX;\Z_2).
$$
There is a canonical identification~$H^*=\Hom(\pi,\Z_2)$. We denote by~$\langle\cdot,\cdot\rangle$ the non-degenerate pairing $H^*\otimes H\to\Z_2$.  The homology class in~$H$ represented by a loop~$\gamma$ or by a homotopy class~$\gamma\in\pi$ will be denoted by~$[\gamma]$.

It is useful to give the following characterization of the spaces~$\hCF_k^q$  in terms of  all monodromy operators~$M_{\gamma}$, $\gamma\in\pi$. For each $\rho\in H^*$, we define the corresponding \textit{twisted variation operator\/} by
\begin{equation*}
\Var_{\gamma}^{\rho}\theta=M_{\gamma}\theta-(-1)^{\rho(\gamma)}\theta.
\end{equation*}

\begin{lem}\label{lem_decompose}
Suppose that\/ $k,q\ge 0$. Then the space\/~$\hCF_k^q$ coincides with the space of all\/ $q$-forms\/ $\theta\in\A^q(\tFX)$ that admit a decomposition
\begin{equation}\label{eq_decompose}
\theta=\sum_{\rho\in H^*}\theta_{\rho}
\end{equation} 
such that\/ $\Var_{\gamma}^{\rho}\theta_{\rho}\in\hCF_{k-1}^q$ for all\/ $\gamma\in\pi$ and all\/ $\rho\in H^*$. If\/ $\theta\in\hCF_k^q,$ then we can take for\/~$\theta_{\rho}$ the\/ $1$-forms given by
\begin{equation}\label{eq_Phi_rho}
\theta_{\rho}=\frac{1}{2^s}\sum_{j=1}^{2^s}(-1)^{\rho(\gamma_j)}M_{\gamma_j}\theta,
\end{equation}
where\/ $\gamma_1=1,\gamma_2,\ldots,\gamma_{2^s}\in\pi$ are some representatives of all\/  $2^s$ cosets\/~$\pi/\hpi$.
\end{lem}

\begin{proof}
First, assume that $\theta$ admits a decomposition of the form~\eqref{eq_decompose}. For all $\gamma\in\hpi$ and all $\rho\in H^*$, we have $\rho(\gamma)=0$, hence, $\Var_{\gamma}\theta_{\rho}=\Var^{\rho}_{\gamma}\theta_{\rho}\in \hCF_{k-1}^q$. Therefore $\theta_{\rho}\in\hCF_{k}^q$ for all~$\rho$, thus, $\theta\in\hCF_{k}^q$.

Second, assume that $\theta\in\hCF_{k}^q$. Consider the $q$-forms~$\theta_{\rho}$ given by~\eqref{eq_Phi_rho}.
Since $[\gamma_1]=0$, $[\gamma_2],\ldots,[\gamma_{2^s}]$ are all different elements of~$H$, we easily see that
$\sum_{\rho\in H^*}\theta_{\rho}=\theta$. 
Let us prove that $\Var_{\gamma}^{\rho}\theta_{\rho}\in\hCF_{k-1}^q$ for all $\gamma\in \pi$. Let $\nu_{\gamma}$ be the involutive permutation of~$\{1,\ldots,2^s\}$ such that $[\gamma_{\nu_{\gamma}(j)}]=[\gamma_j]+[\gamma]$ in~$H$. We have,
\begin{multline*}
\Var_{\gamma}^{\rho}\theta_{\rho}=\frac{1}{2^s}\sum_{j=1}^{2^s}(-1)^{\rho(\gamma_j)}\Var_{\gamma}^{\rho}M_{\gamma_j}\theta=
\frac{1}{2^s}\sum_{j=1}^{2^s}\left((-1)^{\rho(\gamma_j)}M_{\gamma_j\gamma}\theta
-(-1)^{\rho(\gamma_j)+\rho(\gamma)}M_{\gamma_j}\theta\right)\\
{}=\frac{1}{2^s}\sum_{j=1}^{2^s}(-1)^{\rho(\gamma_j)}\left(M_{\gamma_{j}\gamma}\theta
-M_{\gamma_{\nu_{\gamma}(j)}}\theta\right)=\frac{1}{2^s}\sum_{j=1}^{2^s}(-1)^{\rho(\gamma_j)}M_{\gamma_{\nu_{\gamma}(j)}}\Var_{\gamma_{j}\gamma\gamma_{\nu_{\gamma}(j)}^{-1}}\theta.
\end{multline*}
All summands in the latter sum belong to $\hCF_{k-1}^q$, since $\gamma_{j}\gamma\gamma_{\nu_{\gamma}(j)}^{-1}\in\hpi$ for all~$j$. Hence $\Var_{\gamma}^{\rho}\theta_{\rho}\in\hCF_{k-1}^q$.
\end{proof}

Recall the notion of an \textit{iterated integral\/} introduced by Parshin~\cite{Par66} and Chen~\cite{Che67}, see also~\cite{Che73}.
Let  $\omega_1,\ldots,\omega_k$  be  $1$-forms on~$\FX$, and let $\alpha\colon[0,1]\to\FX$ be a path. Let $f_j(t)\,dt$ be the pullback of~$\omega_j$ by the mapping~$\alpha$, $j=1,\ldots,k$. 
By definition, the iterated integral of  $\omega_1,\ldots,\omega_k$ along $\alpha$ is given by
$$
\int_{\alpha}\omega_1\cdots\omega_k=\mathop{\int\!\cdots\!\int}\limits_{0\le t_1\le \cdots \le t_k\le 1}f_1(t_1)\cdots f_k(t_k)\,dt_1\cdots dt_k.
$$ 
A linear combination of iterated integrals
\begin{equation}\label{eq_lin_comb}
\sum_{n}c_n\int_{\alpha}\omega_{n,1}\cdots\omega_{n,k_n},
\end{equation}
where $\omega_{n,j}\in\A^1(\FX)$, $c_n\in\C$, can be considered as a function of~$\alpha$. This function is called \textit{homotopy invariant\/} if its value does not change under homotopies of~$\alpha$ preserving its endpoints. (Notice that the requirement of homotopy invariance is rather restrictive and does not hold automatically.) The maximum of the numbers~$k_n$ will be called the \textit{length\/} of the  linear combination of iterated integrals~\eqref{eq_lin_comb}. A homotopy invariant linear combination of iterated integrals determines a holomorphic function $\Phi\in\CO(\tFX)$ by 
\begin{equation}\label{eq_hiii}
\Phi(\talpha(1))=\sum_{n}c_n\int_{\alpha}\omega_{n,1}\cdots\omega_{n,k_n}
\end{equation}
for all paths~$\alpha$ in~$\FX$ starting at~$z^*$. It is a standard fact that the function~$\Phi$ has unipotent monodromy, see, for instance,~\cite{Aom77}. In our notation, this assertion is formulated as follows.

\begin{propos}\label{propos_iterate}
Any function $\Phi\in\CO(\tFX)$ given by a homotopy invariant linear combination of iterated integrals of length~$k$ of\/ $1$-forms~$\omega_{n,j}$ holomorphic on~$\FX$ belongs to~$\CF^0_k$.
\end{propos}
\begin{proof}
One of the basic properties of iterated integrals is the following formula for the iterated integral along the concatenation of two paths, see~\cite[Prop.~2]{Par66} and~\cite[Prop.~1.5.1]{Che73}:
$$
\int_{\beta\alpha}\omega_1\cdots\omega_k=\sum_{l=0}^k\left(\int_{\beta}\omega_1\cdots\omega_l\right)\left(\int_{\alpha}\omega_{l+1}\cdots\omega_k\right),
$$
where the integral of the empty cortege  of $1$-forms is supposed to be equal to~$1$, and it is assumed that the concatenation~$\beta\alpha$ is well defined, that is, the end of~$\beta$ coincides with the origin of~$\alpha$. This formula easily implies that the variation of the function~$\Phi$ given by~\eqref{eq_hiii} along a loop~$\gamma$ with endpoints at~$z^*$ is given by
\begin{equation*}
(\Var_{\gamma}\Phi)(\talpha(1))=\sum_{n}\sum_{l=1}^{k_n}c_n\left(\int_{\gamma}\omega_{n,1}\cdots\omega_{n,l}\right)\left(\int_{\alpha}\omega_{n,l+1}\cdots\omega_{n,k_n}\right),
\end{equation*}
which is a homotopy invariant linear combination of iterated integrals along~$\alpha$ of length not greater than~$k-1$. The proposition follows by induction on~$k$.
\end{proof}

Similarly, any function $\Phi$ given by a homotopy invariant linear combination of iterated integrals of length~$k$ of\/ $1$-forms~$\omega_{n,j}$ holomorphic on~$\hFX$ will belong to~$\hCF^0_k$.

\section{Totally real functions}\label{section_tot_real}

Throughout this paper, the words `a real (respectively, complex) affine algebraic variety' always mean `a subset of~$\R^m$ (respectively, $\C^m$) that can be given by polynomial equations'. In other words, we do not distinguish between affine varieties that coincide as subsets of~$\R^m$ (respectively, $\C^m$) but correspond to different ideals in the polynomial ring. Besides, we do not require that affine varieties  are irreducible unless this is specified explicitly. We say that an affine variety $X\subset\C^m$ is a \textit{hypersurface\/} if all irreducible components of~$X$ are $(m-1)$-dimensional. So we allow a hypersurface to be  reducible. 

For an affine algebraic variety~$X\subset\C^m$, we denote by~$X^{reg}$ the set of regular  points of~$X$, and we denote by~$X(\R)$ the set of real points of~$X$, i.\,e., the intersection $X\cap\R^m$. We also put $X^{reg}(\R)=X^{reg}\cap X(\R)$.

An irreducible affine algebraic variety $X\subset\C^m$ will be called \textit{essentially real\/} if $X$ is given by a system of polynomial equations with real coefficients, and $X^{reg}(\R)\ne\emptyset$. The latter condition is equivalent to the condition $\dim_{\R}X(\R)=\dim_{\C}X$. For instance, the circle $z_1^2+z_2^2=1$ in~$\C^2$ is essentially real, but the imaginary circle $z_1^2+z_2^2=-1$ is not essentially real though is defined over reals. A reducible affine variety $X\subset\C^m$ will be called \textit{essentially real\/} if all its irreducible components are essentially real.

Now, let $X\subset\C^m$ be an essentially real hypersurface, let $X_1,\ldots,X_s$ be its irreducible components, and let $U$ be a connected component of the set $\R^m\setminus X(\R)$. We consider the complex analytic manifold $\FX=\C^m\setminus X$, choose a base point~$z^*$ for~$\FX$ such that $z^*\in U$, and apply to~$\FX$ all constructions in the previous section.  

\begin{defin}
A function $\Phi\in\CO(\tFX)$ will be called \textit{totally real\/} on~$U$ if the values $\Phi(\tz)$ are real for all  $\tz\in p^{-1}(U)$. 
\end{defin}

Regarding $\Phi$ as a multi-valued function on~$\FX$, we may equivalently say that $\Phi$ is totally real on~$U$ if \textit{all\/} branches of~$\Phi$ on~$U$ are real-valued. For instance, the function $\Arcsin z$ is totally real on $(-1,1)$. On the other hand, the function $\Log z$ is not totally real on~$(0,+\infty)$, though its principal branch is real-valued on this set.

Similarly, a $1$-form $\theta\in\A^1(\tFX)$ is said to be \textit{totally real\/} on~$U$ if $\theta(\tz,v)\in\R$ for all $\tz\in\tFX$ and all $v\in T_{\tz}\tFX$ such that $p(\tz)\in U$ and $p_*(v)\in T_{p(\tz)}\R^m$. Obviously,  a $1$-form $\theta=\sum_{j=1}^m\theta_j\,dz_j$ is totally real on~$U$ if and only if all functions~$\theta_j$ are totally real on~$U$, where $z_1,\ldots,z_m$ are the standard coordinates in~$\C^m$.

The following properties of totally real functions are straightforward:

\begin{enumerate}
\item The set of functions totally real on~$U$ is a subring of~$\CO(\tFX)$.
\item If $\Phi$ is a function totally real on~$U$, then the functions $M_{\gamma}\Phi$ are also totally real on~$U$ for all $\gamma\in\pi$.
\item If $\Phi$ is a function totally real on~$U$, then the $1$-form $d\Phi$ is also totally real on~$U$.
\end{enumerate}

The converse of~(3) is generally not true. For instance, the $1$-form $\frac{dz}{z}$ is totally real on $(0,+\infty)$, but its integral $$\I\left(\frac{dz}{z}\right)=\Log z+c$$ is not totally real on $(0,+\infty)$. Nevertheless, there is a very important for us special case when the converse of~(3) is true. To formulate this result, we need to introduce some notation.

Suppose that we are given a decomposition $X=Y\cup Z\cup W$ such that $Y$ is an irreducible component of~$X$, $Z=Z_1\cup\cdots\cup Z_q$ and $W=W_1\cup\cdots\cup W_r$ are unions of irreducible components of~$X$, and all irreducible components $Y$, $Z_1,\ldots, Z_q,$  $W_1,\ldots,W_r$ are pairwise distinct. Let $f(z)=0$, $g_1(z)=0,\ldots,$ $g_q(z)=0$, $h_1(z)=0,\ldots,$ $h_r(z)=0$ be  irreducible polynomial equations with real coefficients giving the hypersurfaces $Y,$ $Z_1,\ldots,Z_q$, $W_1,\ldots,W_r$, respectively. We agree to choose the signs of the polynomials $h_j(z)$ so that $h_j(z)>0$ on~$U$. (The signs of the polynomials~$f(z)$ and~$g_j(z)$ are unimportant.) 
To each point $z\in \R^m\setminus W(\R)$, we assign the vector $\kappa(z)=(\kappa_1(z),\ldots,\kappa_r(z))\in\Z_2^{r}$ such that $\kappa_j(z)=0$ if $h_j(z)>0$ and $\kappa_j(z)=1$ if $h_j(z)<0$. Then $\kappa_j(z)$ is the modulo 2 linking number of the pair of points $\{z^*,z\}$ and the hypersurface~$W_j(\R)$ in~$\R^m$.

\begin{lem}\label{lem_tech}
Assume that the $4$-tuple $(Y,Z,W,U)$ satisfies the conditions: 
\begin{itemize}
\item[(A)] The boundary $\partial U$ of the domain~$U$ contains a point that belongs to~$Y^{reg}(\R)$ and does not belong to~$Z\cup W$. 
\item[(B)] The vectors $\kappa(z),$  where $z$ runs over all points of\/~$Y^{reg}(\R)$ that do not belong to $Z\cup W,$ generate the whole group~$\Z_2^{r}$.
\end{itemize}
Let $\theta\in\A^1(\tFX)$ be a $1$-form  such that
\begin{itemize}
\item[(i)] $\theta$ is totally real on~$U.$
\item[(ii)] $\theta$ is closed.
\item[(iii)] $\theta$ belongs to~$\hCF_k^1$ for some~$k.$
\item[(iv)]  In a neighborhood of each point $z_0\in Y\setminus (Z\cup W)$ any branch of\/~$\theta$ has the form
\begin{equation}\label{eq_omega}
\theta=\frac{\omega}{\sqrt{f(z)}}\,,
\end{equation}
where $\omega$ is a $1$-form holomorphic at~$z_0$.
\item[(v)] In a neighborhood of each point $z_0\in Z_j^{reg}$ that does not belong to the union of all other irreducible components of~$X$ any branch of\/~$\theta$ has the form
\begin{equation}\label{eq_omega2}
\theta=\frac{ic\,dg_j(z)}{g_j(z)}+\omega,
\end{equation}
where $c$ is a real constant, and\/ $\omega$ is a $1$-form holomorphic at~$z_0$.
\end{itemize}
Then the function 
$\I(\theta)$
is totally real on~$U$.
\end{lem}

\begin{remark}
In conditions~(iv) and~(v) in this lemma, the $1$-form~$\theta$ is regarded as a multi-valued analytic $1$-form on~$\FX$.
\end{remark}

\begin{proof}[Proof of Lemma~\ref{lem_tech}]
The proof is by induction on~$k$. The basis of induction for $k=-1$ is obvious. Assume that the assertion of the lemma is true for  $1$-forms in $\hCF^1_{k-1}$, and prove it for a $1$-form $\theta\in\hCF^1_k$. 

The group $H=H_1(\FX;\Z_2)$ is generated by the homology classes of circuits around the irreducible components of~$X$. Hence $H$ is finite. Therefore,
by Lemma~\ref{lem_decompose}, the $1$-form~$\theta$ is the sum of the $1$-forms $\theta_{\rho}$ given by~\eqref{eq_Phi_rho}, and $\Var_{\gamma}^{\rho}\theta_{\rho}\in\hCF^1_{k-1}$ for all $\gamma\in\pi$. Since $\theta$ satisfies conditions~(i), (ii), (iv), and~(v), we easily see that  the $1$-forms~$\theta_{\rho}$ also satisfy the same conditions~(i), (ii), (iv), and~(v). To prove that $\I(\theta)$ is totally real on~$U$ it is sufficient to prove that $\I(\theta_{\rho})$ are totally real on~$U$ for all~$\rho$. Let us fix a~$\rho$. If $\theta_{\rho}\in\hCF^1_{k-1}$, then $\I(\theta_{\rho})$ is totally real on~$U$ by the inductive assumption. So we assume that $\theta_{\rho}\notin\hCF^1_{k-1}$.

 Our goal is to show that the function~$\Psi=\I(\theta_{\rho})$ is real on~$p^{-1}(U)$. Since the $1$-form $d\Psi=\theta_{\rho}$ is totally real on~$U$, it is sufficient to prove that every connected component of~$p^{-1}(U)$ contains a point at which the value of~$\Psi$ is real. Hence it is sufficient to prove that the values $\Psi(T_{\gamma}\tz^*)$ are real for all $\gamma\in\pi$. Consider the mapping $\mu\colon\pi\to\R$ given by
\begin{equation*}
\mu(\gamma)=
\Im\Psi(T_{\gamma}\tz^*)=\Im\int_{\tz^*}^{T_{\gamma}\tz^*}\theta_{\rho}\,.
\end{equation*}
(The integral is independent of the path, since $\tFX$ is simply connected and $\theta_{\rho}$ is closed.)
Then we need to show that $\mu(\gamma)=0$ for all $\gamma\in\pi$. 

Let $z_0$ be a point in condition~(A). Since $z_0$ is a regular point of~$\partial U$, we can choose a non-self-intersecting path~$\alpha$ from~$z^*$ to~$z_0$ that is contained in~$U$, except for its endpoint~$z_0$. Let $\delta$ be a loop in~$\FX$ that starts at~$z^*$, goes along~$\alpha$ to a point~$z$ close to~$z_0$, goes once around~$Y$ in the positive direction near~$z_0$, and then returns to~$z^*$ along~$\alpha^{-1}$. Obviously, the homotopy class of~$\delta$ is independent of~$z$  and of the chosen circuit around~$Y$ provided that they are close enough to~$z_0$. So the homotopy class $\delta\in\pi$ is well defined. (Nevertheless, the homotopy class $\delta$ may depend on the path~$\alpha$, which is supposed to be fixed.)

Since any branch of~$\theta_{\rho}$  near~$z_0$ has  form~\eqref{eq_omega}, we obtain that  $M_{\delta}\theta_{\rho}=-\theta_{\rho}$. If $\rho(\delta)=0$, then  we would have $\Var_{\delta}^{\rho}\theta_{\rho}=-2\theta_{\rho}$, which would yield a contradiction, since $\Var_{\delta}^{\rho}\theta_{\rho}\in\hCF^1_{k-1}$ and $\theta_{\rho}\notin\hCF_{k-1}^1$. Therefore $\rho(\delta)=1$.

 \begin{lem}\label{slem_mu_delta}
 $\mu(\delta)=0$. 
 \end{lem}

\begin{proof}
Since any branch of~$\theta_{\rho}$  near~$z_0$ has  form~\eqref{eq_omega}, we have
$$
\int_{\tz^*}^{T_{\delta}\tz^*}\theta_{\rho}=\int_{\tilde\delta}\theta_{\rho}=2\int_{\tilde\alpha}\theta_{\rho},
$$
which is real, since $\theta_{\rho}$ is totally real on~$U$. Therefore $\mu(\delta)=0$.
\end{proof}

\begin{lem}\label{slem_mu_ah}
For all\/ $\gamma_1,\gamma_2\in\pi,$ we have
\begin{equation*}
\mu(\gamma_1\gamma_2)=\mu(\gamma_1)+(-1)^{\rho(\gamma_1)}\mu(\gamma_2).
\end{equation*}
\end{lem}

\begin{proof}
We have,
\begin{multline*}
\mu(\gamma_1\gamma_2)=\Im\int_{\tz^*}^{T_{\gamma_1\gamma_2}\tz^*}\theta_{\rho}=\mu(\gamma_1)+\Im\int_{T_{\gamma_1}\tz^*}^{T_{\gamma_1\gamma_2}\tz^*}\theta_{\rho}=\mu(\gamma_1)+\Im\int_{\tz^*}^{T_{\gamma_2}\tz^*}M_{\gamma_1}\theta_{\rho}\\
{}=\mu(\gamma_1)+(-1)^{\rho(\gamma_1)}\mu(\gamma_2)+\Im\int_{\tz^*}^{T_{\gamma_2}\tz^*}\Var_{\gamma_1}^{\rho}\theta_{\rho}.
\end{multline*}
Properties~(i), (ii), (iv), and~(v) for the $1$-form~$\theta_{\rho}$ immediately imply the same properties for
the $1$-form~$\Var_{\gamma_1}^{\rho}\theta_{\rho}$. Besides, $\Var_{\gamma_1}^{\rho}\theta_{\rho}\in\hCF_{k-1}^1$. Hence, by the inductive assumption, we conclude that the function~$\I(\Var_{\gamma_1}^{\rho}\theta_{\rho})$ is totally real on~$U$. Therefore, the value 
$$\I(\Var_{\gamma_1}^{\rho}\theta_{\rho})(T_{\gamma_2}\tz^*)=\int_{\tz^*}^{T_{\gamma_2}\tz^*}\Var_{\gamma_1}^{\rho}\theta_{\rho}$$ is real, which implies the lemma.
\end{proof}

\begin{lem}\label{lem_image_mu}
The image of the mapping~$\mu$ is a finitely generated subgroup of\/~$\R$.
\end{lem}

\begin{proof}
Let $A$ be the image of~$\mu$.  We need to show that, for any $a_1,a_2\in A$, both $a_1+a_2$ and $a_1-a_2$  belong to~$A$. If $a_1,a_2\in A$, then $a_1=\mu(\gamma_1)$, $a_2=\mu(\gamma_2)$ for some $\gamma_1,\gamma_2\in\pi$. Hence the number $\mu(\gamma_1\gamma_2)=a_1+(-1)^{\rho(\gamma_1)}a_2$ belongs to~$A$. On the other hand, we have constructed an element $\delta\in \pi$ such that $\mu(\delta)=0$ and $\rho(\delta)=1$. Hence the number $\mu(\gamma_1\delta\gamma_2)=a_1-(-1)^{\rho(\gamma_1)}a_2$ also belongs to~$A$.  Therefore $A$ is a subgroup of~$\R$. But the group~$\pi$ is finitely generated, since it is the fundamental group of the complement of a complex affine variety. Hence the group~$A$ is also finitely generated.
\end{proof}

Assume that $A\ne 0$.  Any non-trivial finitely generated subgroup of~$\R$ is a free Abelian group, hence, has an epimorphism onto~$\Z_2$. Choose an arbitrary epimorphism $\varpi\colon A\to\Z_2$, and consider the mapping
$$
\nu\colon \pi\xrightarrow{\mu}A\xrightarrow{\varpi}\Z_2.
$$
This mapping is surjective, since both $\mu$ and~$\varpi$ are surjective. Lemma~\ref{slem_mu_ah} implies that $\nu$ is a homomorphism, i.\,e., $\nu\in H^*=\Hom(\pi,\Z_2)$. 

Let $\eta,\zeta_1,\ldots,\zeta_q,\xi_1,\ldots,\xi_r\in H=H_1(\FX;\Z_2)$ be the the homology classes of small circuits around the irreducible components~$Y,Z_1,\ldots,Z_q,W_1,\ldots,W_r$ of~$X$, respectively. It is well known that these elements form a basis of~$H$. For a point $z\in \R^m\setminus W(\R)$, we put $\xi(z)=\sum_{j=1}^r\kappa_j(z)\xi_j$. Also, for any point $z\in \R^m\setminus X(\R)$, we denote by~$\sigma(z)$, $\tau_1(z),\ldots,\tau_q(z)$  the modulo~$2$ linking numbers of~$\{z^*,z\}$ with $Y(\R)$, $Z_1(\R),\ldots,Z_q(\R)$, respectively, and we put 
$$\eta(z)=\sigma(z)\eta,\qquad \zeta(z)=\sum_{j=1}^q\tau_j(z)\zeta_j,\qquad 
\chi(z)=\eta(z)+\zeta(z)+\xi(z).$$

\begin{lem}\label{slem_nu_zeta}
We have $\langle\nu,\zeta_j\rangle=0$ for $j=1,\ldots,q$.
\end{lem}

\begin{proof}
Let $z_0$ be a regular point of~$Z_j$ that does not belong to the union of all other irreducible components of~$X$. Consider a loop~$\gamma$ in~$\FX$  that goes from $z^*$ to a point $z\in \FX$ close to~$z_0$ along some path~$\alpha$,  travels  along a small circuit~$\beta$ around~$Z_j$ in the positive direction, and then returns to~$z^*$ along~$\alpha^{-1}$. Then $[\gamma]=\zeta_j$. By condition~(v) for~$\theta_{\rho}$, any branch of~$\theta_{\rho}$ is meromorphic at~$z_0$. Hence $M_{\gamma}\theta_{\rho}=\theta_{\rho}$. Therefore,
\begin{equation}\label{eq_2pic}
\Psi(T_{\gamma}\tilde z^*)=\int_{\tilde\gamma}\theta_{\rho}=\int_{\tilde\beta}\theta_{\rho}=-2\pi c,
\end{equation}
where $\tilde\beta$ is the lift of~$\beta$ starting at the end of~$\tilde\alpha$, and $c$ is the real constant in~\eqref{eq_omega2} for the branch of~$\theta_{\rho}$ realized on~$\tilde\beta$.
Consequently $\mu(\gamma)=0$, hence, $\langle\nu,\zeta_j\rangle=0$.
\end{proof}

\begin{lem}\label{slem_nu_chi}
Suppose that $z\in\R^m\setminus X(\R)$ and $\langle\rho,\chi(z)\rangle=1$. Then $\langle\nu,\chi(z)\rangle=0$.
\end{lem}

\begin{proof}
Choose an arbitrary path $\beta$ in~$\R^m$ with endpoints $z^*$ and~$z$ that intersects~$X(\R)$  transversely in finitely many regular points of it. Near every intersection point of $\beta$ and~$X(\R)$, we replace a small segment of~$\beta$ with a small positive half-circuit around~$X$ in~$\C^m$, see Fig.~\ref{fig_beta}. The obtained path from~$z^*$ to~$z$  in $\FX$ will be denoted by~$\beta_1$. Let $\gamma$ be the loop in~ $\FX$ with endpoints at~$z^*$ obtained by the concatenation of~$\beta_1$  and~$\bar\beta^{-1}_1$, where  $\bar\beta_1^{-1}$ is the path obtained from~$\beta_1$ by the coordinatewise conjugation traversed in the opposite direction. It is easy to see that $[\gamma]=\chi(z)$, since the incomes to~$[\gamma]$ of intersection points of $\beta$ with the components $Y(\R)$, $Z_1(\R),\ldots,Z_q(\R)$, $W_1(\R),\ldots,W_r(\R)$  of~$X(\R)$ are equal to $\eta,\zeta_1,\ldots,\zeta_q,\xi_1,\ldots,\xi_r$, respectively.
Obviously, the path~$\bar\gamma$ coordinatewise conjugate to~$\gamma$ coincides with~$\gamma^{-1}$.  Since the $1$-form~$\theta_{\rho}$ is real on~$p^{-1}(U)$, we see that the integrals of~$\theta_{\rho}$ along the lifts~$\tilde\gamma$ and $\Tilde{\bar\gamma}$ of~$\gamma$ and~$\bar\gamma$, respectively, starting at~$\tz^*$ are conjugate to each other, i.\,e.,
\begin{equation}\label{eq_conjugate}
\int_{\tz^*}^{T_{\gamma^{-1}}\tz^*}\theta_{\rho}=\int_{\tz^*}^{T_{\bar\gamma}\tz^*}\theta_{\rho}=\overline{\int_{\tz^*}^{T_{\gamma}\tz^*}\theta_{\rho}\,}. 
\end{equation}
Hence $\mu(\gamma^{-1})=-\mu(\gamma)$. On the other hand, by Lemma~\ref{slem_mu_ah}, we obtain that $$\mu(\gamma^{-1})=-(-1)^{\rho(\gamma)}\mu(\gamma)=-(-1)^{\langle\rho,\chi(z)\rangle}\mu(\gamma)=\mu(\gamma).$$ Therefore $\mu(\gamma)=0$. Thus $\langle\nu,\chi(z)\rangle=0$.
\end{proof}

\begin{figure}
\begin{center}
\unitlength=1.3cm

\begin{picture}(9,3.75)

\thinlines

\put(0,0){\line(1,0){7}}
\put(2,3){\line(1,0){7}}
\put(0,0){\line(2,3){2}}
\put(7,0){\line(2,3){2}}

\put(2,1){\circle*{.08}}
\put(7.5,2){\circle*{.08}}
\put(1.7,1){$z^*$}
\put(7.58,1.95){$z$}
\put(3,2.1){$\beta_1$}
\put(3.27,1.45){$\bar\beta_1^{-1}$}

\thicklines

\qbezier(2,1)(3.5,2.5)(5,1.5)
\qbezier(7.5,2)(6.2,.7)(5,1.5)

{\color{white}%
\put(2.4,1.36){\circle*{.424}}
\put(3,1.72){\circle*{.424}}
\put(5,1.5){\circle*{.424}}
\put(6.83,1.48){\circle*{.424}}}

\put(2.4,1.36){\circle{.424}}
\put(3,1.72){\circle{.424}}
\put(5,1.5){\circle{.424}}
\put(6.83,1.48){\circle{.424}}

\thinlines

\qbezier(2,1)(3.5,2.5)(5,1.5)
\qbezier(7.5,2)(6.2,.7)(5,1.5)

\thicklines

\put(2.4,1.36){%
\begin{picture}(0,0)

\qbezier(-.1,.2)(-.165,.17)(-.23,.18)
\qbezier(-.1,.2)(-.13,.135)(-.12,.07)

{\color{white}%
\put(.209,0.055){\circle*{.045}}
\put(.209,-0.055){\circle*{.045}}
\put(.152,-0.152){\circle*{.045}}
\put(.055,-0.209){\circle*{.045}}
\put(-.055,-0.209){\circle*{.045}}
\put(-.138,-0.164){\circle*{.04}}
\put(-.140,-0.166){\circle*{.04}}
\put(-.136,-0.162){\circle*{.04}}
}
\qbezier(.1,-.2)(.165,-.17)(.23,-.18)
\qbezier(.1,-.2)(.13,-.135)(.12,-.07)

\end{picture}%
}

\put(3,1.72){%
\begin{picture}(0,0)

\qbezier(-.1,.2)(-.165,.17)(-.23,.18)
\qbezier(-.1,.2)(-.13,.135)(-.12,.07)

{\color{white}%
\put(.214,0){\circle*{.045}}
\put(.171,-0.127){\circle*{.045}}
\put(.055,-0.209){\circle*{.045}}
\put(-.055,-0.209){\circle*{.045}}
\put(-.152,-0.152){\circle*{.045}}
}
\qbezier(.1,-.2)(.165,-.17)(.23,-.18)
\qbezier(.1,-.2)(.13,-.135)(.12,-.07)

\end{picture}%
}

\put(5,1.5){%
\begin{picture}(0,0)

\qbezier(.2,.1)(.17,.165)(.18,.23)
\qbezier(.2,.1)(.135,.13)(.07,.12)

{\color{white}%
\put(-.209,0.055){\circle*{.045}}
\put(-.152,-0.152){\circle*{.045}}
\put(-.055,-0.209){\circle*{.045}}
\put(.055,-0.209){\circle*{.045}}
\put(.152,-0.152){\circle*{.045}}
\put(-.209,-0.055){\circle*{.045}}
}
\qbezier(-.2,-.1)(-.17,-.165)(-.18,-.23)
\qbezier(-.2,-.1)(-.135,-.13)(-.07,-.12)

\end{picture}%
}

\put(6.83,1.48){%
\begin{picture}(0,0)

\qbezier(-.1,.2)(-.165,.17)(-.23,.18)
\qbezier(-.1,.2)(-.13,.135)(-.12,.07)

{\color{white}%
\put(.212,0.037){\circle*{.045}}
\put(.191,-0.107){\circle*{.045}}
\put(.055,-0.209){\circle*{.045}}
\put(-.055,-0.209){\circle*{.045}}
\put(-.152,-0.152){\circle*{.045}}
}
\qbezier(.1,-.2)(.165,-.17)(.23,-.18)
\qbezier(.1,-.2)(.13,-.135)(.12,-.07)

\end{picture}%
}

\thinlines

\qbezier(2.44,1.65)(2.52,2.58)(2.85,1.98)
\qbezier(2.43,1.5)(2.32,.57)(2,.3)
\qbezier(2.92,1.83)(3.5,.03)(7,.3)
\put(5.09,1.77){\line(1,3){.3}}
\put(5.05,1.65){\line(-1,-3){.5}}
\qbezier(6.92,1.75)(7.22,2.65)(8,2.8)
\qbezier(6.88,1.63)(6.48,.43)(6,.2)

\put(2.41,1.37){\circle*{.05}}
\put(2.97,1.71){\circle*{.05}}
\put(5,1.5){\circle*{.05}}
\put(6.83,1.47){\circle*{.05}}

\put(3.4,3.5){components of~$X(\R)$}

\put(3.45,3.4){\line(-2,-3){.7}}
\put(5,3.4){\line(1,-5){.19}}
\put(6.3,3.4){\line(3,-2){1.1}}

\put(.5,.2){\large$\R^m$}

\end{picture}
\end{center}
\caption{The loop~$\gamma$ in Lemma~\ref{slem_nu_chi}}\label{fig_beta}
\end{figure}
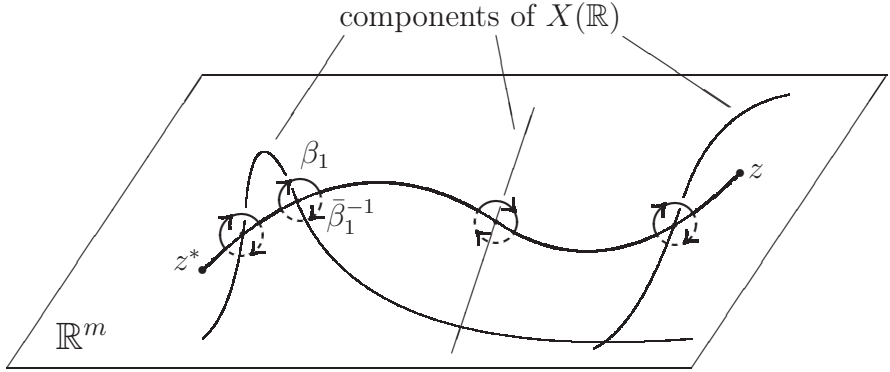

\begin{lem}\label{cor_nu_xi}
Suppose that $z\in Y^{reg}(\R)$ and $z\notin Z\cup W$. Then $\langle\nu,\xi(z)\rangle=0$.
\end{lem}

\begin{proof}
In a neighborhood of~$z$, the subvariety~$Y(\R)\subset\R^m$ is a smooth hypersurface. Choose two points $z_{\pm}\in \R^m\setminus X(\R)$ close to~$z$ and lying on the different sides of~$Y(\R)$. Then $\zeta(z_{\pm})=\zeta(z)$, $\xi(z_{\pm})=\xi(z)$, and $\eta(z_{+})= \eta(z_-)+\eta$. Hence $\chi(z_+)=\chi(z_-)+\eta$. Since $\langle\rho,\eta\rangle=\rho(\delta)=1$, we obtain that $\langle\rho,\chi(z_+)\rangle=\langle\rho,\chi(z_-)\rangle+1$. Therefore, exactly one of the two values $\langle\rho,\chi(z_{\pm})\rangle$ is equal to~$1$. We may assume that $\langle\rho,\chi(z_{+})\rangle=1$. Then, by Lemma~\ref{slem_nu_chi}, $\langle\nu,\chi(z_{+})\rangle=0$. By Lemmas~\ref{slem_mu_delta} and~\ref{slem_nu_zeta}, we have $\langle\nu,\eta\rangle=\nu(\delta)=0$ and $\langle\nu,\zeta_j\rangle=0$ for all~$j$. Hence $\langle\nu,\xi(z)\rangle=\langle\nu,\xi(z_+)\rangle=0$.
\end{proof}

By condition~(B), the vectors $\xi(z)$, where $z$ runs over all points in~$Y^{reg}(\R)$ that do not belong to~$Z\cup W,$ generate the subgroup of~$H$ spanned by $\xi_1,\ldots,\xi_r$. Hence, by Lemma~\ref{cor_nu_xi}, we obtain that $\langle\nu,\xi_j\rangle=0$ for $j=1,\ldots,r$. But, by Lemmas~\ref{slem_mu_delta} and~\ref{slem_nu_zeta}, $\langle\nu,\eta\rangle=0$ and $\langle\nu,\zeta_j\rangle=0$ for $j=1,\ldots,q$. Therefore, $\nu=0$, which is impossible since the homomorphism $\nu\colon\pi\to\Z_2$ must be surjective. This contradiction proves that $A=0$, i.\,e., $\mu(\gamma)=0$ for all $\gamma\in\pi$. Therefore $\I(\theta_{\rho})$ is totally real on~$U$, which completes the proof of Lemma~\ref{lem_tech}.  
\end{proof}

In the sequel, Lemma~\ref{lem_tech} will be used in the inductive proof of Theorem~\ref{theorem_key2}. To prove Theorem~\ref{theorem_key}, we shall need a slightly more complicated lemma. To formulate it, we need to introduce several concepts and some notation. We denote by $\tU$ the connected component of~$p^{-1}(U)$  that contains  the base point~$\tz^*$. Recall that each function~$\Psi$ holomorphic on~$\tFX$ can be considered as a multi-valued analytic function on~$\FX$. Suppose that~$\Psi$ has trivial variations along all loops~$\gamma$ contained in~$U$. Then, by definition, the \textit{principal branch\/} of~$\Psi$ on~$U$ is the branch realized on the connected component~$\tU$ of~$p^{-1}(U)$, i.\,e., the function~$\psi$ on~$U$ such that $\psi(p(\tz))=\Psi(\tz)$ for all $\tz\in\tU$.  The definition of the principal branch of a $1$-form holomorphic on~$\tFX$ with trivial variations along all loops~$\gamma$ contained in~$U$ is completely similar. We shall say that a function $\Psi\in\CO(\tFX)$ satisfies \textit{zero boundary conditions\/} on~$\tU$ if its principal branch~$\psi(z)$ has zero limit as a point~$z$ approaches to any point $z_0\in\partial U$ from values in~$U$.

\begin{lem}\label{lem_tech2}
Assume that  the $4$-tuple $(Y,Z,W,U)$ satisfies conditions~\textnormal{(A)} and~\textnormal{(B)} in Lemma~\ref{lem_tech}. Let  $\theta\in\A^1(\tFX)$ be a $1$-form satisfying conditions~\textnormal{(ii), (iii),} and\/~\textnormal{(iv)} in Lemma~\ref{lem_tech}. Suppose that the following conditions are satisfied instead of conditions~\textnormal{(i)} and~\textnormal{(v):} 
\begin{itemize}
\item[\textnormal{(i${}'$)}] The principal branch of~$\theta$ on~$U$ is well defined and real. In other words, $\Var_{\gamma}\theta=0$ for all loops~$\gamma$ contained in~$U,$ and~$\theta$ is real on~$\tU.$
\item[\textnormal{(i${}''$)}] The $1$-forms $\varphi_{\gamma}=i\bigl(M_{\gamma}\theta-(-1)^{\lk(\gamma,Y)}\theta\bigr)$ are totally real on~$U$ for all $\gamma\in\pi.$ 
\item[\textnormal{(v${}'$)}] In a neighborhood of each point $z_0\in Z_j$ that does not belong to the union of all other irreducible components of~$X$, any branch of\/~$\theta$ has the form
\begin{equation}\label{eq_omega3}
\theta=\frac{c\,dg_j(z)}{g_j(z)}+\omega,
\end{equation}
where $c$ is a real constant and\/ $\omega$ is a $1$-form holomorphic at~$z_0$.
\end{itemize}
Let $\Psi\in\CO(\tFX)$ be a function such that $d\Psi=\theta,$ \ $\Var_{\gamma}\Psi=0$ for all loops~$\gamma$ contained in~$U,$ and $\Psi$ satisfies zero boundary conditions on~$\tU$. Then the functions
$$\Phi_{\gamma}=i\bigl(M_{\gamma}\Psi-(-1)^{\lk(\gamma,Y)}\Psi\bigr)$$
are totally real on~$U$ for all $\gamma\in\pi$.
\end{lem}

\begin{proof}
Condition~(v${}'$) for~$\theta$ immediately implies that all $1$-forms~$\varphi_{\gamma}$ satisfy condition~(v) in Lemma~\ref{lem_tech}. All other conditions of Lemma~\ref{lem_tech} for~$\varphi_{\gamma}$ also follow from the corresponding conditions for~$\theta$. Hence, applying Lemma~\ref{lem_tech} to the $1$-forms~$\varphi_{\gamma}$, we may conclude that all functions~$\I(\varphi_{\gamma})$ are totally real on~$U$. Nevertheless, the functions~$\Phi_{\gamma}$, which we are interested in, do not coincide with the functions~$\I(\varphi_{\gamma})$ but differ from them by constants. Namely,
\begin{equation*}
\Phi_{\gamma}(\tz)=\I(\varphi_{\gamma})(\tz)+\Phi_{\gamma}(\tz^*).
\end{equation*}
Therefore, we still need to prove that the values~$\Phi_{\gamma}(\tz^*)$ are real for all $\gamma\in\pi$. This proof follows the same line as the proof of Lemma~\ref{lem_tech}, though differs in many details. To stress the similarity with the proof of Lemma~\ref{lem_tech}, we shall use the same notation for the objects that play similar roles in these two proofs. 

Consider the homomorphism $\rho\in \Hom(\pi,\Z_2)=H^*$ given by
$$
\rho(\gamma)=\lk(\gamma,Y)\mod 2.
$$
Then $\varphi_{\gamma}=i\Var_{\gamma}^{\rho}\theta$ and $\Phi_{\gamma}=i\Var_{\gamma}^{\rho}\Psi$. Let~$\delta$ be the same loop as in the proof of Lemma~\ref{lem_tech}. Then $\lk(\delta,Y)=1$, hence, $\rho(\delta)=1$. As in the proof of Lemma~\ref{lem_tech}, condition~(iv) implies that $M_{\delta}\theta=-\theta$.

Consider the mapping $\mu\colon\pi\to\R$ given by
\begin{equation*}
\mu(\gamma)=\Im \Phi_{\gamma}(\tz^*)=\Re\Psi(T_{\gamma}\tz^*)-(-1)^{\rho(\gamma)}\Re\Psi(\tz^*).
\end{equation*}
Our aim is to prove that $\mu(\gamma)=0$ for all~$\gamma$. 

\begin{lem}\label{slem_mu_delta2}
 $\mu(\delta)=0$. 
\end{lem}
\begin{proof}
Let $\alpha\colon[0,1]\to\C^m$ be the path from~$z^*$ to a point $z_0\in\partial U$ that was used in the construction of the loop~$\delta$, see the proof of Lemma~\ref{lem_tech}. Since  $\alpha$ is contained in~$U$ except for its endpoint~$z_0$, the zero boundary conditions for~$\Psi$ on~$\tU$ imply that $\lim_{t\to 1}\Psi(\tilde\alpha(t))=0$. But $d\Psi=\theta$. Hence,
\begin{align*}
\Psi(\tz^*)&=-\int_{\talpha}\theta,\\
\Psi(T_{\delta}\tz^*)&=-\int_{T_{\delta}\talpha}\theta=-\int_{\talpha}M_{\delta}\theta=\int_{\talpha}\theta.
\end{align*}
Therefore,
$$
\Phi_{\delta}(\tz^*)=i(\Psi(T_{\delta}\tz^*)+\Psi(\tz^*))=0.
$$
Thus, $\mu(\delta)=0$.
\end{proof}

\begin{lem}\label{slem_mu_ah2}
For all\/ $\gamma_1,\gamma_2\in\pi,$ we have
\begin{equation}\label{eq_mu_sum2}
\mu(\gamma_1\gamma_2)=\mu(\gamma_1)+(-1)^{\rho(\gamma_1)}\mu(\gamma_2).
\end{equation}
\end{lem}

\begin{proof}
As it was shown above, the functions~$\I(\varphi_{\gamma})$ are totally real on~$U$ for all $\gamma\in\pi$. In particular, the value $\I(\varphi_{\gamma_1})(T_{\gamma_2}\tz^*)$ is real. We have
\begin{multline*}
\I(\varphi_{\gamma_1})(T_{\gamma_2}\tz^*)=\Phi_{\gamma_1}(T_{\gamma_2}\tz^*)-\Phi_{\gamma_1}(\tz^*)={}\\{}=i\bigl(\Psi(T_{\gamma_1\gamma_2}\tz^*)-(-1)^{\rho(\gamma_1)}\Psi(T_{\gamma_2}\tz^*)
-\Psi(T_{\gamma_1}\tz^*)+(-1)^{\rho(\gamma_1)}\Psi(\tz^*)
\bigr).
\end{multline*}
Equating the imaginary part of this number with zero, we obtain exactly equality~\eqref{eq_mu_sum2}.
\end{proof}

In the same way as Lemmas~\ref{slem_mu_delta} and~\ref{slem_mu_ah} implied Lemma~\ref{lem_image_mu}, Lemmas~\ref{slem_mu_delta2} and~\ref{slem_mu_ah2} imply that the image of~$\mu$ is a finitely generated subgroup $A\subset\R$.  Assume that $A\ne 0$, choose an arbitrary epimorphism $\varpi\colon A\to\Z_2$, and consider the surjective mapping
$$
\nu\colon\pi\xrightarrow{\mu}A\xrightarrow{\varpi}\Z_2.
$$
It follows from Lemma~\ref{slem_mu_ah2}  that $\nu$ is a homomorphism, i.\,e., $\nu\in H^*$.

\begin{lem}\label{slem_nu_zeta2}
We have $\langle\nu,\zeta_j\rangle=0$ for $j=1,\ldots,q$.
\end{lem}
\begin{proof}
In the proof of Lemma~\ref{slem_nu_zeta} we deduced equality~\eqref{eq_2pic} from condition~(v). In the same way condition~(v${}'$) implies the equality 
\begin{equation*}
\Psi(T_{\gamma}\tilde z^*)-\Psi(\tz^*)=2\pi i c,
\end{equation*}
where the loop $\gamma$ passing around~$Z_j$ is the same as in the proof of Lemma~\ref{slem_nu_zeta}. Since $\rho(\gamma)=0$, it follows that $\mu(\gamma)=0$, hence, $\langle\nu,\zeta_j\rangle=0$. 
\end{proof}

\begin{lem}\label{slem_nu_chi2}
Suppose that $z\in\R^m\setminus X(\R)$ and $\langle\rho,\chi(z)\rangle=0$. Then $\langle\nu,\chi(z)\rangle=0$.
\end{lem}

\begin{proof}
Construct a loop $\gamma$ in the same way as in the proof of Lemma~\ref{slem_nu_chi}. Recall that $[\gamma]=\chi(z)$ and~$\gamma^{-1}=\bar\gamma$. Since $\rho(\gamma)=0$, we have
$$
\mu(\gamma)=\Re\Psi(T_{\gamma}\tz^*)-\Re\Psi(\tz^*)=\Re\int_{\tilde\gamma}\theta.
$$
Similarly,
$$
\mu(\gamma^{-1})=\mu(\bar\gamma)=\Re\int_{\Tilde{\bar\gamma}}\theta.
$$
 Since the $1$-form~$\theta$ is real on~$\tU$, we see that the integrals of~$\theta$ along the lifts~$\tilde\gamma$ and $\Tilde{\bar\gamma}$ of~$\gamma$ and~$\bar\gamma$, respectively, starting at~$\tz^*$ are conjugate to each other. Hence,  $\mu(\gamma^{-1})=\mu(\gamma)$. On the other hand, it follows from Lemma~\ref{slem_mu_ah2} that $\mu(\gamma^{-1})=-\mu(\gamma)$. Therefore, $\mu(\gamma)=0$. Thus, $\langle\nu,\chi(z)\rangle=0$.
\end{proof}

\begin{lem}\label{cor_nu_xi2}
Suppose that $z\in Y^{reg}(\R)$ and $z\notin Z\cup W$. Then $\langle\nu,\xi(z)\rangle=0$.
\end{lem}

\begin{proof}
As in the proof of Lemma~\ref{cor_nu_xi}, take two points $z_{\pm}\in \R^m\setminus X(\R)$ close to~$z$ and lying on the different sides of~$Y(\R)$. Then  $\langle\rho,\chi(z_+)\rangle=\langle\rho,\chi(z_-)\rangle+1$. Therefore, exactly one of the two values $\langle\rho,\chi(z_{\pm})\rangle$ is equal to~$0$. We may assume that $\langle\rho,\chi(z_{-})\rangle=0$. Then, by Lemma~\ref{slem_nu_chi2}, $\langle\nu,\chi(z_{-})\rangle=0$. By Lemmas~\ref{slem_mu_delta2} and~\ref{slem_nu_zeta2}, we have $\langle\nu,\eta\rangle=0$ and $\langle\nu,\zeta_j\rangle=0$ for all~$j$. Hence $\langle\nu,\xi(z)\rangle=\langle\nu,\xi(z_-)\rangle=0$.
\end{proof}

As in the proof of Lemma~\ref{lem_tech}, Lemmas~\ref{slem_mu_delta2}, \ref{slem_nu_zeta2}, and~\ref{cor_nu_xi2}, and condition~(B) imply that $\nu=0$, which is impossible, since the homomorphism $\nu\colon\pi\to\Z_2$ must be surjective. This contradiction yields that  $A=0$, i.\,e., $\mu(\gamma)=0$ for all $\gamma\in\pi$, which completes the proof of Lemma~\ref{lem_tech2}.  
\end{proof}

\section{The analytic continuation of volume}\label{section_continue}

 As in the Introduction, consider the affine space $\CG_{(n)}(\C)\cong \C^{n(n+1)/2}$ of symmetric complex matrices of size $(n+1)\times(n+1)$ with units on the diagonal, the affine hypersurface $X\subset \CG_{(n)}(\C)$ consisting of all matrices~$C$ such that at least one of the principal minors of~$C$ vanishes, and  the domains $\CC_{\bS^n}$ and~$\CC_{\Lambda^n}$ in $\CG_{(n)}(\R)\subset\CG_{(n)}(\C)$ consisting of the Gram matrices of vertices of all non-degenerate simplices in~$\bS^n\subset\R^{n+1}$ and of all non-degenerate bounded simplices in~$\Lambda^n\subset\R^{1,n}$, respectively. 
  
Let $\X^n$ be either~$\bS^n$ or~$\Lambda^n$. We consider the complex analytic manifold $\FX= \CG_{(n)}(\C)\setminus X$, and take for the base point $C^*\in\FX$ the point corresponding to the regular simplex in~$\X^n$ with edge~$1$; then $C^*\in\CC_{\X^n}$.  We shall use  all notation introduced in Section~\ref{section_fh}. However, the base points for~$\tFX$ and~$\hFX$ will be denoted by $\tC^*$ and $\hC^*$ rather than by~$\tz^*$ and~$\hz^*$, respectively.

We introduce the parameter~$\varepsilon$ that is equal to~$1$ for $\X^n=\bS^n$ and to~$-1$ for $\X^n=\Lambda^n$.

\begin{remark}
The base points~$C^*$  in the cases $\X^n=\bS^n$ and $\X^n=\Lambda^n$ are different from each other. Hence, the universal coverings~$\tFX$ of~$\FX$ in these two cases  are isomorphic to each other, but the isomorphism is not canonical. The same is true for~$\hFX$. We shall always remember that~$\tFX$ and~$\hFX$ depend on $\X^n$, but we shall not indicate this explicitly in notation. 
\end{remark}

Let $\varphi$ be a function holomorphic at~$C^*$. Assume that $\varphi$ admits the analytic continuation along every path in~$\FX$. Then the analytic continuation of~$\varphi$ is a well-defined holomorphic function  on~$\tFX$, which we shall denote by~$\widetilde{\varphi}$. More precisely, we put $\widetilde{\varphi}(\tC)=\varphi(p(\tC))$ for $\tC$ close to~$\tC^*$, and then continue analytically the function~$\widetilde{\varphi}(\tC)$. Obviously, if $\varphi$ is holomorphic on the whole manifold~$\FX$, then $\widetilde{\varphi}$ is just the pullback of~$\varphi$ by~$p$, i.\,e., $\widetilde{\varphi}(\tC)=\varphi(p(\tC))$ for all~$\tC\in\tFX$.

As in the Introduction, we consider   the function~$V_{\X^n}(C)$ computing the volume of a simplex in~$\X^n$  from the cosines (if $\X^n=\bS^n$) or the hyperbolic cosines (if $\X^n=\Lambda^n$) of its edge lengths. This function is defined and real analytic on  the domain~$\CC_{\X^n}$. Aomoto~\cite{Aom77}  proved that the function~$V_{\X^n}$ admits the analytic continuation to the multi-valued analytic function on~$\FX$, and the monodromy of this function on~$\hFX$ is unipotent. In our terminology, his result is as follows.

\begin{propos}\label{propos_fh}
The analytic continuation of\/~$V_{\X^n}$ is a well-defined holomorphic function $\tV_{\X^n}\in\CO(\tFX)$. Besides, $d\tV_{\X^n}\in\hCF^1_{\lceil\frac{n}{2}\rceil-1}$ and\/ $\tV_{\X^n}\in\hCF^0_{\lceil\frac{n}{2}\rceil}$.
\end{propos}  

Aomoto proved this result in the following way. He used Schl\"afli's formula to write the function~$\tV_{\X^n}$ as a homotopy invariant linear combination of iterated integrals of $1$-forms holomorphic on~$\hFX$, and then concluded that the monodromy of this function on~$\hFX$ is unipotent (see Proposition~\ref{propos_iterate}). Though the formula for the volume via iterated integrals is interesting in itself, the usage of iterated integrals for the proof of Proposition~\ref{propos_fh} is superfluous. Indeed, Proposition~\ref{propos_fh} can be deduced directly from Schl\"afli's formula and the basic properties of unipotent filtrations listed in Lemma~\ref{lem_properties}. We shall give this simplified version of Aomoto's proof of  Proposition~\ref{propos_fh} below, since the details of this proof will be useful in our further study  of the function~$\tV_{\X^n}$. 

Now, let us introduce some notation.
For any subsets $I,J\subset \{0,\ldots,n\}$ such that $|I|=|J|$, we denote by~$D_{I,J}(C)$ the minor of a matrix~$C\in\CG_{(n)}(\C)$  formed by rows with numbers in~$I$ and columns with numbers in~$J$. The principal minor~$D_{I,I}(C)$ will be denoted by~$D_I(C)$. The submatrix of~$C$ formed by rows and columns with numbers in~$I$ will be denoted by~$C_I$; then $D_I(C)=\det C_I$.  
The determinant  of the matrix~$C$ will be denoted by~$D(C)$, i.\,e., $D(C)=D_{\{0,\ldots, n\}}(C)$. It is well known that the polynomials $D_I(C)$ are irreducible whenever $|I|\ge 3$, and $D_{\{j,k\}}(C)=(1+c_{jk})(1-c_{jk})$. We put $D_{\{j,k\}}^{\pm}(C)=1\pm c_{jk}$. 

For each subset $I\subset \{0,\ldots,n\}$, $|I|\ge 2$, we denote by~$\CH_I$ the affine hypersurface in~$\CG_{(n)}(\C)$ given by the equation $D_I(C)=0$, and we put $\CH=\CH_{\{0,\ldots, n\}}$.
If $|I|=2$, we denote by $\CH_I^{\pm}$ the hypersurfaces (hyperplanes) given by $D_I^{\pm}(C)=0$. Then $X$ is the union of all hypersurfaces~$\CH_I$, $I\subset \{0,\ldots,n\}$, $|I|\ge 2$. The irreducible components of~$X$ are all~$\CH_I$ such that $|I|\ge 3$ and all $\CH_I^+$ and $\CH_I^-$ such that $|I|=2$. (Obviously, no two of these components coincide to each other.) Thus the number of irreducible components of~$X$ is $s=2^{n+1}+\frac{n(n-1)}{2}-2$.

 It is a standard fact that the set~$\CC_{\bS^n}$ consists of all positive definite matrices in~$\CG_{(n)}(\R)$, and the set $\CC_{\Lambda^n}$ consists of all matrices $C\in\CG_{(n)}(\R)$ such that $(-1)^{|I|}D_I(C)<0$ for all~$I$ and $c_{jk}>1$ for all~$j$ and~$k$. The latter condition can be rewritten in the form $D_I^+(C)>0$  and $D_I^-(C)<0$ for all~$I$ such that $|I|=2$. It follows that both $\CC_{\bS^n}$ and~$\CC_{\Lambda^n}$ are connected components of $\CG_{(n)}(\R)\setminus X(\R)$. (The sets~$\CC_{\bS^n}$ and~$\CC_{\Lambda^n}$ are connected, since the spaces of non-degenerate simplices in~$\bS^n$ and in~$\Lambda^n$, respectively, are connected.)

Below we shall often use  induction on the dimension~$n$. Hence we  shall introduce  an additional index~$(n)$ indicating the dimension of the space~$\X^n$ to which an object is related, for instance, we shall write $X_{(n)}$, $\FX_{(n)}$, $\hCF_{k,(n)}^q$, etc. Nevertheless, to simplify the notation, we shall omit this index whenever it is clear. 

Let $I\subset\{0,\ldots,n\}$ be a subset of cardinality $m+1$. We denote by $\pr_I$ the projection $\CG_{(n)}(\C)\to \CG_{(m)}(\C)$ given by  $C\mapsto C_I$. Obviously, $\pr_I(\FX_{(n)})=\FX_{(m)}$, $\pr_I(\CC_{\X^n})=\CC_{\X^m}$, and $\pr_I(C^*_{(n)})=C^*_{(m)}$. The projection $\pr_I\colon\FX_{(n)}\to\FX_{(m)}$ is covered by the well-defined mappings $\hpr_I\colon \hFX_{(n)}\to\hFX_{(m)}$ and $\tpr_I\colon\tFX_{(n)}\to\tFX_{(m)}$ such that $\hpr_I(\hC^*_{(n)})=\hC^*_{(m)}$ and $\tpr_I(\tC^*_{(n)})=\tC^*_{(m)}$. Consider the pullback  homomorphism $$\tpr_I^*\colon\A^q(\tFX_{(m)})\to\A^q(\tFX_{(n)}).$$ It follows immediately from the definition of~$\hCF^q_k$ that $\tpr_I^*\bigl(\hCF^q_{k,(m)}\bigr)\subset\hCF^q_{k,(n)}$ for all $k$ and~$q$.

Now, suppose that $C\in\CC_{\X^n}$, and let $\Delta^n$ be a simplex in~$\X^n$ corresponding to~$C$. For each subset $I\subset\{0,\ldots,n\}$, we denote by $\Delta_I$ the face of~$\Delta^n$ spanned by vertices with numbers in~$I$. If $|I|=n-1$, we denote by~$\alpha_I=\alpha_I(C)$ the dihedral angle of the simplex~$\Delta^n$ at the $(n-2)$-dimensional face~$\Delta_I$.  Consider the standard vector model of~$\X^n$ in~$\V$, where $\V=\R^{n+1}$ if~$\X^n=\bS^n$ and $\V=\R^{1,n}$ if~$\X^n=\Lambda^n$. Let $\bv_0,\ldots,\bv_n\in\V$ be the vectors representing the vertices of~$\Delta^n$, and let $\bv^0,\ldots,\bv^n$ be the basis of~$\V$ dual to the basis $\bv_0,\ldots,\bv_n$. Then $\bv^0,\ldots,\bv^n$ are interior normal vectors to the $(n-1)$-dimensional faces of~$\Delta^n$ if $\X^n=\bS^n$ and exterior normal vectors to the the $(n-1)$-dimensional faces of~$\Delta^n$ if $\X^n=\Lambda^n$. Hence
$$
\cos\alpha_I(C)=-\frac{\varepsilon\langle\bv^j,\bv^k\rangle}
{\sqrt{\langle\bv^j,\bv^j\rangle\langle\bv^k,\bv^k\rangle}}\,.
$$
 Since the Gram matrix of the vectors $\bv^0,\ldots,\bv^n$ is~$C^{-1}$ and $\varepsilon^nD(C)>0$, we  obtain that 
\begin{equation}\label{eq_sin-alpha}
\cos\alpha_I(C)=\frac{\varepsilon^{n-1}D_{I',I''}(C)}{\sqrt{D_{I'}(C)D_{I''}(C)}}\,,\qquad
\sin\alpha_I(C)=\sqrt{\frac{D(C)D_I(C)}{D_{I'}(C)D_{I''}(C)}}\,,
\end{equation}
where $I'$ and $I''$ are the two $n$-element subsets of $\{0,\ldots,n\}$ containing~$I$ (cf.~\cite[(3.10)]{Aom77}). The second formula is obtained from the first one using Jacobi's identity
\begin{equation}\label{eq_Jac}
DD_I=D_{I'}D_{I''}-D_{I',I''}^2.
\end{equation}
Differentiating the first of formulae~\eqref{eq_sin-alpha}, we easily get
\begin{equation}\label{eq_d_alpha}
d\alpha_I=\frac{\varepsilon^{n-1}}{\sqrt{DD_I}}\left(
\frac{D_{I',I''}\,d(D_{I'}D_{I''})}{2D_{I'}D_{I''}}-dD_{I',I''}\right).
\end{equation}
For $C\in\CC_{\X^n}$, we have $\varepsilon^{|I|+1}D_I(C)>0$, and we put 
$$
R_I(C)=\sqrt{\varepsilon^{|I|+1}D_I(C)}\,,
$$
where the positive value of the square root is chosen. In particular, we put $R(C)=\sqrt{\varepsilon^{n}D(C)}$.

By Schl\"afli's formula~\eqref{eq_Schlaefli} for~$\Delta^n$, the following equality holds on~$\CC_{\X^n}$:
\begin{equation}\label{eq_Sch_simp}
d V_{\X^n}(C)=\frac{\varepsilon}{n-1}\sum_{I\subset\{0,\ldots,n\},\,|I|=n-1}V_{I}(C)\,d\alpha_I(C),
\end{equation}
where $V_I(C)=V_{\X^{n-2}}(C_I)$ is the $(n-2)$-dimensional volume of the face~$\Delta_I$. 

\begin{proof}[Proof of Proposition~\ref{propos_fh}]
The proof is by induction on~$n$ with step $2$. It is convenient to start the induction from the cases $n=0$ and $n=1$.

For $n=0$, we shall use the convention that  $\FX$ is a point and $V_{\X^0}=1$. Then $\tFX=\hFX=\FX$ and $\hCF_0^0=\CO(\tFX)=\C$. So the assertion of the proposition is trivial.

For $n=1$, we have only one variable $c_{01}$ and $\FX=\C\setminus\{-1,1\}$. Then
\begin{gather}\label{eq_n=1}
V_{\bS^1}(c_{01})=\arccos c_{01},\qquad V_{\Lambda^1}(c_{01})=\arcosh c_{01},\\
dV_{\X^1}=\frac{-\varepsilon\,dc_{01}}{\sqrt{\varepsilon(1-c_{01}^2)}}\,.
\end{gather}
The analytic continuation of  $d V_{\X^1}$ becomes single-valued on the two-sheeted branched covering of~$\C$ with ramification at~$\pm 1$, hence, a fortiori,  it becomes  single-valued on the four-sheeted covering~$\hFX$ of~$\FX$. Therefore, $d\tV_{\X^1}$ belongs to  $\hCF^1_0=\hp^*\A^1(\hFX)$. It follows from assertion~(4) of Lemma~\ref{lem_properties} that $\tV_{\X^1}\in\hCF^0_1$.

Now, let us prove the induction step. Assume that the proposition is proved for $n-2$ and prove it for~$n$. The proof will be based on Schl\"afli's formula~\eqref{eq_Sch_simp}. Notice that our convention $V_{\X^0}=1$ makes this formula true for $n=2$, too. Indeed, in this case it takes the form $dV_{\X^2}=\varepsilon(d\alpha_{\{0\}}+d\alpha_{\{1\}}+d\alpha_{\{2\}})$, which is true, since 
\begin{equation}\label{eq_area}									
V_{\X^2}=\varepsilon(\alpha_{\{0\}}+\alpha_{\{1\}}+\alpha_{\{2\}}-\pi).
\end{equation}

Obviously, the analytic continuations~$\widetilde{R}_I\in\CO(\tFX)$ of the functions~$R_I$ are well defined. Moreover, the analytic continuations of all~$R_I$ become single-valued not only on~$\tFX$ but also on~$\hFX$. Therefore, all functions~$\widetilde{R}_I$ belong to the subspace $\hp^*\CO(\hFX)\subset\CO(\tFX)$. Now, formula~\eqref{eq_d_alpha} implies that the functions~$\alpha_I$ admit the analytic continuations~$\talpha_I\in \CO(\tFX)$ such that the $1$-forms  
\begin{equation}\label{eq_d_talpha}
d\talpha_I=\frac{\varepsilon^{n-1}}{\tR\tR_I}\left(
\frac{\tD_{I',I''}\,d(\tD_{I'}\tD_{I''})}{2\tD_{I'}\tD_{I''}}-d\tD_{I',I''}\right)
\end{equation}
belong to the subspace $\hCF^1_0=\hp^*\A^1(\hFX)\subset\A^1(\tFX)$. (Notice, however, that the functions~$\talpha_I$ themselves do not belong to $\hp^*\CO(\hFX)$.)

By the inductive assumption, the analytic continuation $\tV_{\X^{n-2}}\in\CO(\tFX_{(n-2)})$ of the function~$V_{\X^{n-2}}$ is well defined. Then
it follows from~\eqref{eq_Sch_simp} that the $1$-form~$dV_{\X^n}$ can be continued analytically to the closed $1$-form $d\tV_{\X^n}\in\A^1(\tFX_{(n)})$ given by
\begin{equation}\label{eq_dCVn}
d\tV_{\X^n}=\frac{\varepsilon}{n-1}\sum_{I\subset\{0,\ldots,n\},\,|I|=n-1}\tpr_I^*(\tV_{\X^{n-2}})\,d\talpha_I.
\end{equation}
 The required function $\tV_{\X^n}\in\CO(\tFX)$ is given by 
\begin{equation}\label{eq_CVn}
\tV_{\X^n}(\tC)=V^*_{\X^n}+\frac{\varepsilon}{n-1}\int_{\tC^*_{(n)}}^{\tC}\sum_{I\subset\{0,\ldots,n\},\,|I|=n-1}\tpr_I^*(\tV_{\X^{n-2}})\,d\talpha_I,
\end{equation}
where $V^*_{\X^n}=V_{\X^n}(C^*_{(n)})$ is the volume of the regular simplex in~$\X^n$ with edge~$1$. 
Since $\tFX$ is simply connected, the integral in the right-hand side of~\eqref{eq_CVn} is  independent of the path. Besides, we see that $\tV_{\X^n}(\widetilde{C})=V_{\X^n}(p(\widetilde{C}))$ in a neighborhood of~$\tC^*_{(n)}$.

By the inductive assumption, we have $\tV_{\X^{n-2}}\in\hCF^0_{\lceil\frac{n}{2}\rceil-1,(n-2)}$. Hence we have $\tpr_I^*(\tV_{\X^{n-2}})\in\hCF^0_{\lceil\frac{n}{2}\rceil-1,(n)}$ for all~$I$ such that $|I|=n-1$. But $d\talpha_I\in\hCF^1_{0,(n)}$. Therefore, assertion~(2) of Lemma~\ref{lem_properties} implies that $\tpr_I^*(\tV_{\X^{n-2}})\,d\talpha_I\in \hCF^1_{\lceil\frac{n}{2}\rceil-1,(n)}$. Consequently, by~\eqref{eq_CVn}, $d \tV_{\X^n}\in\hCF^1_{\lceil\frac{n}{2}\rceil-1,(n)}$. Thus, by  assertion~(4) of Lemma~\ref{lem_properties}, $\tV_{\X^n}\in\hCF^0_{\lceil\frac{n}{2}\rceil,(n)}$.
\end{proof}

To prove Theorems~\ref{theorem_key2} and~\ref{theorem_key}, we would like to apply Lemmas~\ref{lem_tech}  and~\ref{lem_tech2}, respectively, to the $1$-form~$d\tV_{\X^n}$. 
So we need to decompose~$X$ into the union $Y\cup Z\cup W$. We put,
$$
Y=\CH,\qquad Z=\bigcup_{|I|=n}\CH_I\,,\qquad W=\bigcup_{2\le |I|<n}\CH_I\,.
$$
Below in this section and in the next section we always assume that $n\ge 2$; then $Y$ is irreducible.
In the next section we shall check that $X$ is essentially real and that the $4$-tuples $(Y,Z,W,\CC_{\bS^n})$ and $(Y,Z,W,\CC_{\Lambda^n})$ satisfy conditions~(A) and~(B) in Lemma~\ref{lem_tech}.
Obviously, the $1$-form $d\tV_{\X^n}$ is closed, i.\,e., satisfies condition~(ii) in Lemma~\ref{lem_tech}. By Proposition~\ref{propos_fh}, the $1$-form $d\tV_{\X^n}$ satisfies condition~(iii).

\begin{lem}\label{lem-iv-a}
The $1$-form $d\tV_{\X^n}$ satisfies condition \textnormal{(iv)} in Lemma~\ref{lem_tech}.
\end{lem}

\begin{proof}
Since $Y=\CH$, we obtain that the irreducible equation giving $Y$ is $D(C)=0$. Let $C^0$ be a point in~$Y$ that does not belong to $Z\cup W$. Then formula~\eqref{eq_d_talpha} immediately yields that,  in a neighborhood of~$C^0$, any branch of every  $d\talpha_I$ has the form $\omega_I/\sqrt{D(C)}$, where $\omega_I$ is a  $1$-from holomorphic at~$C^0$. But all branches of all multi-valued functions $\tpr_I^*\tV_{\X^{n-2}}(C)= \tV_{\X^{n-2}}(C_I)$ have no singularities off~$W$. Hence condition~(iv) for the $1$-form $d\tV_{\X^n}$ follows from~\eqref{eq_dCVn}.
\end{proof}

To apply Lemma~\ref{lem_tech} (respectively, Lemma~\ref{lem_tech2}) we still need to check conditions~(i) and~(v) (respectively, conditions~(i${}'$), (i${}''$), and~(v${}'$)). Further,  we shall check that   $d\tV_{\X^n}$ satisfies conditions~(i) and~(v) if $\X^n$ is either $\bS^n$ or~$\Lambda^{2m}$, and satisfies conditions~(i${}'$), (i${}''$), and~(v${}'$)  if $\X^n=\Lambda^{2m+1}$ using Theorems~\ref{theorem_key2} and~\ref{theorem_key} for~$\X^{n-2}$, respectively. Then Theorems~\ref{theorem_key2} and~\ref{theorem_key} for~$\X^{n}$ will follow from Lemmas~\ref{lem_tech} and~\ref{lem_tech2}, respectively. Such inductive proofs of Theorems~\ref{theorem_key2} and~\ref{theorem_key} will be given in Section~\ref{section_proof_key} after we check conditions~(A) and~(B).

\section{Conditions~(A) and~(B)}\label{section_AB}

In this section we always assume that $n\ge 2$.

\begin{lem}
The hypersurface $X\subset\CG_{(n)}(\C)$ is essentially real. 
\end{lem}

\begin{proof}
Irreducible components of~$X$ are given by the polynomial equations~$D_I(C)=0$ and $D_I^{\pm}(C)=0$ with real coefficients.  Moreover, each of the polynomials~$D_I(C)$ and~$D_I^{\pm}(C)$ takes on~$\CG_{(n)}(\R)$ both positive and negative values.  Therefore each of the irreducible components of~$X(\R)$ separates~$\CG_{(n)}(\R)$, hence, contains regular points.
\end{proof}

{\sloppy
\begin{lem}\label{lem_A}
The $4$-tuples\/ $(Y,Z,W,\CC_{\bS^n})$ and\/ $(Y,Z,W,\CC_{\Lambda^n})$ satisfy condition\/~\textnormal{(A)} in Lemma~\ref{lem_tech}.
\end{lem}

}

\begin{proof}
Take $n+1$ points $v_0,\ldots,v_n$ in~$\X^n$ such that they lie in an $(n-1)$-dimensional plane, but any $n$ of them do not lie in an  $(n-2)$-dimensional plane. Then the simplex with vertices $v_0,\ldots,v_n$ is degenerate but all its proper faces are non-degenerate. Let $C^0=(c_{jk}^0)$ be the matrix of the cosines (if $\X^n=\bS^n$) or hyperbolic cosines (if $\X^n=\Lambda^n$) of the pairwise distances between the points $v_0,\ldots,v_n$. Then $D(C^0)=0$ and $D_I(C^0)\ne 0$ for all $I$ such that $|I|\le n$. Hence $C^0\in Y(\R)$ and $C^0\notin Z\cup W$. Besides, using Jacobi's identity~\eqref{eq_Jac}, we see that
$$
\left.\frac{\partial D}{\partial c_{01}}\right|_{C^0}=-2D_{\{0,2,\ldots,n\},\{1,2,\ldots,n\}}(C^0)=
\pm 2\sqrt{D_{\{0,2,\ldots,n\}}(C^0)D_{\{1,2,\ldots,n\}}(C^0)}\ne 0.
$$
Therefore $C^0$ is a regular point of~$Y$. Finally, since the points $v_0,\ldots,v_n$ can be shifted arbitrarily small so that the simplex with vertices at them will become non-degenerate, we see that $C^0\in\partial\CC_{\X^n}$.
\end{proof}

If $n=2$, then condition~(B) is satisfied, since $W=\emptyset$. So we assume that $n\ge 3$. 
To prove that  the $4$-tuples $(Y,Z,W,\CC_{\bS^n})$  and  $(Y,Z,W,\CC_{\Lambda^n})$ satisfy condition~(B), we need the following auxiliary lemma. Let $W_1,\ldots, W_r$ be the irreducible components of~$W$ listed in an arbitrary order. It is easy to compute that $$r=2^{n+1}+\frac{n^2-3n}{2}-4.$$

\begin{lem}\label{lem_Cs}
For every $\lambda=1,\ldots,r$, there exists a matrix~$C^{\lambda}\in\CG_{(n)}(\R)$ such that the smooth hypersurfaces $Y^{reg}(\R)$ and $W_{\lambda}^{reg}(\R)$ intersect transversely at\/~$C^{\lambda},$ and $C^{\lambda}$ does not lie in the union of\/~$Z$ and all components\/~$W_{\mu}$ such that\/ $\mu\ne \lambda$.    
\end{lem}

\begin{proof}
The irreducible component $W_{\lambda}$ is either $\CH_I$, where $3\le |I|<n$, or~$\CH_I^{\sigma}$, where $|I|=2$ and $\sigma\in\{+,-\}$. Without loss of generality, we may assume that $I=\{0,\ldots,k-1\}$, where $2\le k<n$. If $|I|\ge 3$, we choose a sign~$\sigma$ arbitrarily. Consider the matrix $C^{\lambda}=(c_{jl}^{\lambda})\in\CG_{(n)}(\R)$ such that
\begin{gather*}
\begin{aligned}
c_{0j}^{\lambda}&=c_{j0}^{\lambda}=
-\sigma\,\frac{1}{\sqrt{k-1}}\,,&&1\le j\le k-1,\\
c_{0j}^{\lambda}&=c_{j0}^{\lambda}=2,&&k\le j\le n-1,\\
\end{aligned}\\
c_{1n}^{\lambda}=c_{n1}^{\lambda}=\left(1-\frac{1}{4(k-1)(n-k)}\right)^{-1/2},
\end{gather*}
and all other non-diagonal entries of~$C^{\lambda}$ are equal to~$0$.
We denote the set $\{k,\ldots,n-1\}$ by~$K$. A direct computation shows that the principal minors of~$C^{\lambda}$ are as follows:
$$
D_J\bigl(C^{\lambda}\bigr)=\left\{
\begin{aligned}
&-\frac{4|K\setminus J|+\frac{1}{k-1}|I\setminus J|}{4(k-1)(n-k)-1}& &\text{if $0,1,n\in J$,}\\
&1-\frac{|J\cap I|-1}{k-1}-4|J\cap K|& &\text{if $0\in J$ and $\{1,n\}\not\subset J$,}\\
&-\frac{1}{4(k-1)(n-k)-1}& &\text{if $0\notin J$ and $1,n\in J$,}\\
&1&&\text{if $0\notin J$ and $\{1,n\}\not\subset J$.}
\end{aligned}
\right.
$$
Hence $D\bigl(C^{\lambda}\bigr)=D_I\bigl(C^{\lambda}\bigr)=0$ and all other principal minors of~$C^{\lambda}$ are non-zero. Besides, if $|I|=2$, then $D_I^{\sigma}\bigl(C^{\lambda}\bigr)=0$ and $D_I^{-\sigma}\bigl(C^{\lambda}\bigr)\ne0$. Therefore $C^{\lambda}\in Y(\R)\cap W_{\lambda}(\R)$ and $C^{\lambda}$ does not lie in any other irreducible components of~$X$. Further, we have
$$
\left.\frac{\partial D_I}{\partial c_{01}}\right|_{C^{\lambda}}=\sigma \frac{2}{\sqrt{k-1}}\,,\qquad\left.\frac{\partial D}{\partial c_{0k}}\right|_{C^{\lambda}}=\frac{16}{4(k-1)(n-k)-1}\,.
$$
Since $D_I(C)$ is independent of~$c_{0k}$, we see that  the gradients of the polynomials~$D_I$ and~$D$ at~$C^{\lambda}$ are  linearly independent. Therefore $C^{\lambda}$ is a regular point of both hypersurfaces~$Y$ and~$W_{\lambda}$, and these hypersurfaces intersect transversely at~$C^{\lambda}$.
\end{proof}

{\sloppy
\begin{lem}
The $4$-tuples\/ $(Y,Z,W,\CC_{\bS^n})$ and\/ $(Y,Z,W,\CC_{\Lambda^n})$ satisfy condition\/~\textnormal{(B)} in Lemma~\ref{lem_tech}.
\end{lem}

\begin{proof}
Let $a_1,\ldots,a_r$ be the standard basis of the group~$\Z_2^r$. Note that the functions $\kappa\colon \CG_{(n)}(\R)\setminus W\to\Z_2^r$ corresponding to the $4$-tuples $(Y,Z,W,\CC_{\bS^n})$ and $(Y,Z,W,\CC_{\Lambda^n})$ are different from each other, since the signs of the equations $h_{\lambda}(C)=0$ giving the hypersurfaces~$W_{\lambda}$  are chosen in the different ways. Nevertheless, both cases are treated in the same way. So let $\X^n$ be either~$\Lambda^n$ or~$\bS^n$, and let $\kappa$ be the function corresponding to the $4$-tuple $(Y,Z,W,\CC_{\X^n})$. Denote by~$\mathcal{K}$ the subgroup of~$\Z_2^r$ generated by all $\kappa(C)$ such that $C\in Y^{reg}(\R)\setminus(Z\cup W)$. Our goal is to prove that $\mathcal{K}=\Z_2^r$.

For every $\lambda=1,\ldots,r,$  let $C^{\lambda}$ be the matrix in Lemma~\ref{lem_Cs}. Then $Y^{reg}(\R)$ and $W_{\lambda}^{reg}(\R)$ intersect transversely at $C^{\lambda}$, and no other irreducible component of~$X$ contains~$C^{\lambda}$. Let $C^{\lambda}_+$ and $C^{\lambda}_-$ be two points in $Y^{reg}(\R)$ close to~$C^{\lambda}$ and lying on the different sides of~$W_{\lambda}(\R)$. Then $\kappa_{\lambda}\bigl(C^{\lambda}_+\bigr)=\kappa_{\lambda}\bigl(C^{\lambda}_-\bigr)+1$ and $\kappa_{\mu}\bigl(C^{\lambda}_+\bigr)=\kappa_{\mu}\bigl(C^{\lambda}_-\bigr)$ for all $\mu\ne\lambda$. Hence $\kappa\bigl(C^{\lambda}_+\bigr)= \kappa\bigl(C^{\lambda}_-\bigr)+a_{\lambda}$. Therefore $a_{\lambda}\in\mathcal{K}$. Since this holds true for all~$\lambda$, we see that $\mathcal{K}=\Z_2^r$.
\end{proof}

}

\section{Proofs of Theorems~\ref{theorem_key} and~\ref{theorem_key2}}\label{section_proof_key}

We shall prove Theorems~\ref{theorem_key} and~\ref{theorem_key2} by induction on the dimension~$n$. In both cases the induction will be with step~$2$, that is, we shall deduce the assertion of Theorem~\ref{theorem_key} (respectively, Theorem~\ref{theorem_key2}) for $n$ from the assertion of this theorem for $n-2$. So we need to start by proving the basis  of induction for $n=1$ and $n=2$.

\textsl{Basis of induction: $n=1$.\/} In this case we have only one variable $c=c_{01}$. The volume of a one-dimensional simplex is just the length of the segment, hence, $\tV_{\Lambda^1}=\Arcosh(c)$ and $\tV_{\bS^1}=\Arccos(c)$. We have,  $\CC_{\Lambda^1}=(1,+\infty)$ and $\CC_{\bS^1}=(-1,1)$. The assertion of Theorem~\ref{theorem_key2} holds true, since all branches of the multi-valued function $\Arccos(c)$ are real on~$(-1,1)$. The assertion of Theorem~\ref{theorem_key} holds true, since every branch of the multi-valued function $\Arcosh(c)$ on $(1,+\infty)$ has the form $\pm\arcosh c+2\pi i k$, $k\in\Z$, and the sign~$\pm$ is given by the  evenness of the linking number of the loop along which the analytic continuation was taken with the pair of points~$\CH=\{-1,1\}$.

\textsl{Basis of induction: $n=2$.\/}
Combining~\eqref{eq_sin-alpha} and~\eqref{eq_area}, we obtain that
\begin{multline*}
V_{\X^2}(C)=\varepsilon\arccos\left(\frac{\varepsilon(c_{12}-c_{01}c_{02})}{\sqrt{(1-c_{01}^2)(1-c_{02}^2)}}\right)+\varepsilon\arccos\left(\frac{\varepsilon(c_{02}-c_{01}c_{12})}{\sqrt{(1-c_{01}^2)(1-c_{12}^2)}}\right)\\{}+\varepsilon\arccos\left(\frac{\varepsilon(c_{01}-c_{02}c_{12})}{\sqrt{(1-c_{02}^2)(1-c_{12}^2)}}\right)-\varepsilon\pi.
\end{multline*}
The arguments of the arccosines in the right-hand side of this formula belong to $(-1,1)$ whenever $C\in\CC_{\X^2}$, since they are the cosines of the angles of a triangle. After the analytic continuation along a loop, the square roots in the denominators may change their signs. Nevertheless, the arguments of the arccosines will still belong to $(-1,1)$. Since any branch of the arccosine takes real values on $(-1,1)$, we obtain that the multi-valued functions~$\tV_{\bS^2}(C)$ and~$\tV_{\Lambda^2}(C)$ are totally real on the sets~$\CC_{\bS^2}$ and~$\CC_{\Lambda^2}$, respectively.

\textsl{Induction step for Theorem~\ref{theorem_key2}.}  Assume that $\X^n$ is either a sphere or an even-dimensional Lobachevsky space, and $n\ge 3$. Assume that the assertion of Theorem~\ref{theorem_key2} holds true for $\X^{n-2}$. Let us prove the assertion of Theorem~\ref{theorem_key2} for~$\X^n$.

We would like to apply Lemma~\ref{lem_tech} to the $1$-form~$d\tV_{\X^n}$. This $1$-form is closed, that is, satisfy condition~(ii). Proposition~\ref{propos_fh} and Lemma~\ref{lem-iv-a} imply that~$d\tV_{\X^n}$ satisfies conditions~(iii) and~(iv). In Section~\ref{section_AB} we have proved that $X$ is essentially real and the $4$-tuple $(Y,Z,W,\CC_{\X^n})$ satisfies conditions~(A) and~(B). By the inductive assumption, any branch of the function $\tV_{\X^{n-2}}$ is real on~$\CC_{\X^{n-2}}$. Hence any branch of any of the functions $\tV_I=\tpr^*_I(\tV_{\X^{n-2}})$, where $|I|=n-1$, is real  on~$\CC_{\X^{n}}$. On the other hand, formula~\eqref{eq_d_talpha}  implies that both branches of~$d\talpha_I$ on~$\CC_{\X^{n}}$ are real. Therefore, by~\eqref{eq_dCVn}, the $1$-form~$d\tV_{\X^n}$ is totally real on~$\CC_{\X^{n}}$, that is, satisfies condition~(i).

If $\varphi$ is a function holomorphic at~$C^*$ and $\beta$ is a path  starting at~$C^*$ such that $\varphi$ admits the  analytic continuation along~$\beta$, then we  denote by~$\varphi^{(\beta)}$ the holomorphic function in a neighborhood of the end of~$\beta$ obtained by this analytic continuation.

Now, let us prove that the $1$-form~$d\tV_{\X^n}$ satisfies condition~(v). We need to study the behavior of the multi-valued $1$-form~$d\tV_{\X^n}$ in a neighborhood of a point $C^0$ such that $C^0$ is a regular point of~$\CH_J$ for some subset $J\subset\{0,\ldots,n\}$, $|J|=n$, and $C^0$ does not lie in the union of all other hypersurfaces~$\CH_K$, $K\ne J$. Without loss of generality, we may assume that $J=\{0,\ldots,n-1\}$. Let $\beta$ be a path from~$C^*$ to~$C^0$ lying in~$\FX$ except for its endpoint~$C^0$.  Since $C^0\notin W$, the functions~$V_I^{(\beta)}$  are holomorphic at~$C^0$ for all~$I$ such that $|I|=n-1$. Further, formula~\eqref{eq_d_talpha} implies that, if  $I\not\subset J$, then the $1$-form~$d\alpha_I^{(\beta)}$ is holomorphic at~$C^0$, and if $I\subset J$, then the $1$-form~$d\alpha_I^{(\beta)}$ is meromorphic at~$C^0$, and 
\begin{equation}\label{eq_alpha_gamma}
d\alpha^{(\beta)}_I=\frac{\varepsilon^{n-1}D_{I\cup\{n\},J}}{2R^{(\beta)}R^{(\beta)}_I}\cdot\frac{dD_J}{D_J}+\eta_{I,\beta},
\end{equation}
where $\eta_{I,\beta}$ is a $1$-form holomorphic at~$C^0$. The coefficient at~$dD_J/D_J$ in~\eqref{eq_alpha_gamma} is holomorphic at~$C^0$. Hence it follows from~\eqref{eq_dCVn} that the $1$-form~$dV_{\X^n}^{(\beta)}$ is meromorphic at~$C^0$, and 
\begin{equation}\label{eq_d_V_gamma}
dV_{\X^n}^{(\beta)}=Q_{\beta}\cdot\frac{dD_J}{D_J}+\eta_{\beta},
\end{equation}
where $Q_{\beta}$ and $\eta_{\beta}$ are a function and a $1$-form holomorphic at~$C^0$, respectively, and
\begin{equation}\label{eq_Q_beta}
Q_{\beta}(C^0)=\frac{\varepsilon^{n-1}}{2}\sum_{I\subset J,\,|I|=n-1}\frac{D_{I\cup\{n\},J}(C^0)}{R^{(\beta)}(C^0)R^{(\beta)}_I(C^0)}\cdot V_I^{(\beta)}(C^0).
\end{equation}
Since $D_J(C^0)=0$, Jacobi's identity~\eqref{eq_Jac} implies that
\begin{equation*}
\frac{D_{I\cup\{n\},J}(C^0)}{R^{(\beta)}(C^0)R^{(\beta)}_I(C^0)}=\pm i
\end{equation*}
for all $I\subset J$ such that $|I|=n-1$. Hence, 
\begin{equation}\label{eq_Q_beta_new}
Q_{\beta}(C^0)=\frac{i}{2}\sum_{I\subset J,\,|I|=n-1}\pm V_I^{(\beta)}(C^0)
\end{equation}
for an appropriate choice of signs~$\pm$.

Since the $1$-form $dV_{\X^n}^{(\beta)}$ is closed, we obtain that the function~$Q_{\beta}(C)$ is constant on~$\CH_J$ in a neighborhood of~$C^0$. Hence $$Q_{\beta}(C)=Q_{\beta}(C^0)+D_J(C)F_{\beta}(C),$$ where the function~$F_{\beta}$ is holomorphic at~$C^0$. Hence formula~\eqref{eq_d_V_gamma} can be rewritten in the form
\begin{equation}\label{eq_d_V_gamma2}
dV_{\X^n}^{(\beta)}=Q_{\beta}(C^0)\cdot\frac{dD_J}{D_J}+\omega_{\beta},
\end{equation}
where $\omega_{\beta}$ is a $1$-form holomorphic at~$C^0$. 

Our aim is to prove that the number~$Q_{\beta}(C^0)$ is purely imaginary for any~$C^0$ and any~$\beta$. The integral of the $1$-form $dV_{\X^n}^{(\beta)}$ along a small circuit around~$\CH_J$ is equal to $2\pi i Q_{\beta}(C^0)$. Hence the number $Q_{\beta}(C^0)$ does not change under   continuous deformations of the pair~$(C^0,\beta)$ such that $C^0$ remains a regular point of~$\CH_J$ lying on none of~$\CH_K$, $K\ne J$, and $\beta$ remains a path that leads to~$C^0$ and is contained in~$\FX$, except for its endpoint. Since $n\ge 3$, the hypersurface~$\CH_J$ is irreducible. Hence the set $\CH_J^{reg}\setminus\bigcup_{K\ne J}\CH_K$  is connected. Besides, if the point~$C^0$ moves continuously in this set, we can always change continuously the path~$\beta$ leading to~$C^0$. Therefore it is sufficient to prove that $Q_{\beta}(C^0)$ is purely imaginary for some particular point~$C^0$ and for all paths~$\beta$ leading to it. We choose the point~$C^0$ in the following way. Consider a regular $(n-2)$-dimensional simplex $[u_1\ldots u_{n-1}]$ with edge~$1$ in~$\X^{n-2}$. Let~$u_0$ be the centre of this simplex. Let $\uC$ be the Gram matrix of the points $u_0,\ldots,u_{n-1}$ (in the standard vector model of~$\X^{n-2}$). Then $\uC\in\partial\CC_{\X^{n-1}}\cap\CH^{reg}_{(n-1)}$, and all principal minors of~$\uC$ are non-zero, see the proof of Lemma~\ref{lem_A}. Let $L$ be the space of all matrices $C\in\CG_{(n)}(\R)$ whose upper-left submatrix~$C_J$ of size $n\times n$ coincides with~$\uC$. Then $L\subset\CH_J^{reg}(\R)$. The matrix entries $c_{0n},c_{1n},\ldots,c_{n-1,n}$ can be taken for coordinates on~$L$. For each~$K$ such that $n\in K$, the principal minor $D_K(C)$ restricted to~$L$ is a polynomial $P_K(c_{0n},\ldots,c_{n-1,n})$. Since all proper principal minors of~$\uC$ are non-zero, we easily see that none of these polynomials~$P_K$  is identically zero. Hence there exist real numbers $c_{0n}^0,\ldots,c_{n-1,n}^0$ such that $P_I(c_{0n}^0,\ldots,c_{n-1,n}^0)\ne 0$ for all~$K$ such that $n\in K$. Then all proper principal minors of the obtained matrix~$C^0$ are non-zero. Therefore, $C^0\in\CH^{reg}_J(\R)\setminus\bigcup_{K\ne J}\CH_K$.

Now, let us show that $Q_{\beta}(C^0)\in i\R$ for any path~$\beta$ from~$C^*$ to~$C^0$ that is contained in~$\FX$ except for its endpoint~$C^0$. For each subset $I\subset J$ such that $|I|=n-1$, the matrix~$C^0_I$ is the Gram matrix of vertices of a non-degenerate simplex, hence,  $C_I^0\in \CC_{\X^{n-2}}$. Then the assertion of Theorem~\ref{theorem_key2} for~$\X^{n-2}$ yields that  $V^{(\beta)}_I(C^0)\in\R$. Now, it follows from~\eqref{eq_Q_beta_new} that $Q_{\beta}(C^0)\in i\R$. 

Thus, we have proved that the $1$-form $d\tV_{\X^n}$ satisfies all conditions of Lemma~\ref{lem_tech}. Therefore, the function $\CI(d\tV_{\X^n})=\tV_{\X^n}-V_{\X^n}^*$ is totally real on~$\CC_{\X^n}$. (Recall that $V_{\X^n}^*=V_{\X^n}(C^*)$.) Since $V_{\X^n}^*\in\R$, we conclude that the function $\tV_{\X^n}$ is totally real on~$\CC_{\X^n}$, which completes the proof of Theorem~\ref{theorem_key2}.

\textsl{Induction step for Theorem~\ref{theorem_key}.}  
Assume that $n$ odd and $n\ge 3$. Assume that the assertion of Theorem~\ref{theorem_key} holds true for~$\Lambda^{n-2}$, and prove the assertion of Theorem~\ref{theorem_key} for~$\Lambda^n$. 
The assertion of Theorem~\ref{theorem_key} can be reformulated as follows:

\textit{For any element $\gamma\in\pi=\pi_1(\FX,C^*),$ the function 
\begin{equation}\label{eq_tW_gamma}
\tW_{\gamma}=i\left(M_{\gamma}\tV_{\Lambda^n}-(-1)^{\lk(\gamma,\CH)}\tV_{\Lambda^n}\right)
\end{equation}
is totally real on~$\CC_{\Lambda^n}$.}

We would like to apply Lemma~\ref{lem_tech2} to the $1$-form $d\tV_{\Lambda^n}$. Conditions~(A) and~(B) have been checked in Section~\ref{section_AB}, condition~(ii) is obvious, and conditions~(iii) and~(iv) are satisfied by Proposition~\ref{propos_fh} and Lemma~\ref{lem-iv-a}, respectively. The principal branch~$dV_{\Lambda^n}$ of the $1$-form~$d\tV_{\Lambda^n}$ is well defined and real on~$U$, hence, condition~(i${}'$) is also satisfied.

Condition~(i${}''$) for~$d\tV_{\Lambda^n}$ says that all $1$-forms~$d\tW_{\gamma}$ are totally real on~$\CC_{\Lambda^n}$. Let us prove this assertion. Obviously, $M_{\gamma}\tR=(-1)^{\lk(\gamma,\CH)}\tR$ and
$M_{\gamma}\tR_{I}=(-1)^{\lk(\gamma,\CH_{I})}\tR_{I}$. By~\eqref{eq_d_talpha}, we obtain that \begin{equation}\label{eq_mgdta}
M_{\gamma}d\talpha_{I}=(-1)^{\lk(\gamma,\CH)+\lk(\gamma,\CH_{I})}d\talpha_{I}.
\end{equation}
The linking number of the loop $\gamma$ and the hypersurface $\CH_I=\CH_{I,(n)}$ in~$\CG_{(n)}(\C)$ is equal to the linking number of the loop~$\pr_I(\gamma)$ and the hypersurface $\CH_{(n-2)}$ in~$\CG_{(n-2)}(\C)$.
Hence,  combining~\eqref{eq_dCVn}, \eqref{eq_tW_gamma}, and \eqref{eq_mgdta}, we obtain that
\begin{equation}\label{eq_d_Phi_gamma}
d\tW_{\gamma,(n)}=\frac{(-1)^{\lk(\gamma,\CH)+1}}{n-1}\mathop{\sum_{I\subset\{0,\ldots,n\},}}\limits_{|I|=n-1}(-1)^{\lk(\gamma,\CH_I)}\tpr^*_I\left(\tW_{\pr_I(\gamma),(n-2)}\right)d\talpha_{I,(n)}\,.
\end{equation}
By the inductive assumption, the functions $\tW_{\pr_I(\gamma),(n-2)}$ are totally real on~$\CC_{\Lambda^{n-2}}$. Hence the functions $\tpr^*_I\left(\tW_{\pr_I(\gamma),(n-2)}\right)$ are totally real on~$\CC_{\Lambda^{n}}$. On the other hand, it follows from~\eqref{eq_d_talpha}  that the $1$-forms~$d\talpha_{I,(n)}$ are totally real on~$\CC_{\Lambda^{n}}$. Therefore, by~\eqref{eq_d_Phi_gamma}, the $1$-form~$d\tW_{\gamma,(n)}$ is totally real on~$\CC_{\Lambda^{n}}$.

Let us prove that the $1$-form $d\tV_{\Lambda^n}$ satisfies condition~(v${}'$). As in the proof of Theorem~\ref{theorem_key2}, let $C^0$ be a regular point of~$\CH_J$, where~$|J|=n$, such that $C^0$ does not belong to the union of all other hypersurfaces~$\CH_K$, and let~$\beta$ be a path from~$C^*$ to~$C^0$ that is contained in~$\FX$, except for its endpoint~$C^0$. Again, we may assume that $J=\{0,\ldots,n-1\}$. Formulae~\eqref{eq_Q_beta_new} and~\eqref{eq_d_V_gamma2} are obtained literally in the same way as in the proof of Theorem~\ref{theorem_key2}. We need to prove that the numbers~$Q_{\beta}(C^0)$ are real for all~$C^0$ and~$\beta$. As in the proof of Theorem~\ref{theorem_key2}, we see that it is sufficient to prove this assertion for some particular point~$C^0$ and for all paths~$\beta$ leading to it. We take the same matrix~$C^0$ as in the proof of Theorem~\ref{theorem_key2}. Recall that $C^0_J$ is the Gram matrix of the points $u_0,\ldots,u_{n-1}$ in $\Lambda^{n-2}$ such that $u_1,\ldots,u_{n-1}$ are vertices of a regular $(n-2)$-dimensional simplex, and  $u_0$ is the centre of this simplex. 
The matrices~$C_I^0$ lie in~$\CC_{\Lambda^{n-2}}$ for all $I\subset J$ such that $|I|=n-1$, that is, for all $I=J\setminus\{j\}$, $j=0,\ldots,n-1$. The assertion of Theorem~\ref{theorem_key} for~$\Lambda^{n-2}$ implies that    
$$\Re V^{(\beta)}_{I}(C^0)=\pm V_{\Lambda^{n-2}}(C^0_I).$$ 
Substituting this to~\eqref{eq_Q_beta_new}, we obtain that
\begin{equation}\label{eq_ImQ}
\Im Q_{\beta}(C^0)=\frac{1}{2}\sum_{j=0}^{n-1}s_j V_{\Lambda^{n-2}}(C^0_{J\setminus\{j\}})
\end{equation}
for some $s_0,\ldots,s_{n-1}\in\{-1,1\}$. 

Let us vary continuously the pair~$(C^0,\beta)$ so that~$C^0$ remains in~$\CH_J^{reg}(\R)\setminus\bigcup_{K\ne J}\CH_K$. Then the submatrix $C^0_J$ remains in~$\partial\CC_{\Lambda^{n-1}}$. Obviously, the coefficients~$s_j$ do not change under such continuous deformations. On the other hand, as it was mentioned above, the value~$Q_{\beta}(C^0)$ also does not change under such deformations. Hence, the linear combination
\begin{equation}\label{eq_summa}
\sum_{j=0}^{n-1}s_j V_{\Lambda^{n-2}}(C_{J\setminus\{j\}})
\end{equation}
is constant as~$C$ runs over a neighborhood of~$C^0$ in~$\CH_J(\R)$. Now, notice that the sum~\eqref{eq_summa} depends only on the submatrix~$C_J$ rather than on the whole matrix~$C$. Besides, any sufficiently small deformation of~$C_J^0$ such that~$C_J^0$ remains real and degenerate can be extended to a deformation of the pair~$(C^0,\beta)$ such that $C^0$ remains in~$\CH_J^{reg}(\R)\setminus\bigcup_{K\ne J}\CH_K$. Indeed, the deformation of~$C^0$ can be chosen so that the matrix entries $c^0_{0n},\ldots,c^0_{n-1,n}$ remain constant, and the deformation of~$\beta$ can be easily built using the smoothness of the hypersurface~$\CH_J$ in a neighborhood of~$C^0$. Therefore, the linear combination
\begin{equation}\label{eq_summa2}
\sum_{j=0}^{n-1}s_j V(v_0,\ldots,v_{j-1},v_{j+1},\ldots,v_{n-1})
\end{equation}
is constant as the points $v_0,\ldots,v_{n-1}\in\Lambda^{n-2}$ vary in sufficiently small neighborhoods of the points  $u_0,\ldots,u_{n-1}$, respectively. Here by $V(w_1,\ldots,w_{n-1})$ we denote the $(n-2)$-dimensional volume of the simplex in~$\Lambda^{n-2}$ with vertices $w_1,\ldots,w_{n-1}$. 
We put $V_j=V(v_0,\ldots,v_{j-1},v_{j+1},\ldots,v_{n-1})$, $j=0,\ldots,n-1$.

{\sloppy
\begin{lem}\label{lem_lin_comb}
A unique \textnormal{(}up to multiplication by a number\textnormal{)} linear combination of the volumes $V_0,\ldots,V_{n-1}$ remaining constant as the points $v_0,\ldots,v_{n-1}$ vary in sufficiently small neighborhoods of the points $u_0,\ldots,u_{n-1}$, respectively, is the linear combination
\begin{equation}\label{eq_lin_comb2}
V_0-\sum_{j=1}^{n-1}V_j=0.
\end{equation}
\end{lem}

}

\begin{proof}
The equality~\eqref{eq_lin_comb2} holds true, since the simplex $[v_1\ldots v_{n-1}]$ with volume~$V_0$ is decomposed into $n$ simplices $[v_0\ldots v_{j-1}v_{j+1}\ldots v_{n-1}]$, $j=1,\ldots,n-1$, with volumes $V_1,\ldots,V_{n-1}$, respectively.

Now, let $\sum_{j=0}^{n-1}\lambda_jV_j$ be a constant linear combination. Take $v_j=u_j$ for  $j=1,\ldots,n-1$, and allow to vary the point~$v_0$ only. Then the volume~$V_0$ is constant, and the volumes $V_1,\ldots, V_{n-1}$ become functions in~$v_0$. It is clear from the symmetry reasons that the gradients of these functions at the point $v_0=u_0$ are tangent vectors  $\xi_1,\ldots,\xi_{n-1}$ to the segments $[u_0u_1],\ldots,[u_0u_{n-1}]$, respectively, and the lengths of all these vectors are equal to each other. Then   $\xi_1+\cdots+\xi_{n-1}$ is the only zero linear combination of these vectors. Hence, all coefficients $\lambda_1,\ldots,\lambda_{n-1}$ must be equal to each other. Then equality~ \eqref{eq_lin_comb2} yields \mbox{$\sum_{j=0}^{n-1}\lambda_jV_j=(\lambda_0+\lambda_1)V_0$.} However, the volume~$V_0$ is obviously non-constant if we allow to vary all points~$v_j$. Therefore, we obtain that $\lambda_0+\lambda_1=0$.
\end{proof}

Lemma~\ref{lem_lin_comb} implies that the linear combination~\eqref{eq_summa2} is proportional to the linear combination~\eqref{eq_lin_comb2}, hence, is equal to zero. Therefore, the right-hand side of~\eqref{eq_ImQ} is equal to zero. Hence, the $1$-form~$d\tV_{\Lambda^n}$ satisfies condition~(v${}'$).

The principal branch~$V_{\Lambda^n}(C)$ of the multi-valued function~$\tV_{\Lambda^n}$ on~$\CC_{\Lambda^n}$ is well defined and satisfies zero boundary conditions, since the volume of any degenerate simplex vanishes. Thus, Lemma~\ref{lem_tech2} can be applied to the $1$-form~$d\tV_{\Lambda^n}$ and the function~$\tV_{\Lambda^n}$. This lemma implies that the functions~$\tW_{\gamma}$ are totally real on~$\CC_{\Lambda^n}$ for all $\gamma\in\pi$, which completes the proof of Theorem~\ref{theorem_key}.

\section{Flexible polyhedra and their volumes}\label{section_def_res}

In this section we give a rigorous definition of a flexible polyhedron, and reformulate Theorem~\ref{theorem_main} more precisely. First of all, it is standard for the  theory of flexible polyhedra to restrict ourselves to considering only simplicial polyhedra. Indeed, an arbitrary polyhedral surface has a simplicial subdivision. Passing to this subdivision, we introduce new hinges. Hence all flexions of the initial polyhedral surface induce flexions of the obtained simplicial polyhedral surface, and, possibly, some new flexions appear. Therefore Theorem~\ref{theorem_main} for all flexible polyhedra follows immediately from Theorem~\ref{theorem_main} for simplicial flexible polyhedra. Below, we shall always work with simplicial polyhedra only.

Recall that a $k$-dimensional \textit{pseudo-manifold\/} is a finite simplicial complex~$K$ such that
\begin{enumerate}
\item every simplex of~$K$ is contained in a $k$-dimensional simplex,
\item every $(k-1)$-dimensional  simplex of~$K$ is contained in exactly two $k$-dimensional simplices,
\item $K$ is strongly connected, i.\,e., the complement~$K\setminus\mathrm{Sk}^{k-2}(K)$ is connected, where $\mathrm{Sk}^{k-2}(K)$ is the $(k-2)$-skeleton of~$K$.
\end{enumerate}
A pseudo-manifold~$K$ is said to be \textit{oriented\/} if its $k$-dimensional simplices are endowed with orientations such that, for any two $k$-dimensional simplices~$\sigma_1$ and~$\sigma_2$ with a common $(k-1)$-dimensional face, the orientations  of  this face induced by the orientations of~$\sigma_1$ and~$\sigma_2$ are opposite to each other.

As before, we shall identify the Lobachevsky space $\Lambda^n$ with its standard vector model in~$\R^{1,n}$.
Let $\Delta^k$ be an affine simplex with vertices $v_0,\ldots,v_k$. A mapping $ P\colon\Delta^k\to\Lambda^n$ will be  called \textit{pseudo-linear\/} if 
$$
 P(\beta_0v_0+\cdots+\beta_kv_k)=\frac{\beta_0 P(v_0)+\cdots+\beta_k P(v_k)}{
|\beta_0 P(v_0)+\cdots+\beta_k P(v_k)|}
$$
for any non-negative numbers $\beta_0,\ldots,\beta_k$ such that $\beta_0+\cdots+\beta_k=1$, where $|\bx|=\sqrt{\langle\bx,\bx\rangle}$. It is easy to see that the restriction of a pseudo-linear mapping to a face of~$\Delta^k$ is also pseudo-linear.

In the sequel, we shall assume that the dimension $n$ is greater than or equal to~$2$.

\begin{defin}\label{defin_simplic}
Let $K$ be an oriented $(n-1)$-dimensional pseudo-manifold. A \textit{bounded \textnormal{(}simplicial\textnormal{)}  polyhedron\/} of combinatorial type\/~$K$ in the Lobachevsky space~$\Lambda^n$ is a mapping $ P\colon K\to\Lambda^n$ such that the restriction of~$ P$ to every simplex of~$K$ is pseudo-linear. A \textit{flexion\/} of a bounded  polyhedron of combinatorial type\/~$K$ is a continuous family of polyhedra $P_t\colon K\to\Lambda^n$, where $t$ runs over some interval~$(a,b)$, such that the lengths of all edges of $P_t$ are constant as $t$ varies.  A flexion~$P_t$ is called \textit{non-trivial\/} if the mappings $P_{t_1}$ and $P_{t_2}$ cannot be obtained from each other by an isometry of~$\Lambda^n$ for any  sufficiently close to each other $t_1\ne t_2$. 
\end{defin}

Notice that, for bounded simplicial  polyhedra, the requirement that the lengths of edges are constant during the flexion immediately implies that all faces of the polyhedron remain congruent to themselves during the flexion.

\begin{remark}\label{remark_variant}
Definition~\ref{defin_simplic} is a natural analogue of the standard definition of a flexible polyhedron in the Euclidean space used in~\cite{Sab96}--\cite{Sab98b}, \cite{CSW97}, \cite{Gai11}, \cite{Gai12}. The only difference is that in  the Euclidean case we use affine mappings instead of pseudo-linear. Actually, the usage of pseudo-linear mappings in Definition~\ref{defin_simplic} is convenient but not obligatory. Let us refuse from the requirement that the restriction of~$P$ to every simplex of~$K$ is pseudo-linear, and require only that the restriction of~$P$ to every simplex $\Delta$ of~$K$ is a continuous mapping onto  the convex hull of the points $ P(v)$, where $v$ runs over all vertices of~$\Delta$. Besides, let us not distinguish between polyhedra of the same combinatorial type such that all pairs of their corresponding vertices coincide. This definition is equivalent to Definition~\ref{defin_simplic}, and was used in~\cite{Gai15}.
\end{remark}

Let $\bigl[v_1^{(s)}\ldots v_n^{(s)}\bigr]$, $s=1,\ldots,q$, be all positively oriented $(n-1)$-dimensional simplices of the pseudo-manifold~$K$.  By definition, the \textit{generalized oriented volume\/} of the polyhedron $ P\colon K\to\Lambda^n$ is given by
\begin{equation}\label{eq_gen_vol}
\CV_K( P)=\sum_{s=1}^qV_{or}\bigl(\bo, P\bigl(v_1^{(s)}\bigr),\ldots, P\bigl(v_n^{(s)}\bigr)\bigr),
\end{equation}
where $\bo$ is an arbitrary point in~$\Lambda^n$ and $V_{or}(v_0,\ldots,v_n)$ denotes the oriented volume of the simplex in~$\Lambda^n$ with vertices $v_0,\ldots,v_n$ with the orientation given by the indicated order of vertices. It is easy to see that $\CV_K( P)$ is independent of the choice of the point~$\bo$. 

An equivalent definition of the generalized oriented volume can be given in the following way. For each point $x$ in the complement of the surface $ P(K)$, let $\lambda_P(x)$ be the algebraic intersection number of a path from~$x$ to infinity and  the $(n-1)$-dimensional cycle~$P(K)$. It is easy to show that this intersection number is independent of the choice of the path. The function $\lambda_P(x)$ is called the \textit{indicator function\/} of the polyhedron~$P$.
Then 
\begin{equation}\label{eq_gen_vol2}
\CV_K( P)=\int_{\Lambda^n}\lambda_P(x)\,d V(x),
\end{equation}
where $dV(x)$ is the standard volume measure in~$\Lambda^n$.

If $ P\colon K\to \Lambda^n$ is an embedding, then $\CV_K( P)$ is the usual oriented volume of the region bounded by the polyhedral hypersurface $ P(K)$. 

The following theorem is a more precise formulation of Theorem~\ref{theorem_main}, which is the main result of this paper.

\begin{theorem}\label{theorem_main2}
Suppose that $n$ is odd. Let $K$ be an oriented $(n-1)$-dimensional pseudo-manifold, and let $P_t\colon K\to\Lambda^n$, $t\in(a,b)$, be a flexion of a bounded polyhedron of combinatorial type~$K$. Then the generalized oriented volume $\CV_K(P_t)$ is constant.
\end{theorem}

Definition~\ref{defin_simplic} allows flexible polyhedra to be self-intersecting and degenerate. For instance, one can easily construct an example of a flexible polyhedron in sense of this definition that geometrically is the union of the boundaries of two tetrahedra with a common edge, see~\cite[Sect.~2]{Gai15} for details. Certainly, these two tetrahedra can rotate around their common edge independently of each other. When we consider the question on existence of flexible polyhedra, we would like to avoid such pathological examples. So some additional conditions should be added to the definition of a flexible polyhedron. For instance,  in~\cite{Gai15} the following two conditions are added:
\begin{enumerate}
\item The image of any simplex of~$K$ under the mapping~$P$ is a non-degenerate simplex of the same dimension in~$\Lambda^n$.
\item The simplicial complex~$K$ cannot be decomposed into the union of two its subcomplexes~$K_1$ and~$K_2$ such that $\dim K_1= \dim K_2=n-1$ and the set $P(K_1\cap K_2)$ is contained in an $(n-2)$-dimensional plane in~$\Lambda^n$. 
\end{enumerate}
Polyhedra satisfying~(1) are called \textit{polyhedra with non-degenerate faces,} and polyhedra satisfying both~(1) and~(2) are called \textit{non-degenerate polyhedra.}  In~\cite{Gai13} the author constructed examples of non-degenerate flexible polyhedra in the spaces~$\E^n$, $\bS^n$, and~$\Lambda^n$ of all dimensions. This result makes substantial the assertion of Theorem~\ref{theorem_main2}.

First, we shall prove Theorem~\ref{theorem_main2} for polyhedra with non-degenerate faces. From the geometric viewpoint only this case is sensible. However, for the completeness of our result, we then shall show that Theorem~\ref{theorem_main2} for all flexible polyhedra in sense of Definition~\ref{defin_simplic} follows from Theorem~\ref{theorem_main2} for polyhedra with non-degenerate faces.

\begin{remark}
Definition~\ref{defin_simplic} can be easily modified so as to obtain a rigorous definition of a not necessarily simplicial flexible polyhedron. Namely,  in the definition of a pseudo-manifold one should replace a simplicial complex by a cell complex glued out of affine convex polytopes along affine isomorphisms of their faces. In this case no appropriate analogue of a pseudo-linear mapping exists, so one should replace the requirement that the restriction of~$P$ to every cell of~$K$ is pseudo-linear 
with the requirement  that the restriction of~$P$ to every cell $\sigma$ of~$K$ is a continuous mapping onto  the convex hull of the points $ P(v)$, where $v$ runs over all vertices of~$\sigma$, as in Remark~\ref{remark_variant}. The generalized oriented volume of a polyhedron $P\colon K\to\Lambda^n$ is given by the same formula~\eqref{eq_gen_vol2}. However, as it was mentioned in the beginning of this section, Theorem~\ref{theorem_main} for arbitrary flexible polyhedra follows immediately from Theorem~\ref{theorem_main} for simplicial flexible polyhedra. Hence further we consider only  simplicial flexible polyhedra in sense of Definition~\ref{defin_simplic}.
\end{remark}

\section{The configuration space of polyhedra with given  edge lengths}\label{section_config}

Let $K$ be an oriented $(n-1)$-dimensional pseudo-manifold with $m$ vertices and $r$ edges, and let $\V(K)$ and $\E(K)$ be the sets of vertices and edges of~$K$, respectively. A polyhedron $ P\colon K\to \Lambda^n$ is determined solely by the positions of its vertices. 
For each vertex $v$ of~$K$, we denote by $\bx_v$ the vector representing the point~$P(v)$ in the standard vector model of~$\Lambda^{n}$, and we denote by $x_{v,0},\ldots,x_{v,n}$ the coordinates of~$\bx_v$. Consider the space~$\R^{m(n+1)}=(\R^{1,n})^m$ with the coordinates $x_{v,j}$, $v\in\V(K)$, $j=0,\ldots,n$. The point of this space corresponding to a polyhedron $ P\colon K\to \Lambda^n$ will again be denoted by~$ P$. Then the space of all bounded  polyhedra $ P\colon K\to \Lambda^n$ of combinatorial type~$K$ is the subset of~$\R^{m(n+1)}$ given by
\begin{gather}
\langle \bx_v,\bx_v\rangle=x_{v,0}^2-x_{v,1}^2-\cdots-x_{v,n}^2=1,\label{eq_vL}\\
x_{v,0}>0,\label{eq_vL+}
\end{gather}
where $v$ runs over~$\V(K)$.

We fix a set $\bell=(\ell_e)$ of positive real numbers indexed by edges $e\in\E(K)$. 
Then the configuration space~$\Sigma^+(\bell)$ of all polyhedra $ P\colon K\to\Lambda^n$ with the prescribed set of edge lengths~$\bell$ is the subset of~$\R^{m(n+1)}$ given by $m$~equations~\eqref{eq_vL}, $m$~inequalities~\eqref{eq_vL+}, and $r$~equations
\begin{equation}
\langle\bx_u,\bx_v\rangle=x_{u,0}x_{v,0}-x_{u,1}x_{v,1}-\cdots-x_{u,n}x_{v,n}=\cosh\ell_{[uv]},\label{eq_uv}
\end{equation}
where $[uv]$ runs over all edges of~$K$. (Certainly, the configuration space~$\Sigma^+(\bell)$ may be empty.) We also denote by~$\Sigma(\bell)$ the real affine variety in~$\R^{m(n+1)}$ given by the $m+r$ quadratic equations~\eqref{eq_vL} and~\eqref{eq_uv}. Since $K$ is connected, it follows easily from~\eqref{eq_uv} that, for each point $ P=(x_{v,j})\in\Sigma(\bell)$, either all~$x_{v,0}$ are positive or all~$x_{v,0}$ are negative. Hence $\Sigma(\bell)=\Sigma^+(\bell)\cup(-\Sigma^+(\bell))$.

The following proposition is a reformulation of Theorem~\ref{theorem_main2}.

\begin{propos}\label{propos_const0}
Suppose that $n$ is odd. Then the function $\CV_K( P)$ is constant on every connected component of\/~$\Sigma^+(\bell)$.
\end{propos}

Any real affine variety $\Sigma\subset\R^N$ possesses a natural stratification that can be constructed in the following way, see~\cite[Sect.~11(b)]{Whi57}. Let $d=\dim\Sigma$. Then $\Sigma=\Sigma_1\cup\Sigma_2$, where $\Sigma_1$ is the the union of all $d$-dimensional irreducible components of~$\Sigma$, and $\Sigma_2$ is the the union of all irreducible components of~$\Sigma$ of dimensions strictly less than~$d$. Put $M=\Sigma^{reg}_1\setminus\Sigma_2$ and $\Sigma'=\Sigma^{sing}_1\cup\Sigma_2$; then $\Sigma=M\sqcup\Sigma'$. The set $\Sigma'$ is a real affine variety of dimension strictly less than~$d$. All connected components of~$M$ are taken for $d$-dimensional stata. Then the same procedure is recursively applied to~$\Sigma'$. The obtained stratification will be referred to as the \textit{standard stratification\/} of~$\Sigma$. By construction, every stratum~$S$ of this stratification is a connected open (in the analytic topology) subset of the set of regular points of an irreducible real affine variety. This easily implies the following proposition.

\begin{propos}\label{propos_strat}
Let $S$ be a stratum of the standard stratification of a real affine variety~$\Sigma\subset\R^N,$ and let $\Xi$ be the Zariski closure of~$S$ in~$\C^N$. Then $\Xi$ is an irreducible complex affine variety, $\dim_{\C}\Xi=\dim_{\R}S,$ and $S\subset \Xi^{reg}$. 
\end{propos}

The set~$\Sigma^+(\bell)$, which we are interested in, generally is not a real affine variety, but it is the union of several connected components of the real affine variety~$\Sigma(\bell)$. The restriction of the standard stratification of~$\Sigma(\bell)$ to~$\Sigma^+(\bell)$ will be called the \textit{standard stratification\/} of~$\Sigma^+(\bell)$.  

Since the standard stratification of~$\Sigma^+(\bell)$ consists of finitely many strata, Proposition~\ref{propos_const0}  will follow if we show that  the  function $\CV_K( P)$ is constant on every stratum~$S$ of the standard stratification of~$\Sigma^+(\bell)$. 
This will be done in three steps. First, using Schl\"afli's formula, we shall show that the function~$\CV_K( P)$ restricted to~$S$ can be  continued to a multi-valued analytic function~$\tCV_K( P)$ on a dense Zariski open subset $\Upsilon\subset \Xi^{reg}$, where $\Xi$ is the Zariski closure of~$S$ in~$\C^{m(n+1)}$. Second, we shall deduce from Theorem~\ref{theorem_key} that the analytic function~$\tCV_K( P)$ is in fact single-valued. Third, we shall use Liouville's theorem on entire functions  
to show that the function~$\tCV_K( P)$ is constant. This scheme will be realized in Sections~\ref{section_proof_main} and \ref{section_Liouville} for flexible polyhedra with non-degenerate faces. Thus we shall prove Theorem~\ref{theorem_main2} for such flexible polyhedra. In Section~\ref{section_proof_main_arbitrary}, we shall show that Theorem~\ref{theorem_main2} for arbitrary polyhedra follows from Theorem~\ref{theorem_main2} for polyhedra with non-degenerate faces.

\section{Proof of Theorem~\ref{theorem_main2} for polyhedra with non-degenerate faces}\label{section_proof_main}

In this section all polyhedra under consideration are supposed to have non-degenerate faces. Notice that the property of a polyhedron $P\colon K\to\Lambda^n$ to have non-degenerate faces is determined solely by the combinatorial type and the set of edge lengths of the polyhedron. So in this section we consider only those pairs $(K,\bell)$ that polyhedra of combinatorial type~$K$ with the set of edge lengths~$\bell$ have non-degenerate faces.

We define the \textit{oriented dihedral angle\/}~$\alpha_F$ of a polyhedron $ P\colon K\to \Lambda^n$  at an $(n-2)$-dimensional face $F$ of~$K$ in the following way. Let $\sigma_1$ and $\sigma_2$ be the two $(n-1)$-dimensional simplices of~$K$ containing~$F$. Take any point~$\bx$ in~$ P(F)$. Let $\bn_1$ and~$\bn_2$ be the unit vectors in the tangent spaces $T_{\bx} P(\sigma_1)$ and~$T_{\bx} P(\sigma_2)$, respectively, orthogonal to~$T_{\bx} P(F)$ and pointing inside the simplices~$ P(\sigma_1)$ and~$ P(\sigma_2)$, respectively. Let $\bm_1$ and $\bm_2$ be the \textit{outer normal vectors\/} to $ P(\sigma_1)$ and~$ P(\sigma_2)$, respectively, at~$\bx$, i.\,e., the unit vectors in~$T_{\bx}\Lambda^n$ orthogonal to $T_{\bx} P(\sigma_1)$ and to~$T_{\bx} P(\sigma_2)$, respectively, such that the product of the direction of~$\bm_i$ and the positive orientation of~$ P(\sigma_i)$ yields the positive orientation of~$\Lambda^n$. We say that the positive direction of rotation around~$ P(F)$ is from $\bm_1$ to~$\bn_1$, and denote by~$\alpha_F=\alpha_F( P)$ the angle from~$\bn_1$ to~$\bn_2$ in this positive direction. This angle is well defined up to~$2\pi k$, $k\in\Z$. We shall always consider~$\alpha_F$ as an element of~$\R/(2\pi\Z)$. It is easy to see that~$\alpha_F$ is independent of the choice of the point~$\bx$ and does not change if we interchange the simplices~$\sigma_1$ and~$\sigma_2$.

Schl\"afli's formula for the differential of the volume is usually written for convex polytopes in~$\Lambda^n$, see the Introduction.
However, the convexity is unimportant. The following version of Schl\"afli's formula follows immediately from Schl\"afli's formula for a simplex, formula~\eqref{eq_gen_vol}, and the above definition of oriented dihedral angles.

\begin{lem}[Schl\"afli's formula]\label{lem_Schlaefli}
The differential of the generalized oriented volume of a polyhedron $ P\colon K\to\Lambda^n$ with non-degenerate faces is given by
\begin{equation}\label{eq_Sch_or}
d \CV_K( P)=-\frac{1}{n-1}\sum_{F\subset K,\,\dim F=n-2}V_F( P)\,d\alpha_F( P),
\end{equation}
where the sum is taken over all\/ $(n-2)$-dimensional faces~$F$ of\/~$K,$ and\/ $V_F(P)$ is the \textnormal{(}unoriented\textnormal{)} $(n-2)$-volume of the simplex~$ P(F)$.
\end{lem}

Let us restrict equation~\eqref{eq_Sch_or} to~$\Sigma^+(\bell)$. For each~$F$, the volume~$V_F( P)$ depends only on the set~$\bell$ of edge lengths, hence, is constant on~$\Sigma^+(\bell)$; we denote it by~$V_{F,\bell}$.

\begin{lem}
For each $(n-2)$-dimensional face~$F$ of~$K,$ the restriction of the function $\exp(i\alpha_F( P))$ to the set\/~$\Sigma^+(\bell)$ coincides with the restriction of some polynomial~$Q_F( P)$ in the coordinates~$x_{v,j}$.
\end{lem}

\begin{proof} 
Let $\sigma_1$ and $\sigma_2$ be the two $(n-1)$-dimensional simplices of~$K$ containing~$F$. We denote the vertices of~$\sigma_1$ and $\sigma_2$ opposite to~$F$ by~$v_0$ and~$v_1$, respectively, and the vertices of~$F$ by $v_2,\ldots,v_{n}$, where the order of vertices is chosen so that the order $v_1,v_2,\ldots,v_{n}$ yields the positive orientation of~$\sigma_2$. Let $C$ be the Gram matrix of the vectors $\bx_{v_0},\ldots,\bx_{v_n}$. It follows easily from the definition of the oriented dihedral angles that the sign of $\sin\alpha_F(P)$ coincides with the sign of the determinant $\det(\bx_{v_0},\ldots,\bx_{v_n})$. The square of this determinant is equal to $(-1)^nD(C)$. By~\eqref{eq_sin-alpha}, we obtain that 
\begin{align*}
\cos\alpha_F(P)&=\frac{(-1)^{n-1}D_{I',I''}(C)}{\sqrt{D_{I'}(C)D_{I''}(C)}}\,,\\
\sin\alpha_F(P)&=\sqrt{\frac{(-1)^nD_I(C)}{D_{I'}(C)D_{I''}(C)}}\cdot\det(\bx_{v_0},\ldots,\bx_{v_n}),
\end{align*}
where $I=\{2,\ldots,n\}$, $I'=\{0,2,\ldots,n\}$, and $I''=\{1,2,\ldots,n\}$. The minors $D_I(C)$, $D_{I'}(C)$, and~$D_{I''}(C)$ are constant on~$\Sigma^+(\bell)$, while $D_{I',I''}(C)$ and $\det(\bx_{v_0},\ldots,\bx_{v_n})$ are polynomials in the coordinates~$x_{v,j}$. Therefore the restrictions of $\cos\alpha_F(P)$ and $\sin\alpha_F(P)$ to~$\Sigma^+(\bell)$ are polynomials, hence,  the restriction of $\exp(i\alpha_F( P))$ to~$\Sigma^+(\bell)$ is also a polynomial.
\end{proof}

The polynomials~$Q_F(P)$ such that $\exp(i\alpha_F( P))=Q_F(P)$ on~$\Sigma^+(\bell)$ may be not unique. We choose and fix some polynomials satisfying these condition.

Now, let $S$ be a stratum of the standard stratification of~$\Sigma^+(\bell)$.
Since $S$ is connected and $V_F( P)=V_{F,\bell}$ are constants on~$S$, we can integrate equation~\eqref{eq_Sch_or}. We obtain that the following equality holds on~$S$:
\begin{equation}\label{eq_volume_explicit}
\CV_K( P)=\frac{i}{n-1}\sum_{F\subset K,\,\dim F=n-2}V_{F,\bell}\Log Q_F( P)+c_S,
\end{equation}
where $c_S$ is a real constant that depends on~$S$ only. This formula should be understood in the following way. Each summand $V_{F,\bell}\Log Q_F( P)$ may have non-trivial variation along a loop in~$S$. However, the whole sum in the right-hand side of~\eqref{eq_volume_explicit} has trivial variations along all loops in~$S$, since we know that the volume $\CV_K( P)$ is well defined on~$S$. Formula~\eqref{eq_volume_explicit} means that we can choose a branch of the sum in the right-hand side so that the equality will hold. (Choosing another branch, we change the constant~$c_S$ only.)

Let $\Xi$ be the Zariski closure of~$S$ in~$\C^{m(n+1)}$. By Proposition~\ref{propos_strat}, $\Xi$ is irreducible, $\dim_{\C}\Xi=\dim_{\R}S$, and $S\subset\Xi^{reg}$. Hence the analytic continuation of a function defined on~$S$ along a path in~$\Xi$ is unique whenever exists.

The angles~$\alpha_F( P)$ are well defined for all $ P\in S$. Hence the polynomials~$Q_F( P)$ do not take zero values in~$S$.  Nevertheless, the polynomials~$Q_F( P)$ may take zero values in~$\Xi$. Let $\Upsilon\subset \Xi$ be the Zariski open subset consisting of all regular points~$ P$ of~$\Xi$ such that all values $Q_F( P)$ are non-zero. The subset~$\Upsilon$ is non-empty, since it contains~$S$, hence, $\Upsilon$ is dense in~$\Xi$. We immediately obtain the following proposition.

\begin{propos}\label{propos_contin}
The function~$\CV_K( P)$ admits the analytic continuation along every path in~$\Upsilon$. The obtained multi-valued analytic function $\tCV_K( P)$ on~$\Upsilon$ is given by the formula
\begin{equation}\label{eq_volume_explicit2}
\tCV_K( P)=\frac{i}{n-1}\sum_{F\subset K,\,\dim F=n-2}V_{F,\bell}\Log Q_F( P)+c_S.
\end{equation}
\end{propos}

\begin{remark}
Formula~\eqref{eq_volume_explicit2} should be understood in the following way: For any branch of~$\tCV_K( P)$, one can choose  branches of the logarithms in the right-hand side  so that equality~\eqref{eq_volume_explicit2} will hold true. However, it is possible that not any choice of branches of the logarithms in the right-hand side of~\eqref{eq_volume_explicit2} will yield a branch of~$\tCV_K( P)$. This situation is rather common when we deal with multi-valued functions. It can be illustrated by the following simplest example: Certainly, the equality $2\sqrt{z}=\sqrt{z}+\sqrt{z}$ is true for an appropriate choice of branches of the square roots in the right-hand side. However, branches of the square roots in the right-hand side can be chosen so that their sum will equal zero rather than~$2\sqrt{z}$.
\end{remark}

Since all volumes $V_{F,\bell}$ are real constants, Proposition~\ref{propos_contin} immediately implies the following assertion.

\begin{cor}\label{cor_real_const}
For each point $ P_0\in \Upsilon,$ any two branches of the multi-valued analytic function~$\tCV_K( P)$ in a neighborhood of~$ P_0$ in~$\Upsilon$ differ by a real constant.
\end{cor}

Take $\bo=(1,0,\ldots,0)\in\Lambda^n\subset\R^{1,n}$, and consider formula~\eqref{eq_gen_vol}. 
For every $s=1,\ldots,q$, let $C^{(s)}=\bigl(c_{jk}^{(s)}\bigr)$ be the Gram matrix of the vertices of the simplex $\bigl[\bo\, P\bigl(v_1^{(s)}\bigr)\ldots P\bigl(v_n^{(s)}\bigr)\bigr]$. For each  $I\subset \{0,\ldots,n\}$, we put $D_I^{(s)}=D_I(C^{(s)})$. To indicate the dependence of the matrices~$C^{(s)}$ and the minors~$D_I^{(s)}$ on~$ P$, we shall sometimes write $C^{(s)}( P)$ and $D_I^{(s)}( P)$, respectively.

All elements of all matrices $C^{(s)}$ are polynomials in the coordinates~$x_{v,j}$, hence, are regular functions on~$\Xi$. Therefore all minors~$D_I^{(s)}$ are also regular functions on~$\Xi$. Let $\Omega\subset\Xi$ be the Zariski open subset consisting of all regular points~$P$ of~$\Xi$ such that all minors~$D_I^{(s)},$ where $s=1,\ldots,q,$  $I\subset\{0,\ldots,n\},$ are non-zero at~$ P$. 

\begin{lem}\label{lem_nonzero}
The subset $\Omega\cap S$ is dense in~$S$ \textnormal{(}in the analytic topology\textnormal{)}.
\end{lem}

\begin{proof}
 The isometry group~$\mathrm{Isom}(\Lambda^n)$ acts naturally on~$\Sigma^+(\bell)$. This action is compatible with the standard stratification, that is, the image of any stratum under the transformation given by each element of~$\mathrm{Isom}(\Lambda^n)$ is again a stratum. Since the subgroup $\mathrm{Isom}^+(\Lambda^n)$  of orientation-preserving isometries is connected and the number of strata is finite, we obtain that every stratum is invariant under the action of~$\mathrm{Isom}^+(\Lambda^n)$. Let $ P_0$ be a point in~$S$. Then the $(n-1)$-simplices $\bigl[ P_0\bigl(v_1^{(s)}\bigr)\ldots P_0\bigl(v_n^{(s)}\bigr)\bigr]$ are non-degenerate, $s=1,\ldots,q$. Therefore, there exists an isometry $h\in\mathrm{Isom}^+(\Lambda^n)$ arbitrarily close to the identity isometry such that, for $ P=h\circ P_0$, the $n$-simplices $\bigl[\bo\, P\bigl(v_1^{(s)}\bigr)\ldots  P\bigl(v_n^{(s)}\bigr)\bigr]$ are non-degenerate, $s=1,\ldots,q$. Then $ P\in \Omega\cap S$.
\end{proof}

In particular, we see that $\Omega$ is non-empty. By Proposition~\ref{propos_strat}, $\Xi$ is irreducible. Therefore, $\Omega$ is a dense Zariski open subset of~$\Xi$.

 \begin{propos}\label{propos_one_val}
If the dimension $n$ is odd, then the function $\tCV_K( P)$ in Proposition~\ref{propos_contin} is a single-valued holomorphic function on~$\Upsilon$.
\end{propos}
 
\begin{proof}
 
Take any point $ P_0\in \Omega\cap S$. All simplices $\bigl[\bo\, P_0\bigl(v_1^{(s)}\bigr)\ldots  P_0\bigl(v_n^{(s)}\bigr)\bigr]$ are non-deg\-e\-ne\-rate, $s=1,\ldots,q$.  Put $\varepsilon_s=1$ whenever the simplex $\bigl[\bo\, P_0\bigl(v_1^{(s)}\bigr)\ldots  P_0\bigl(v_n^{(s)}\bigr)\bigr]$ is positively oriented, and $\varepsilon_s=-1$ whenever the simplex $\bigl[\bo\, P_0\bigl(v_1^{(s)}\bigr)\ldots  P_0\bigl(v_n^{(s)}\bigr)\bigr]$ is negatively oriented. 
Then, for  $ P$ in a small neighborhood of~$ P_0$ in~$S$, formula~\eqref{eq_gen_vol} takes the form
\begin{equation*}
\CV_K( P)=\sum_{s=1}^q\varepsilon_sV\bigl(\bo, P\bigl(v_1^{(s)}\bigr),\ldots, P\bigl(v_n^{(s)}\bigr)\bigr),
\end{equation*}
where  $V(v_0,\ldots, v_n)$ denotes the \textit{unoriented\/} volume of the simplex with vertices $v_0,\ldots, v_n$. Using the function~$V_{\Lambda^n}(C)$, this can be rewritten as
\begin{equation}\label{eq_decomp}
\CV_K( P)=\sum_{s=1}^q\varepsilon_sV_{\Lambda^n}\bigl(C^{(s)}( P)\bigr).
\end{equation}

By Proposition~\ref{propos_contin}, the function~$\CV_K( P)$ is real analytic on~$S$ and admits the analytic continuation along any path $\gamma$ in~$\Upsilon$ starting at~$P_0$. Since $ P_0$ is a regular point of~$\Xi$ and $\dim_{\C}\Xi=\dim_{\R}S$, we see that the analytic continuation of~$\CV_K( P)$ along~$\gamma$ is unique. To prove that the obtained analytic function~$\tCV_K( P)$ is single-valued, we need to prove that the analytic continuation~$\CV_K^{(\gamma)}( P)$ of~$\CV_K( P)$ along any loop~$\gamma$ in~$\Upsilon$ with endpoints at~$ P_0$ is again the same function~$\CV_K( P)$. Moreover, we need to prove this for loops~$\gamma$ contained in~$\Upsilon\cap\Omega$ only. Indeed, since $\Omega$ is a dense Zariski open subset of~$\Xi$,  any loop in~$\Upsilon$ is homotopic to a loop in~$\Upsilon\cap\Omega$.

Let $\gamma$ be a loop in~$\Upsilon\cap\Omega$ with endpoints at~$ P_0$.  By Corollary~\ref{cor_real_const}, we have $\CV_{K}^{(\gamma)}( P)=\CV_K( P)+c_{\gamma}$ for a real constant~$c_{\gamma}$. Let $\gamma^2$ be the loop  going twice along~$\gamma$. Then  $\CV_{K}^{(\gamma^2)}( P)=\CV_K( P)+2c_{\gamma}$. For each~$s$, consider the path $C^{(s)}(\gamma^2)$. It is a closed path in~$\CG_{(n)}(\C)\setminus X$  that has even linking numbers with all irreducible components of~$X$.  Continuing analytically both sides of~\eqref{eq_decomp} along~$\gamma^2$, we obtain that
 \begin{equation*}
\CV_K( P)+2c_{\gamma}=\sum_{s=1}^q\varepsilon_sV^{\left(C^{(s)}(\gamma^2)\right)}_{\Lambda^n}\bigl(C^{(s)}( P)\bigr)
\end{equation*}
for $ P$ in a neighborhood of~$ P_0$. In particular, this equality holds for $ P= P_0$.
Therefore,
\begin{equation}\label{eq_2c}
2c_{\gamma}=\sum_{s=1}^q\varepsilon_s\left(V_{\Lambda^n}^{\left(C^{(s)}(\gamma^2)\right)}\bigl(C^{(s)}( P_0)\bigr)-V_{\Lambda^n}\bigl(C^{(s)}( P_0)\bigr)\right).
\end{equation}
For each $s$, the matrix $C^{(s)}( P_0)$ belongs to~$\CC_{\Lambda^n}$, since it is the Gram matrix of the vertices of the non-degenerate simplex $\bigl[\bo\, P_0\bigl(v_1^{(s)}\bigr)\ldots  P_0\bigl(v_n^{(s)}\bigr)\bigr]$. Since the linking number of~$C^{(s)}(\gamma^2)$ and the hypersurface~$\CH$ is even, Theorem~\ref{theorem_key} yields that all summands in the right-hand side of~\eqref{eq_2c} are purely imaginary. But $c_{\gamma}\in\R$. Hence $c_{\gamma}=0$, which completes the proof of the proposition.
\end{proof}

We need the following consequence of Liouville's theorem on entire functions. It seems to be standard. However, for the convenience of the reader, we shall give a proof of it in Section~\ref{section_Liouville}. 

\begin{lem}\label{lem_log}
Let $A$ be a smooth irreducible complex affine variety, and let $\varphi$ be a holomorphic function on~$A$. Assume that there exist regular functions $f_1,\ldots,f_N\in\C[A]$ such that 
\begin{equation}\label{eq_estimate}
\Im \varphi(x)\le \max_{n=1,\ldots,N} \log|f_n(x)|
\end{equation}
for all $x\in A$.  Then $\varphi$ is a constant. \textnormal{(}Here we use the convention $\log 0=-\infty$.\textnormal{)}
\end{lem}

\begin{propos}\label{propos_const1}
If the dimension~$n$ is odd, then the function~$\tCV_K(P)$ is constant on~$\Upsilon$. Hence the function~$\CV_K(P)$ is constant on~$S$.
\end{propos}

\begin{proof}
By Proposition~\ref{propos_one_val}, $\tCV_K( P)$ is a holomorphic function on~$\Upsilon$.
By~\eqref{eq_volume_explicit2}, we have the following estimate 
\begin{multline*}
\Im\tCV_K( P)=\frac{1}{n-1}\sum_{F\subset K,\,\dim F=n-2}V_{F,\bell}\log|Q_F( P)|\\
\le \max_{F\subset K,\,\dim F=n-2}\frac{NV_{F,\bell}\log|Q_F( P)|}{n-1}
{}\le\max\left(0, 
\max_{F\subset K,\,\dim F=n-2}\log\left|Q_F( P)^L\right|\right),
\end{multline*}
where $N$ is the number of $(n-2)$-dimensional simplices of~$K$, and $L$ is a positive integer that is greater than all numbers $NV_{F,\bell}/(n-1)$.
Recall that $Q_F( P)$ are regular functions on~$\Upsilon$. Take a principal Zariski open subset $\Xi_f\subset\Xi$ such that $\Xi_f\subset\Upsilon$.  (Recall that a Zariski open subset of an irreducible affine variety is called \textit{principal\/} if it is the complement of the set of zeros of a regular function.) Then  $\Xi_f$ is an irreducible affine variety. Besides, $\Xi_f$ is smooth, since $\Upsilon\subset\Xi^{reg}$. Applying Lemma~\ref{lem_log} to the restriction of the function~$\tCV_K( P)$ to~$\Xi_f$, we obtain that ~$\tCV_K( P)$ is  constant on~$\Xi_f$, hence, it is constant on~$\Upsilon$. 
\end{proof}

Since the standard stratification of~$\Sigma^+(\bell)$ consists of finitely many strata, Proposition~\ref{propos_const1} implies that the function~$\CV_K(P)$ is constant on every connected component of~$\Sigma^+(\bell)$. Thus Proposition~\ref{propos_const0} and, hence, Theorem~\ref{theorem_main2} hold true for polyhedra with non-degenerate faces.

\begin{remark}\label{remark_TMC}
Substituting the oriented dihedral angles~$\alpha_F$ to formula~\eqref{eq_TMC}, we obtain the definition of the \textit{total mean curvature\/} of any polyhedron $P\colon K\to\Lambda^n$ with non-degenerate faces. Since the angles $\alpha_F$ are defined modulo~$2\pi\Z$, the total mean curvature is defined modulo the subgroup of~$\R$ generated by the numbers $2\pi V_F(P)$ for all $(n-2)$-dimensional faces of~$K$. However, for any flexible polyhedron $P_t\colon K\to\Lambda^n$, we can choose a real-valued branch of the total mean curvature of~$P_t$. Theorem~\ref{theorem_main2} and Schl\"afli's formula~\eqref{eq_Sch_or} immediately imply the following assertion, which is a more precise formulation of Corollary~\ref{cor_TMC}. 
\end{remark}

\begin{cor}
Suppose that $n$ is odd. Let $K$ be an oriented $(n-1)$-dimensional pseudo-manifold, and let $P_t\colon K\to\Lambda^n$ be a flexion of a bounded polyhedron of combinatorial type~$K$ with non-degenerate faces. Then the total mean curvature of~$P_t$ is constant.
\end{cor}

\section{Proof of Lemma~\ref{lem_log}}\label{section_Liouville}

\begin{lem}\label{lem_multi_log}
Let $A$ be an irreducible complex affine variety, and let $f$ be a non-constant regular function on~$A$ that does not take zero value. Then  the multi-valued function $\Log f$ does not have a single-valued holomorphic branch on~$A$.  
\end{lem}

\begin{proof}
 First, suppose that $A$ is a curve. Then we can delete from~$\C$ a finite set of points~$B=\{b_0=0,b_1,\ldots,b_q\}$ so that the restriction of~$f$ will yield a finite-sheeted (non-ramified) covering
$
f^{-1}(\C\setminus B)\to\C\setminus B
$. Choose a positive number $R$ different from $|b_1|,\ldots,|b_q|$, and put $$S^1_R=\{z\in\C\mid |z|=R\}.$$ Then the restriction of~$f$ to $f^{-1}(S^1_R)$ is a finite-sheeted covering $f^{-1}(S^1_R)\to S^1_R$. Any connected component~$\gamma$ of~$f^{-1}(S^1_R)$ is a closed path in~$A$. As a point $x$ passes along~$\gamma$, its image $f(x)$ passes several times along~$S^1_R$ in the same direction. Hence any branch of~$\Log f$ changes by a non-zero constant.

Second, suppose that $A\subset\C^N$ is an irreducible affine variety of dimension $d>1$. Since $f$ is non-constant on~$A$, there exists a regular point $x\in A$ such that $df$ is non-zero at~$x$. Take any tangent vector $\xi\in T_xA$ such that $\langle df,\xi\rangle\ne 0$. It is not hard to see that there exists an irreducible affine curve $\Gamma\subset A$ such that $x$ is a regular point of~$\Gamma$ and $\xi$ is a tangent vector to~$\Gamma$ at~$x$. Indeed, one can take for $\Gamma$ the irreducible component containing~$x$ of an arbitrary plane section $A\cap\Pi$, where $\Pi\subset\C^N$ is an $(N-d+1)$-dimensional plane such that $T_xA\cap T_x\Pi=\langle\xi\rangle$. Then the restriction of $f$ to~$\Gamma$ is non-constant. Hence $\Log f$ does not have a single-valued branch on~$\Gamma$. Therefore it does not have a single-valued branch on~$A$.
\end{proof}

\begin{proof}[Proof of Lemma~\ref{lem_log}]
Let $d$ be the dimension of~$A$. Then there exists a finite regular mapping $F\colon A\to\C^d$. Let $k$ be the degree of~$F$. Then there is a non-empty Zariski open subset $U\subset \C^d$ such that every point $z\in U$ has exactly $k$ pre-images under~$F$. Moreover, the pre-images of every point~$z\in\C^d\setminus U$ can be assigned multiplicities such that the sum of the multiplicities of all pre-images is equal to~$k$, and the resulting mapping $F^{-1}\colon\C^d\to\mathop{\mathrm{Sym}}^k(A)$ is continuous, where $\mathop{\mathrm{Sym}}^k(A)$ is the $k$th symmetric power of~$A$. We put 
$$
\psi(x)=\exp(-i\varphi(x)).
$$
For each $z\in\C^d$, we put
\begin{equation}\label{eq_Psi_j}
\Psi_j(z)=\sigma_j(\psi(x_1),\ldots,\psi(x_k)),
\end{equation}
where $\sigma_j$ is the $j$th elementary symmetric polynomial, and $x_1,\ldots,x_k$ are the pre-images of~$z$ under~$F$.  The functions~$\Psi_j(z)$ are holomorphic on~$U$ and continuous on~$\C^d$. Hence they have removable singularities on~$\C^d\setminus U$, i.\,e., they are entire functions on~$\C^d$, see~\cite[Sect.~32, Theorem~3]{Sha69}.

Now, let us estimate the functions~$\Psi_j(z)$.  Inequality~\eqref{eq_estimate} yields 
\begin{equation}\label{eq_estimate2}
|\Psi_j(z)|\le\sigma_j\left(\max_{n=1,\ldots,N} |f_n(x_1)|,\ldots,\max_{n=1,\ldots,N} |f_n(x_k)|\right).
\end{equation}
The ring $\C[A]$ is an integral extension of the ring~$\C[\C^d]=\C[z_1,\ldots,z_d]$. Hence every function $f_n$ satisfies a polynomial relation of the form
$$
f_n^{s_n}+g_{n,1}f_n^{s_n-1}+\cdots+g_{n,s_n}=0,
$$
where $g_{n,l}$ are polynomials in~$z=F(x)$. Therefore estimates~\eqref{eq_estimate2} yield  estimates 
$$
|\Psi_j(z)|\le C_j(1+|z|^{K_j})
$$ 
for some $C_j,K_j>0$.  By Liouville's theorem on entire functions, the functions~$\Psi_j(z)$ are polynomials, see~\cite[Sect.~A1.1]{Chi85}. 

By~\eqref{eq_Psi_j}, the function $\psi$ satisfies the polynomial relation
$$
\psi^k(x)-\Psi_1(F(x))\psi^{k-1}(x)+\Psi_2(F(x))\psi^{k-2}(x)-\cdots+(-1)^k\Psi_k(F(x))=0.
$$
The coefficients~$\Psi_j(F(x))$ are regular functions on~$A$. Hence, the function $\psi$ is integral over the ring~$\C[A]$. On the other hand, the function~$\psi$ is holomorphic on~$A$. It follows that $\psi$ is rational function on~$A$, see~\cite[Ch.~8, Sect.~3.1]{Sha72}. Since $A$ is smooth and $\psi$ takes finite values at all points of~$A$, we obtain that $\psi$ is a regular function on~$A$. The function~$-i\varphi$ is a single-valued holomorphic branch of the function $\Log\psi$ on~$A$. By Lemma~\ref{lem_multi_log}, this would be impossible if the function~$\psi$ were non-constant. Thus $\psi$ and~$\varphi$ are constants.
\end{proof}

\section{Proof of Theorem~\ref{theorem_main2} for arbitrary polyhedra}\label{section_proof_main_arbitrary}

Theorem~\ref{theorem_main2} for arbitrary flexible polyhedra follows immediately from Theorem~\ref{theorem_main2} for flexible polyhedra with non-degenerate faces and the following lemma.

\begin{lem}\label{lem_degenerate}
Let $P_t\colon K\to\Lambda^n$ be a bounded flexible polyhedron. Then there exist oriented $(n-1)$-dimensional pseudo-manifolds $K^{(1)},\ldots,K^{(k)}$ and bounded flexible polyhedra with non-degenerate faces $P_t^{(l)}\colon K^{(l)}\to\Lambda^n,$ $l=1,\ldots,k,$  such that the indicator function $\lambda_{P_t}(x)$ is equal to the sum $\sum_{l=1}^k\lambda_{P_t^{(l)}}(x)$ for all~$t$ and all $x\in \Lambda^n\setminus P_t(K)$. Hence, $$\CV_K(P_t)=\sum_{l=1}^k\CV_{K^{(l)}}(P_t^{(l)}).$$
\end{lem}

\begin{proof}
Let $\bell$ be the set of edge lengths of the flexible polyhedron $P_t\colon K\to\Lambda^n$. For each simplex $\sigma$ of~$K$, the image~$P_t(\sigma)$ is  determined up to isometry of~$\Lambda^n$ by the set of edge lengths of~$\sigma$. We shall conveniently choose a polytope $ \Delta_{\sigma}\subset\Lambda^n$ isometric to~$P_t(\sigma)$ for all~$t$, the projection~$\varpi_{\sigma}\colon\sigma\to  \Delta_{\sigma}$, and the isometries  $\varphi_{\sigma,t}\colon  \Delta_{\sigma}\to P_t(\sigma)$ such that $\varphi_{\sigma,t}\circ\varpi_{\sigma}=P_t|_{\sigma}$ for all~$t$. 
We shall say that a simplex $\sigma$ of~$K$ is \textit{$\bell$-non-degenetrate\/} if $\dim  \Delta_{\sigma}=\dim\sigma$, and \textit{$\bell$-degenetrate\/} if $\dim  \Delta_{\sigma}<\dim\sigma$. If $\sigma$ is $\bell$-non-degenetrate, then $ \Delta_{\sigma}$ is a simplex.
 
A sequence $\theta=(\sigma_0,\ldots,\sigma_m)$  of $(n-1)$-dimensional simplices of~$K$ is called a \textit{thick path\/} if $\dim(\sigma_{j-1}\cap\sigma_j)=n-2$ for $j=1,\ldots,m$, and $\sigma_{j-1}\ne\sigma_{j+1}$ for $j=1,\ldots,m-1$. The number~$m$ will be called the \textit{length\/} of~$\theta$.
 A thick path~$\theta$ will be called \textit{admissible\/}  if $m\ge 1$ and the following conditions are satisfied:
\begin{enumerate}
\item $\sigma_0$  is $\bell$-non-degenerate and $\sigma_1,\ldots,\sigma_{m-1}$ are $\bell$-degenerate.
\item The simplices $\tau_j=\sigma_{j-1}\cap\sigma_j$ are $\bell$-non-degenerate, $j=1,\ldots,m$.
\item $\dim\left(P_t(\tau_1)\cap\cdots\cap P_t(\tau_m)\right)=n-2$.
\end{enumerate}

Let us prove that the latter  condition is independent of~$t$. Indeed, let $\iota_{t_2,t_1}$ be an isometry of~$\Lambda^n$ taking $P_{t_1}(\sigma_0)$ to~$P_{t_2}(\sigma_0)$, and let $\Pi_{t_1}$ and~$\Pi_{t_2}$ be the $(n-2)$-dimensional planes in~$\Lambda^n$ containing~$P_{t_1}(\tau_1)$ and~$P_{t_2}(\tau_1)$, respectively. It is easy to see that conditions~(1) and~(2) imply that $P_{t_1}$  and~$P_{t_2}$ map all simplices $\sigma_1,\ldots,\sigma_{m-1}$ to the planes $\Pi_{t_1}$ and~$\Pi_{t_2}$, respectively, and $$P_{t_2}|_{\sigma_1\cup\cdots\cup\sigma_{m-1}}=\iota_{t_2,t_1}\circ P_{t_1}|_{\sigma_1\cup\cdots\cup\sigma_{m-1}}.$$ Therefore, the polytopes $P_{t_1}(\tau_1)\cap\cdots\cap P_{t_1}(\tau_m)$ and $P_{t_2}(\tau_1)\cap\cdots\cap P_{t_2}(\tau_m)$ are isometric. Moreover, we see that the $(n-2)$-dimensional convex polytope 
\begin{equation*}
F_{\theta}=\varphi^{-1}_{\sigma_0,t}\left(P_t(\tau_1)\cap\cdots\cap P_t(\tau_m)\right)\subset\partial  \Delta_{\sigma_0}
\end{equation*}
is independent of~$t$.

An admissible thick path $\theta=(\sigma_0,\ldots,\sigma_m)$ will be called \textit{maximal\/} if $\sigma_m$ is $\bell$-non-dege\-n\-erate. Obviously, if $\theta$ is a maximal admissible thick path, then the inverse sequence $\theta^{-1}=(\sigma_m,\ldots,\sigma_0)$ is also a maximal admissible thick path, and the isometry 
$$
\gamma_{\theta}\colon F_{\theta}\xrightarrow{\varphi_{\sigma_0,t}}P_t(\tau_1)\cap\cdots\cap P_t(\tau_m)\xrightarrow{\varphi_{\sigma_m,t}^{-1}}F_{\theta^{-1}}
$$
is independent of~$t$.

\begin{lem}
The number of admissible thick paths is finite.
Let $\sigma$ be an $\bell$-non-degene\-rate $(n-1)$-dimensional simplex of~$K$.  Then the  polytopes $F_{\theta},$ where $\theta$ runs over all maximal admissible thick paths starting with~$\sigma,$ form a decomposition of~$\partial  \Delta_{\sigma}$ into finitely many convex polytopes with disjoint interiors.
\end{lem}
\begin{proof}
To prove that the number of admissible thick paths is finite it is sufficient to show that any admissible thick path $\theta=(\sigma_0,\ldots,\sigma_m)$ is non-self-intersecting, i.\,e., the simplices $\sigma_0,\ldots,\sigma_m$ are pairwise disjoint. Assume the converse. Take the smallest~$j$ such that $\sigma_j=\sigma_i$ for some $i<j$. 
If $i\ne 0$, then the simplices $\tau_{i}$, $\tau_{i+1}$, and $\tau_{j}$ are three pairwise different $(n-2)$-dimensional faces of~$\sigma_{i}$. It follows easily that $\dim(P_t(\tau_{i})\cap P_t(\tau_{i+1})\cap P_t(\tau_{j}))<n-2$, which contradicts property~(3) in the definition of an admissible thick  path. If $i=0$, then $j=m$. Because of the minimality of~$j$, we see that $\sigma_1\ne\sigma_{m-1}$. Hence~$\tau_1$ and~$\tau_{m}$  are different $(n-2)$-dimensional faces of the $\bell$-non-degenerate simplex~$\sigma_0$. Therefore $\dim(P_t(\tau_1)\cap P_t(\tau_m))<n-2$ and we again obtain a contradiction with property~(3).

Let us prove that the interiors of polytopes~$F_{\theta}$ and~$F_{\theta'}$ corresponding to different maximal admissible thick paths $\theta=(\sigma,\sigma_1,\ldots,\sigma_m)$ and  $\theta'=(\sigma,\sigma'_1,\ldots,\sigma'_{m'})$ are disjoint. If $\sigma_1\ne\sigma_1'$, then~$F_{\theta}$ and~$F_{\theta'}$ lie in different $(n-2)$-dimensional faces of~$ \Delta_{\sigma}$, hence their interiors are disjoint. Assume that $\sigma_1=\sigma_1'$. Take the smallest~$j$ such that $\sigma_j\ne \sigma_j'$. Then $\tau_{j-1}$, $\tau_j$, and $\tau_j'$ are three pairwise different $(n-2)$-dimensional faces of~$\sigma_{j-1}$. Hence $\dim(P_t(\tau_{j-1})\cap P_t(\tau_{j})\cap P_t(\tau_{j}'))<n-2$. Therefore $\dim (F_{\theta}\cap F_{\theta'})<n-2$. Thus the interiors of~$F_{\theta}$ and~$F_{\theta'}$ are disjoint.

Finally, let us prove that the union of the polytopes~$F_{\theta}$ for all maximal admissible thick paths~$\theta$ starting with~$\sigma$ coincides with~$\partial  \Delta_{\sigma}$. We denote by  $A_m(\sigma)$ the set of all admissible thick paths of length~$m$ starting with~$\sigma$, and we denote by $MA(\sigma)$ the set of all maximal admissible thick paths starting with~$\sigma$.  For $m>1$, let us prove  that 
\begin{equation}\label{eq_AF_theta}
\bigcup_{\theta\in A_{m-1}}F_{\theta}\subset \bigcup_{\theta\in A_m(\sigma)\cup MA(\sigma)}F_{\theta}.
\end{equation}
Let $\theta=(\sigma,\sigma_1,\ldots,\sigma_{m-1})$ be an admissible thick path that is not maximal.  Let $\tau_m^{(1)},\ldots,\tau_m^{(p)}$ be all $\bell$-non-degenerate $(n-2)$-dimensional faces of~$\sigma_{m-1}$ different from~$\tau_{m-1}$. For $j=1,\ldots,p$, let $\sigma_m^{(j)}$ be the $(n-1)$-dimensional simplex of~$K$ such that $\sigma_{m-1}\cap\sigma_m^{(j)}=\tau_m^{(j)}$, and let
$\theta^{(j)}=(\sigma,\sigma_1,\ldots,\sigma_{m-1},\sigma_m^{(j)})$. Since $\sigma_{m-1}$ is $\bell$-degenerate, we have $P_t(\tau_{m-1})\subset \bigcup_{j=1}^pP_t\bigl(\tau_m^{(j)}\bigr)$. Hence,
\begin{equation*}
P_t(\tau_1)\cap\cdots\cap P_t(\tau_{m-1})\subset \bigcup_{j=1}^p\left(P_t(\tau_1)\cap\cdots\cap P_t(\tau_{m-1})\cap P_t\bigl(\tau_m^{(j)}\bigr)\right).
\end{equation*}
Moreover, since $P_t(\tau_1)\cap\cdots\cap P_t(\tau_{m-1})$ is an $(n-2)$-dimensional convex polytope, it is contained in the union of only those polytopes $P_t(\tau_1)\cap\cdots\cap P_t(\tau_{m-1})\cap P_t\bigl(\tau_m^{(j)}\bigr)$ which are $(n-2)$-dimensional, that is, correspond to admissible thick paths~$\theta^{(j)}$. Therefore the polytope $F_{\theta}$ is contained in the union of the polytopes~$F_{\theta^{(j)}}$ corresponding to admissible thick paths~$\theta^{(j)}$, which immediately yields~\eqref{eq_AF_theta}. Since the sets $A_m(\sigma)$ are empty for sufficiently large~$m$, we obtain that $\bigcup_{\theta\in MA(\sigma)}F_{\theta}\supset\bigcup_{\theta\in A_1(\sigma)}F_{\theta}=\partial  \Delta_{\sigma}$. The latter equality holds, since all thick paths $(\sigma,\sigma_1)$ of length~$1$ are admissible.
\end{proof}

Let us proceed with the proof of Lemma~\ref{lem_degenerate}. Take the disjoint union of the hyperbolic simplices~$ \Delta_{\sigma}$, where $\sigma$ runs over all $\bell$-non-degenerate $(n-1)$-dimensional simplices of~$K$. For each maximal admissible thick path $\theta=(\sigma_0,\ldots,\sigma_m)$, glue the $(n-2)$-dimensional polytopes $F_{\theta}\subset \partial  \Delta_{\sigma_0}$ and $F_{\theta^{-1}}\subset \partial  \Delta_{\sigma_m}$ along the isometry~$\gamma_{\theta}$. Let $\mathcal{K}$ be the obtained cell complex. For each~$t$, the mappings $\varphi_{\sigma,t}\colon  \Delta_{\sigma}\to\Lambda^n$ constitute a well-defined mapping $\mathcal{P}_t\colon\mathcal{K}\to \Lambda^n$. In particular, it follows that the tautological mappings $ \Delta_{\sigma}\to\mathcal{K}$ are embeddings, i.\,e., that we never glue any two different points of the same cell~$ \Delta_{\sigma}$. By the construction, every cell of~$\mathcal{K}$ is contained in an $(n-1)$-dimensional cell and every $(n-2)$-dimensional cell of~$\mathcal{K}$ is contained in exactly two $(n-1)$-dimensional cells. We subdivide the cell complex~$\mathcal{K}$ into convex hyperbolic simplices. Let $K^{(1)},\ldots,K^{(k)}$ be the connected components of the obtained simplicial complex. Then  $K^{(1)},\ldots,K^{(k)}$ are $(n-1)$-dimensional pseudo-manifolds. The orientation of every $\bell$-non-degenerate $(n-1)$-dimensional simplex~$\sigma$ of~$K$ induces the orientation of~$ \Delta_{\sigma}$. Hence the pseudo-manifolds $K^{(1)},\ldots,K^{(k)}$ obtain canonical orientations. Let $P^{(j)}_t$ be the restriction of~$\mathcal{P}_t$ to~$K^{(j)}$. Since the restriction of~$\mathcal{P}_t$ to every cell~$ \Delta_{\sigma}$ is an isometry, we see that $P^{(j)}_t\colon K^{(j)}\to\Lambda^n$ is a flexible polyhedron with non-degenerate faces.  It follows immediately from the construction that $P^{(j)}_t(K^{(j)})\subset P_t(K)$ for all~$j$ and~$t$, and $\lambda_{P_t}(x)=\sum_{l=1}^k\lambda_{P_t^{(l)}}(x)$ for all~$t$ and all $x\in \Lambda^n\setminus P_t(K)$.
\end{proof}


\begin{thebibliography}{99}

\bibitem{AVS88} D.\,V.\,Alekseevskij, E.\,B.\,Vinberg, A.\,S.\,Solodovnikov, \textit{Geometry of spaces of constant curvature,} In the book: \textit{Geometry II} (ed. E.\,B.\,Vinberg), Encycl. of Math. Sci. \textbf{29}, Springer, 1993, 1--146. 

\bibitem{Ale85} R.\,Alexander, \textit{Lipschitzian mappings and total mean curvature of polyhedral surfaces, I,\/} Trans. Amer. Math. Soc. \textbf{288}:2 (1985), 661--678.

\bibitem{Ale97} V.\,Alexandrov, \textit{An example of a flexible polyhedron with nonconstant volume in the spherical space,} Beitr. Algebra Geom. \textbf{38}:1 (1997), 11--18.

\bibitem{AlRi98} F.\,Almgren, I.\,Rivin, \textit{The mean curvature integral is invariant under bending,} in: The Epstein Birthday Schrift, University of Warwick, 1998, 1--21.

\bibitem{Aom77} K.\,Aomoto, \textit{Analytic structure of Schl\"afli function,} Nagoya Math. J. \textbf{68} (1977), 1--16.

\bibitem{Bri97} R.\,Bricard, \textit{M\'emoire sur la th\'eorie de l'octa\`edre articul\'e,} J. Math. Pures Appl. \textbf{5}:3 (1897), 113--148.

\bibitem{Che67} K.-T.\,Chen, \textit{Iterated path integrals and generalized paths,} Bull. Amer. Math. Soc. \textbf{73} (1967), 935--938.

\bibitem{Che73} K.-T.\,Chen, \textit{Iterated integrals of differential forms and loop space homology,} Ann. Math. (2) \textbf{97}:2 (1973), 217--246.

\bibitem{Chi85} E.\,M.\,Chirka, \textit{Complex Analytic Sets,} Mathematics and Its Applications (Soviet series),  Kluwer Acad. Publ., 1989.

\bibitem{ChKi99} Y.\,Cho, H.\,Kim, \textit{On the Volume Formula for Hyperbolic Tetrahedra,} Discrete Comput. Geom. \textbf{22}:3 (1999), 347--366.

\bibitem{Con77} R.\,Connelly, \textit{A counterexample to the rigidity conjecture for polyhedra,} Inst. Hautes \'Etudes Sci. Publ. Math. \textbf{47} (1977), 333--338.

\bibitem{Con78a} R.\,Connelly
\textit{A flexible sphere,}
Mathematical Intelligencer
\textbf{1}:3 (1978), 130--131.

\bibitem{Con78} R.\,Connelly, \textit{Conjectures and open questions in rigidity,} Proc. Internat. Congress Math. (Helsinki, 1978), Acad. Sci. Fennica, Helsinki, 1980, 407--414.

\bibitem{CSW97} R.\,Connelly, I.\,Sabitov, A.\,Walz, \textit{The Bellows Conjecture,} Beitr. Algebra Geom. \textbf{38}:1 (1997), 1--10.

\bibitem{Cox35} H.\,S.\,M.\,Coxeter, \textit{The functions of Schl\"afli and Lobatschefsky,} Quart. J. Math. \textbf{6} (1935), 13--29.

\bibitem{DeMe05} D.\,A.\,Derevnin, A.\,D.\,Mednykh, \textit{A formula for the volume of a hyperbolic tetrahedron,} Uspekhi Matem. Nauk \textbf{60}:2(362) (2005),  159--160 (in Russian), Russ. Math. Surv., \textbf{60}:2 (2005), 346--348 (English translation).

\bibitem{Gai11} A.\,A.\,Gaifullin, \textit{Sabitov polynomials for volumes of polyhedra in four dimensions,} Adv. Math. \textbf{252} (2014), 586--611, arXiv: 1108.6014.

\bibitem{Gai12} A.\,A.\,Gaifullin, \textit{Generalization of Sabitov's theorem to polyhedra of arbitrary dimensions,} Discrete Comput. Geom. \textbf{52}:2 (2014), 195--220, arXiv: 1210.5408.

\bibitem{Gai13} A.\,A.\,Gaifullin, \textit{Flexible cross-polytopes in spaces of constant curvature,} Trudy MIAN \textbf{286} (2014), 88--128 (in Russian), Proc. Steklov Inst. Math. \textbf{286} (2014), 77--113 (English translation), arXiv: 1312.7608.

\bibitem{Gai15} A.\,A.\,Gaifullin, \textit{Embedded flexible spherical cross-polytopes with non-constant volumes,} Trudy MIAN \textbf{288} (2015), 67--94 (in Russian), Proc. Steklov Inst. Math. \textbf{288} (2015), 56--80 (English translation), arXiv: 1501.06198.


\bibitem{Kui78} N.\,H.\,Kuiper
\textit{Sph\`eres Poly\`edriques Flexibles dans $E^3$, d'apr\`es Robert Connelly,}
 S\`eminaire Bourbaki, 1977/78, Lecture Notes in Mathematics
\textbf{710}, Springer-Verlag, 1979, 147--168.

\bibitem{Luo06} F.\,Luo, \textit{Continuity of the volume of simplices in classical geometries,} Commun. Contemp. Math. \textbf{8}:3 (2006), 411--431, arXiv: math/0412208.

\bibitem{MuUs05} J.\,Murakami, A.\,Ushijima, \textit{A volume formula for hyperbolic tetrahedra in terms of edge lengths,} J. Geom. \textbf{83}:1--2 (2005), 153--163, arXiv: math/0402087.

\bibitem{MuYa05} J.\,Murakami, M.\,Yano, \textit{On the volume of a hyperbolic and spherical tetrahedron,} Comm. Anal. Geom. \textbf{13} (2005), 379--200.

\bibitem{Par66} A.\,N.\,Parshin, \textit{On a certain generalization of Jacobian manifold,} Izv. Akad. Nauk SSSR Ser. Mat. \textbf{30} (1966), 175--182 (in Russian).

\bibitem{Riv08} I.\,Rivin, \textit{Volumes of degenerating polyhedra --- on a conjecture of J.\,W.\,Milnor,} 
Geom. Dedicata \textbf{131}:1 (2008), 73--85, arXiv: math/0512065.

\bibitem{Sab96} I.\,Kh.\,Sabitov, \textit{Volume of a polyhedron as a function of its metric,} Fundamental and Applied Math. \textbf{2}:4 (1996), 1235--1246 (in Russian). 

\bibitem{Sab98a} I.\,Kh.\,Sabitov, \textit{A generalized Heron--Tartaglia formula and some of its consequences,} Mat. Sb. \textbf{189}:10 (1998), 105--134 (in Russian); Sb. Math. \textbf{189}:10 (1998), 1533--1561 (English translation).

\bibitem{Sab98b} I.\,Kh.\,Sabitov, \textit{The volume  as a  metric invariant of polyhedra,} Discrete Comput. Geom. \textbf{20}:4 (1998), 405--425. 

\bibitem{Sab11} I.\,Kh.\,Sabitov, \textit{Algebraic methods for solution of polyhedra,} Uspekhi Matem. Nauk \textbf{66}:3(399) (2011), 3--66 (in Russian); Russ. Math. Surv. \textbf{66}:3 (2011), 445--505 (English translation).

\bibitem{Sch58} L.\,Schl\"afli, \textit{On the multiple integral $\int^n dxdy\ldots dz,$ whose limits are $p_1=a_1x+b_1y+\ldots+h_1z>0,$ $p_2>0,\ldots p_n>0,$ and $x^2+y^2+\ldots+z^2<1,$} Quart. J. Pure Appl. Math. \textbf{2} (1858), 269--301; \textbf{3} (1860), 54--68, 97--108.

\bibitem{Sfo07} G.\,Sforza, \textit{Spazi metrico-proiettivi,} Ricerche di Estensionimetria Integrale, Ser. III, \textbf{VIII} (Appendice) (1907), 41--66.


\bibitem{Sha69} B.\,V.\,Shabat, \textit{Introduction to Complex Analysis. Part II: Functions of Several Variables,} Transl. Math. Monographs \textbf{110}, Amer. Math. Soc., 1992.


\bibitem{Sha72} I.\,R.\,Shafarevich, \textit{Basic Algebraic Geometry 2: Schemes and Complex Manifolds,} 3rd edn., Springer-Verlag, 2013.

\bibitem{Sta00} H.\,Stachel, \textit{Flexible cross-polytopes in the Euclidean $4$-space,} J. Geom. Graph. \textbf{4}:2 (2000), 159--167.

\bibitem{Sta06} H.\,Stachel, \textit{Flexible octahedra in the hyperbolic space,} In the book: \textit{Non-Euclidean geometries. J\'anos Bolyai memorial volume\/} (Eds. A. Pr\'ekopa et al.). New York: Springer. Mathematics and its Applications (Springer) \textbf{581}, 209--225 (2006).

\bibitem{Whi57} H.\,Whitney, \textit{Elementary structure of real algebraic varieties,} Ann. Math. (2) \textbf{66}:3 (1957), 545--556.


\end{thebibliography}
\end{document}